\documentclass[11pt,a4paper,abstract=on]{scrartcl}

\usepackage[english]{babel}

\usepackage{lmodern}
\usepackage[T1]{fontenc}

\usepackage{ucs}
\usepackage[utf8x]{inputenc}

\usepackage[comma,numbers,sort&compress]{natbib}

\usepackage{amsfonts,amstext,amsmath,amssymb,amsopn,amsthm}

\usepackage{mathtools}

\usepackage{dsfont}
\usepackage{mathrsfs}

\usepackage{listings}
\usepackage{physics} 
\usepackage[dvipsnames, svgnames, x11names]{xcolor}

\usepackage{siunitx}
\usepackage{graphicx}
\usepackage{subfig}
\usepackage{wrapfig}

 \usepackage{todonotes}

\usepackage{csquotes}
\usepackage[colorlinks=true, 
linkcolor=blue,
pdfstartview=FitH,      
breaklinks=true,        
bookmarksopen=true,     
bookmarksnumbered=true  
]{hyperref}
\usepackage[capitalize,nameinlink]{cleveref}
\usepackage{url}
\usepackage{todonotes}
\usepackage{enumerate}
\usepackage{caption}

\usepackage[top=1in, bottom=1.5in, left=1in, right=1in]{geometry}
\usepackage[linesnumbered,lined,ruled, noend]{algorithm2e}

\newcommand{\paren}[1]{\left(#1\right)}
\newcommand{\ecklam}[1]{\left[#1\right]}

\renewcommand{\equiv}{\ensuremath{:=}}

\newcommand{\act}[1]{\left\langle {#1} \right\rangle}

\newcommand{\mymap}[3]{#1:\,#2 \to #3\,}
\newcommand{\mmap}[3]{#1:\,#2 \rightrightarrows #3\,}

\newcommand{\mysetc}[2]{\left\{#1\,\middle|\,#2\right\}}
\newcommand{\set}[2]{\left\{#1\,\middle|\,#2\right\}}

\renewcommand{\ip}[2]{\left\langle #1,\, #2\right\rangle}

\newcommand{\cex}[2]{\ensuremath{\mathbb{E}\left[#1\,\middle|\,#2\right]}}
\newcommand{\cpr}[2]{\ensuremath{\mathbb{P}\left(#1\,\middle|\,#2\right)}}
\newcommand{\1}{\ensuremath{\mathds{1}} }

\newcommand{\icol}[1]{
  \left(\begin{smallmatrix}#1\end{smallmatrix}\right)%
}

\newcommand{\cb}{\ensuremath{\overline{\mathbb{B}}}}

\newcommand{\Nbb}{\mathbb{N}}

\newcommand{\Rn}{\mathbb{R}^n}
\newcommand{\Pcal}{\mathcal{P}}
\newcommand{\Jcal}{\mathcal{J}}
\newcommand{\Fcal}{\mathcal{F}}

\newcommand{\xbar}{{\overline{x}}}
\newcommand{\hhbar}{\overline{h}}

\newcommand{\Fbar}{\overline{F}}
\newcommand{\ibar}{\overline{i}}

\DeclareMathOperator{\Supp}{supp}
\DeclareMathOperator{\dist}{dist}

\DeclareMathOperator*{\argmin}{\arg\!\min}

\DeclareMathOperator{\prox}{prox}

\DeclareMathOperator{\id}{Id}
\DeclareMathOperator{\Id}{Id}

\DeclareMathOperator{\Fix}{Fix}

\DeclareMathOperator{\gph}{gph}

\DeclareMathOperator{\diam}{diam}

\DeclareMathOperator{\inv}{inv}
\DeclareMathOperator{\indep}{\perp \!\!\! \perp\,}

\DeclareMathOperator{\Card}{Card}

\makeatletter
\def\@endtheorem{\endtrivlist\@endpefalse }
\makeatother

\newtheorem{thm}{Theorem}[section]
\newtheorem{cor}[thm]{Corollary}
\newtheorem{lemma}[thm]{Lemma}
\newtheorem{prop}[thm]{Proposition}
\theoremstyle{definition}
\newtheorem{example}[thm]{Example}
\newtheorem{definition}[thm]{Definition}
\newtheorem{assumption}[thm]{Assumption}

\newtheoremstyle{note}
{3pt}
{3pt}
{}
{}
{\bfseries}
{\bfseries :}
{.5em}
{}
\theoremstyle{note}

\newtheorem{rem}[thm]{Remark}

\title{Random Function Iterations for Stochastic Fixed Point Problems}
\author{Neal Hermer\thanks{Institute for Numerical and Applied Mathematics,
    University of Goettingen,
    37083 Goettingen, Germany. NH was supported by 
    Deutsche Forschungsgemeinschaft Research Training Grant 2088 TP-B5.
    E-mail:  \texttt{n.hermer@math.uni-goettingen.de}}, 
  D. Russell Luke\thanks{Institute for Numerical and Applied Mathematics,
    University of Goettingen,
    37083 Goettingen, Germany. DRL was supported in part by 
    Deutsche Forschungsgemeinschaft Research Training Grant 2088 TP-B5.
    E-mail:  \texttt{r.luke@math.uni-goettingen.de}}  
  and  Anja Sturm\thanks{Institute for Mathematical Stochastic,
    University of Goettingen,
    37077 Goettingen, Germany. AS was supported in part by Deutsche 
Forschungsgemeinschaft 
    Research Training Grant 2088 TP-B5.
    E-mail:  \texttt{asturm@math.uni-goettingen.de}}}
\date{\today}

\begin{document}
 \maketitle

 \begin{abstract}
   We study the convergence of random function iterations for finding
   an invariant measure of the corresponding Markov operator.  We call
   the problem of finding such an invariant measure the {\em
     stochastic fixed point problem}.  This generalizes earlier work
   studying the {\em stochastic feasibility problem}, namely, to find
   points that are, with probability 1, fixed points of the random
   functions \cite{HerLukStu19a}.  When no such points exist, the
   stochastic feasibility problem is called {\em inconsistent}, but
   still under certain assumptions, the more general stochastic fixed
   point problem has a solution and the random function iterations
   converge to an invariant measure for the corresponding Markov
   operator.  There are two major types of convergence: almost sure
   convergence of the iterates to a fixed point in the case of
   stochastic feasibility, and convergence in distribution more
   generally.  We show how common structures in deterministic fixed
   point theory can be exploited to establish existence of invariant
   measures and convergence of the Markov chain.  We show that weaker
   assumptions than are usually encountered in the analysis of Markov
   chains guarantee linear/geometric convergence.  This framework
   specializes to many applications of current interest including, for
   instance, stochastic algorithms for large-scale distributed
   computation, and deterministic iterative procedures with
   computational error.  The theory developed in 
this study provides a solid basis for describing the convergence
of simple computational methods without the assumption of infinite precision
arithmetic or vanishing computational errors.

\end{abstract}

{\small \noindent {\bfseries 2010 Mathematics Subject Classification:}
  Primary 60J05, 
  46N10, 
  46N30, 
  65C40, 
  49J55 
    Secondary  49J53,   
    65K05.\\ 
  }


\section{Introduction}
\label{sec:introduction}
Random function iterations (RFI) \cite{Diaconis1999} generalize
deterministic fixed point iterations, and are a useful framework for
studying a number of important applications.  An RFI is a stochastic
process of the form $ X_{k+1} := T_{\xi_{k}} X_{k} $ ($k=0,1,2,\dots$)
initialized by a random variable $X_0$ with distribution $\mu_0$ and
values on some set $G$.  This of course includes initialization from
a deterministic point $x_0\in G$ via the $\delta$-distribution.  
Here $\xi_{k}$ ($k=0,1,2,\dots$) is an
element of a sequence of i.i.d.\ random variables that map from a
probability space into a measurable space of indices $I$ (not
necessarily countable) and $T_{i}$ $(i \in I)$ are self-mappings on
$G$.  The iterates $X_{k}$ form a Markov chain of random variables on
the space $G$, which is, for our purposes, a \emph{Polish space}.  
Deterministic fixed point iterations are included when the index set
$I$ is just a singleton.

Our main motivation for the fundamental and abstract study pursued here is 
the very concrete application of X-ray free electron laser (X-FEL) imaging experiments 
\cite{Chap11,Bou12,ArdGru}. 
We return to this specific application in Section \ref{sec:incFeas}.  There are many more 
applications than one could reasonably list, but to reach the broadest possible audience, 
a simple example from first
semester numerical analysis is illustrative.  Consider the
underdetermined linear system of equations
\[
 Ax=b, \quad A\in\mathbb{R}^{m\times n}, b\in \mathbb{R}^m, ~m<n.  
\]
Equivalent to this problem is the problem of finding the intersection of 
the hyperplanes defined by the single equations $\ip{a_j}{x}=b_j$ ($j=1,2,\dots. m$)
where $a_j$ is the $j$th row of the matrix $A$:
\begin{equation}\label{eq:ifeas}
 \mbox{Find }\quad \bar x\in \cap_{j=1}^m \{x~|~\ip{a_j}{x}=b_j\}.
\end{equation}
An intuitive technique to solve this problem is 
the method of cyclic projections:  Given an initial guess $x_0$, 
construct the sequence $(x_k)$ via
\begin{equation}\label{eq:CP}
 x_{k+1} = P_mP_{m-1}\cdots P_1 x_k,
\end{equation}
where $P_j$ is the orthogonal projection onto the $j$th hyperplane above. 
This method was proposed by von Neumann 
in \cite{Neumann50} where he also proved that, without numerical error, the iterates
converge to the projection of the initial point $x_0$ onto the 
itersection;  Aronszajn 
showed that the rate of convergence is linear \cite{Aronszajn} 
(referred to as geometric or exponential
in other communities).

The projectors have a closed form representation, and the algorithm is 
easily implemented.  The results of one implementation for a randomly 
generated matrix $A$ and vector $b$ with $m=50$
and $n=60$ yields the following graph shown in Figure \ref{fig:Axb}(a). 
\begin{figure}[h]
  (a) \includegraphics[width=6cm,
  height=4cm]{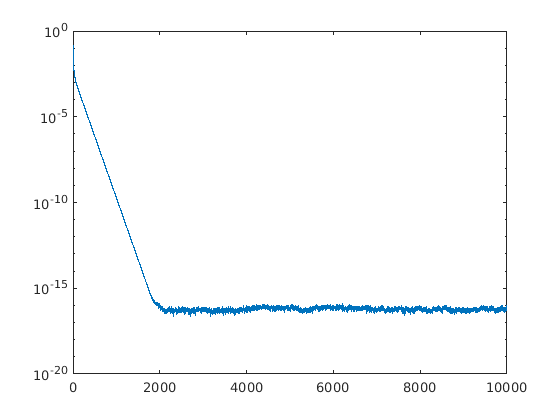} \hfill (b)
  \includegraphics[width=6cm,
  height=4cm]{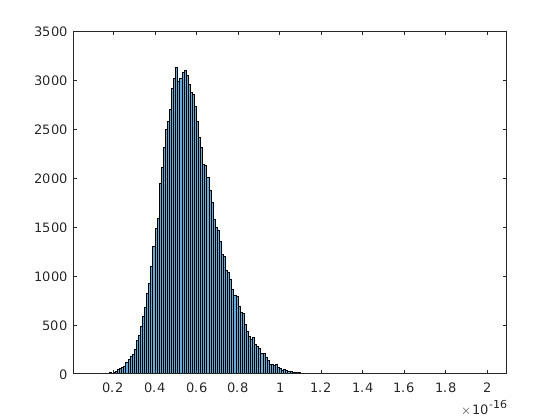}
\caption{(a) The residual $x_{k+1}-x_{k}$ of iterates of the cyclic projections
algorithm for solving the linear system $Ax=b$ for $A\in \mathbb{R}^{50\times 60}$,
and $b\in \mathbb{R}^{50}$ randomly generated.  (b) A histogram of the residual 
sizes over the last $8000$ iterations. }\label{fig:Axb}
\end{figure}
As the figure shows, the method performs as predicted by the theory, up 
to the numerical precision of the implementation.  After that point, the 
iterates behave more like random variables with distribution indicated by the 
histogram shown in Figure \ref{fig:Axb}(b).  The theory developed in 
this study provides a solid basis for describing the convergence
of simple computational methods without the assumption of infinite precision
arithmetic or vanishing computational errors \cite{Rockafellar76, SolodovSvaiter99}.    
This particular 
situation could be analyzed in the stability framework of perturbed convergent 
fixed point iterations with unique fixed points developed in \cite{ButReiZas07};
our approach captures their results and opens the way to a much broader range of 
applications.  
An analysis of nonmonotone fixed point iterations with error can be found already in 
\cite{IusPenSva03}, though the precision is assumed to increase quickly to exact evaluation. 

One of the main goals of the present study is to extend the approach established in \cite{LukNguTam18} 
to the tomographic problem associated with X-FEL measurements in particular, and to noncontractive random 
function iterations more generally, accounting for randomness not only in the model and the algorithm, but also in  
the computers we use for implementations.  The main object of interest is the Markov operator on 
a space of probability measures with the appropriate metric.  We take for granted much of the basic theory of 
Markov chains, which interested readers can find, for instance, in \cite{HerLerLas} or \cite{MeyTwe}. 
We are indebted to the work of Butnariu and collaborators 
who studied stochastic iterative procedures for solving infinite dimensional linear operator 
equations
in \cite{ButnariuCensorReich97, ButnariuIusemBurachik00, ButnariuFlam95,
ButnariuCensorReich97}. 
Another important application motivating our analytical strategy involves stochastic 
implementations of deterministic algorithms for large-scale optimization problems
\cite{BoydParikhChuPeleatoEckstein11, Combettes2018, Villa2019,
  Richtarik15, Nedic2011}.  Such stochastic algorithms are popular for
distributed computation with applications in machine learning
\cite{BaldassiE7655, combettes2017a,DieuDurBac20, HardtMaRecht18, Richtarik16}.  
Here each $T_{\xi_{k}}$ represents a randomly selected, low-dimensional update
mechanism in an iterative procedure.  Our approach to the analysis of
such algorithms allows for the first time {\em expansive} mappings and, in some cases, the 
analysis is simpler than
current approaches.

We are concerned in this paper with (i) {\em existence of invariant
  distributions} of the Markov operators associated with the random
function iterations, (ii) {\em convergence} of the Markov chain to an
invariant distribution, and (iii) {\em rates of convergence}.  As with
classical fixed point iterations, the limit -- or more accurately,
{\em limiting  distribution} -- of the Markov chain, if it exists, will
in general depend on the initialization.  Uniqueness of invariant
measures of the Markov operator is not a particular concern for feasibility
problems where {\em any} feasible point will do.  
The notation and necessary 
background is developed in Section \ref{sec:consistent_feas_prob}, which 
we conclude with the main statements of this study (Section \ref{sec:mainres}).
Section \ref{sec:theory} contains the technical details, starting with 
existence theory in Section \ref{sec:existence}, general ergodic theory in  
Section \ref{sec:suppCvg} with gradually increasing regularity assumptions on the 
Markov operators, 
equicontinuity in Section \ref{sec:ergodicNonexp} and finally Markov operators
generated by nonexpansive mappings in Section \ref{sec:nexp}.  
The assumptions on the mappings generating the Markov operators are commonly 
employed in the analysis of deterministic algorithms in continuous optimization.   
Our first main 
result, Theorem \ref{cor:cesaroConvergenceRn},  establishes convergence for 
Markov chains that are generated from 
nonexpansive mappings in $\Rn$ and follows easily in Section \ref{sec:setCvgNE} 
upon establishing tightness of the sequence of measures. 
Section \ref{sec:furtherprops} collects further facts needed for the 
second main result of this study, Theorem \ref{thm:a-firm convergence Rn}, 
which establishes convergence in the Prokhorov-L\`evy metric 
of Markov chains to an invariant measure (assuming this exists) when 
the Markov operators are constructed from 
{\em $\alpha$-firmly nonexpansive mappings} in $\Rn$ (Definition \ref{d:a-fne}).  
We conclude 
Section \ref{sec:theory} with the  proof in Section \ref{sec:rates} of the last main result, 
Theorem \ref{t:msr convergence}, 
which provides for a quantification of convergence of the RFI 
when the underlying mappings are only {\em almost}
$\alpha$-firmly nonexpansive {\em in expectation} (Definition \ref{d:afne oa}) 
and when the {\em discrepancy} between a given measure and the set of invariant 
measures of the Markov operator, \eqref{eq:Psi},  
is {\em metrically subregular} (Definition \ref{d:(str)metric (sub)reg}). 
We conclude this study with Section \ref{sec:incFeas} where 
we focus on applications to optimization on measure spaces and 
(inconsistent) feasibility.  

\section{RFI and the Stochastic Fixed Point Problem}
\label{sec:consistent_feas_prob}

In this section we give a rigorous formulation of the RFI, 
then interpret this as a Markov chain and define the
corresponding Markov operator. We then formulate modes of convergence
of these Markov chains to invariant measures for the Markov operators
and formulate the stochastic feasibility and stochastic fixed point problems.
At the end of this section we present the  main
results of this article.  The proofs of these results are developed in 
Section \ref{sec:theory}.  

Our notation is standard.  As usual, $\mathbb{N}$ denotes the natural
numbers {\em including} $0$.  
For $G$, an abstract topological space,  
$\mathcal{B}(G)$ denotes the Borel $\sigma$-algebra and 
$(G,\mathcal{B}(G))$ is the 
corresponding measure space.  We denote by $\mathscr{P}(G)$ the set of all
probability measures on $G$.  The \emph{support of the probability measure}
$\mu$ is the smallest closed set $A$, for which $\mu(A)=1$  
and is denoted by $\Supp \mu$. 

There is a lot of overlapping notation in probability theory.  Where possible
we will try to stick to the simplest conventions, but the 
context will make certain notation preferable. 
The notation $X \sim \mu\in \mathscr{P}(G)$ means that the law
 of $X$, denoted $\mathcal{L}(X)$, satisfies 
 $\mathcal{L}(X)\equiv\mathbb{P}^{X} :=\mathbb{P}(X \in \cdot) = \mu$, 
 where $\mathbb{P}$ is the probability
 measure on some underlying probability space.  All of these different ways of 
 indicating a measure $\mu$ will be used.  

Throughout, the pair $(G,d)$ denotes a  separable metric space with metric
$d$ and $\mathbb{B}(x,r)$ is the open ball centered at $x\in G$ with
radius $r>0$; the closure of the ball is denoted
$\overline{\mathbb{B}}(x,r)$.  All of our results concerning existence
of invariant measures, tightness of sequences and convergence will assume
that $(G,d)$ is Polish (i.e. also complete);  in characterizing the regularity
of the building blocks, completeness of the metric space is not required.  

The distance of a point $x\in G$ to a set 
$A\subset G$ is denoted by $d(x,A)\equiv \inf_{w\in A}d(x,w)$. 
For the ball of radius $r$ around a
subset of points $A\subset G$, we write
$\mathbb{B}(A,r)\equiv \bigcup_{x\in A} \mathbb{B}(x,r)$.  The
$0$-$1$-indicator function of a set $A$ is given by
\[
\mathds{1}_A(x)=\begin{cases}1&\mbox{ if }x\in A,\\
                 0&\mbox{ else.}
                \end{cases}
\]

Continuing with the development initiated in the introduction, we will consider 
a collection of mappings $\mymap{T_{i}}{G}{G}$, $i \in I$, on $(G,d)$
(a separable complete metric space), where $I$ is an arbitrary
index set.    
The measure space of indexes is denoted by
$(I,\mathcal{I})$, and $\xi$ is an $I$-valued random variable on the
probability space $(\Omega,\mathcal{F},\mathbb{P})$. 
The pairwise independence of two random variables $\xi$ and $\eta$ is
denoted $\xi\indep\eta$.  The random variables $\xi_k$ in the sequence
$(\xi_{k})_{k\in\mathbb{N}}$ (abbreviated $(\xi_{k})$) 
are independent and identically
distributed (i.i.d.) \ with $\xi_{k}$ distributed as
$\xi$ ($\xi_k\sim \xi$).  
The method of random
function iterations is formally presented in Algorithm \ref{algo:RFI}.
\begin{algorithm}    
\SetKwInOut{Output}{Initialization}
  \Output{Set $X_{0} \sim \mu_0 \in \mathscr{P}(G)$, $\xi_k\sim\xi\quad \forall k\in\Nbb$.}
    \For{$k=0,1,2,\ldots$}{
            {$ X_{k+1} = T_{\xi_{k}} X_{k}$}\\
    }
  \caption{Random Function Iterations (RFI)}\label{algo:RFI}
\end{algorithm}

\noindent We will use the notation
\begin{equation}\label{eq:X_RFI}
  X_{k}^{X_0} := T_{\xi_{k-1}} \ldots T_{\xi_{0}} X_{0}
\end{equation}
to denote the sequence of the RFI initialized with $X_0\sim \mu_0$.
When characterizing sequences initialized
with the delta distribution of a point we use the notation $X_{k}^{x}$.
The following assumptions will be employed
throughout.

\begin{assumption}\label{ass:1}
  \begin{enumerate}[(a)]
  \item\label{item:ass1:indep} $\xi_{0},\xi_{1}, \ldots, \xi_{k}$
    are i.i.d with values on $I$ and $\xi_k\sim \xi$.  $X_0$ is an 
    random variable with values on $G$, independent from $\xi_k$. 
  \item\label{item:ass1:Phi} The function $\mymap{\Phi}{G\times
      I}{G}$, $(x,i)\mapsto T_{i}x$ is measurable.
  \end{enumerate}
\end{assumption}

\subsection{RFI as a Markov chain}
\label{sec:SPMasMC}
Markov chains are conveniently defined in terms of {\em transition kernels}.
A  transition kernel is a mapping 
$\mymap{p}{G\times \mathcal{B}(G)}{[0,1]}$ that is measurable in the first
argument and is a probability measure in the second argument; 
that is, \ $p(\cdot,A)$ is measurable for all $A \in
\mathcal{B}(G)$  and
$p(x,\cdot)$ is a probability measure for all $x \in G$.
\begin{definition}[Markov chain]
  A sequence of random variables $(X_{k})$,
  $\mymap{X_{k}}{(\Omega,\mathcal{F},\mathbb{P})}{(G,\mathcal{B}(G))}$
  is called Markov chain with transition kernel $p$ if for all $k \in
  \mathbb{N}$ and $A \in \mathcal{B}(G)$
  $\mathbb{P}$-a.s.\ the following hold:
  \begin{enumerate}[(i)]
  \item $\cpr{X_{k+1} \in A}{X_{0}, X_{1}, \ldots, X_{k}} =
    \cpr{X_{k+1} \in A}{X_{k}}$;
  \item $\cpr{X_{k+1} \in A}{X_{k}} = p(X_{k},A)$.
  \end{enumerate}
\end{definition}

\begin{prop}\label{thm:RFIasMC}
  Under Assumption \ref{ass:1},
  the sequence of random variables $(X_{k})$
  generated by Algorithm \ref{algo:RFI} is a Markov chain with transition
  kernel $p$ given by 
\begin{equation}\label{eq:trans kernel}
  (x\in G) (A\in
  \mathcal{B}(G)) \qquad p(x,A) \equiv \mathbb{P}(\Phi(x,\xi) \in A) =
  \mathbb{P}(T_{\xi}x \in A)
\end{equation}
for the measurable \emph{update function} 
$\mymap{\Phi}{G\times  I}{G}$, $(x,i)\mapsto T_{i}x$.
\end{prop}
\begin{proof}
It follows from \cite[Lemma 1.26]{kallenberg1997} that the 
mapping  $p(\cdot,A)$ defined by  \eqref{eq:trans kernel} is measurable for all 
$A \in\mathcal{B}(G)$,  and it is immediate from the definition that 
$p(x,\cdot)$ is
a probability measure for all $x \in G$.  So $p$ defined by 
\eqref{eq:trans kernel} 
is a transition kernel.  The remainder of the 
statement is an immediate consequence of the disintegration theorem 
(see, for example, \cite{stroock2010probability}). 
\end{proof}

The Markov operator $\mathcal{P}$ is defined pointwise for a measurable 
function 
$\mymap{f}{G}{\mathbb{R}}$ via
\begin{align*}
  (x\in G)\qquad \mathcal{P}f(x):= \int_{G} f(y) p(x,\dd{y}),
\end{align*}
when the integral exists. Note that
\begin{align*}
  \mathcal{P}f(x) = \int_{G} f(y)
  \mathbb{P}^{\Phi(x,\xi)}(\dd{y}) = \int_{\Omega}
  f(T_{\xi(\omega)}x) \mathbb{P}(\dd{\omega})= \int_{I}
  f(T_{i}x) \mathbb{P}^{\xi}(\dd{i}).
\end{align*}

The Markov operator $\mathcal{P}$ 
is \emph{Feller} if $\mathcal{P}f \in C_{b}(G)$ whenever $f \in
C_{b}(G)$, where $C_{b}(G)$ is the set of bounded and continuous
functions from $G$ to $\mathbb{R}$. 
This property is central to the theory of existence of invariant measures 
introduced below. 
The next fundamental result establishes the relation of the Feller property 
of the Markov operator to the generating mappings $T_{i}$.  
\begin{prop}[Theorem 4.22 in \cite{Bellet2006}]
\label{thm:Feller}
  Under Assumption \ref{ass:1}, if $T_{i}$ is continuous for all $i\in
  I$, then the Markov operator $\mathcal{P}$ is Feller.
\end{prop}

Let $\mu\in \mathscr{P}(G)$.  In a slight abuse of notation we denote the
dual Markov operator $\mymap{\mathcal{P}^{*}}{\mathscr{P}(G)}{\mathscr{P}(G)}$ 
acting on a measure $\mu$ by action on 
the right by $\mathcal{P}$ via
\begin{align*}
  (A \in \mathcal{B}(G))\qquad (\mathcal{P}^{*}\mu) (A):=
  (\mu\mathcal{P}) (A) := \int_{G} p(x,A) \mu(\dd{x}).
\end{align*}
This notation allows easy identification of the distribution of the $k$-th 
iterate of the Markov chain generated
by Algorithm \ref{algo:RFI}: $\mathcal{L}(X_{k}) =
\mu_0 \mathcal{P}^{k}$.

\subsection{The Stochastic Fixed Point Problem}\label{sec:consist RFI}
As studied in \cite{HerLukStu19a}, the {\em stochastic feasibility}
problem is stated as follows:
\begin{align}
  \label{eq:stoch_feas_probl}
\mbox{ Find }  x^{*} \in C := \mysetc{x \in G}{\mathbb{P}(x = T_{\xi}x) = 
1}.
\end{align}
A point $x$ such that $x = T_{i}x$ is a {\em fixed point} of the operator $T_{i}$;
the set of all such points is denoted by 
\begin{align*}
  \Fix T_{i} = \mysetc{x \in G}{x = T_{i}x}.
\end{align*}
In \cite{HerLukStu19a} it was assumed that $C\neq\emptyset$.  If
$C=\emptyset$ we call this the {\em inconsistent stochastic
  feasibility} problem.  
  
Inconsistent stochastic feasibility is far from exotic. 
Take, for example, the not unusual assumption of additive noise:
define $T_\xi(x)\equiv f(x) + \xi$ where $\mymap{f}{\Rn}{\Rn}$ and 
$\xi$ is a measure without point masses.  Then 
$\mathbb{P}(T_\xi(x) = x) = \mathbb{P}(\xi(x) = x-f(x)) = 0$.  More concretely,
let $f=\Id - t\nabla F$ where $\mymap{F}{\Rn}{\mathbb{R}}$ is a differentiable, strongly
convex function and $t$ is some appropriately small stepsize.  This yields the noisy
gradient descent method 
\[
 T_\xi(x) = x - t\nabla F(x)+\xi. 
\]
Though this has no fixed point, the additive noise $\xi$ can be constructed so that the 
resulting Markov chain is ergodic and its (unique!) invariant distribution concentrates around
the unique global minimum of $F$.  See Example  \ref{ex:spg ncvx} for 
more discussion. 

The inconsistency of the problem formulation is an artifact of asking
the wrong question.  
A fixed point of the (dual) Markov operator $\mathcal{P}$
is called an \emph{invariant} measure, i.e.\ $\pi \in \mathscr{P}(G)$
is invariant whenever $\pi \mathcal{P} = \pi$.  The
set of all invariant probability measures is denoted by 
$\inv \mathcal{P}$.
We are interested in the
following generalization of \eqref{eq:stoch_feas_probl}:
\begin{align}
  \label{eq:stoch_fix_probl}
  \mbox{Find}\qquad \pi\in\inv\mathcal{P}.
\end{align}
We refer to this as the {\em stochastic fixed
  point problem.}  


\subsection{Modes of convergence}
\label{sec:modesofcvg}
In \cite{HerLukStu19a}, we considered
almost sure convergence of the sequence $(X_k)$ to a random variable
$X$:
\[  
X_k\to X \text{ a.s.} \quad \text{as } k \to \infty.
\]
Almost sure convergence is commonly encountered
in the studies of stochastic algorithms in optimization, 
and can be guaranteed for consistent stochastic feasibility 
under most of the regularity assumptions on $T_i$ considered 
here (see \cite[Theorem 3.8 and 3.9]{HerLukStu19a}) though this
does not require the full power of the theory of general Markov
processes.   In fact, the 
next result shows that almost sure convergence is only possible
for consistent stochastic feasibility.
The following statement first appeared in Lemma 3.2.1 of \cite{Hermer2019}.  
\begin{prop}[a.s. convergence implies consistency]\label{thm:Hermer} 
Let $\mymap{T_i}{G}{G}$ 
be continuous for all $i\in I$ with respect to the metric $d$.  
Let $\pi\in \inv\mathcal{P}\neq\emptyset$ 
and $X_0\sim\pi$.  Generate the sequence $(X_k^{X_0})_{k\in\Nbb}$
via Algorithm \ref{algo:RFI} where $X_0\indep \xi_k$ for all $k$.  
If the sequence $(X_k^{X_0})$ converges almost surely, 
then the stochastic feasibility problem is consistent.   Moreover, 
$\Supp\pi\subset C$.   
\end{prop}
Before proceeding to the proof, note that the {\em measure} remains 
the same for each iterate $X_k^{X_0}$ -  the issue here is when the 
{\em iterates} converge (almost surely).  
\begin{proof}
In preparation for our argument, which is by contradiction, choose any 
$x\in\Supp\pi$ where 
$\pi\in\inv\mathcal{P}$,  and define 
 \[
 I^\epsilon\equiv \set{i\in I}{d(T_ix,x)>\epsilon}
 \]
 for $\epsilon\geq 0$.   Note that $I^\epsilon\supseteq I^{\epsilon_0}$ 
whenever 
 $\epsilon\leq \epsilon_0$.  Define the set
\[
 J^\epsilon_\delta\equiv \set{i\in I}{d(T_ix, T_iy)\leq \epsilon, \quad 
 \forall y\in \overline{\mathbb{B}}\paren{x, \delta}}.
\]
These sets 
satisfy $ J^\epsilon_{\delta_1}\subset J^\epsilon_{\delta_2}$ whenever 
$\delta_1\geq \delta_2$
and, since $T_i$ is continuous for all $i\in I$, we have that for each $\epsilon>0$,  
$ J^\epsilon_\delta\uparrow I$ as $\delta\to 0$.
A short argument shows 
that $I^\epsilon$ and $J^\epsilon_\delta$ are measurable for each $\delta$ and  $\epsilon>0$.  
 
Suppose now, to the contrary, that $C=\emptyset$.  
Then $\mathbb{P}(T_\xi x=x)<1$ and hence 
$\mathbb{P}(d(T_\xi x, x)>0)>0$.    Since $I^\epsilon\supseteq I^{\epsilon_0}$ 
for 
 $\epsilon\leq \epsilon_0$ 
we have $\mathbb{P}^\xi(I^{\epsilon_0})\leq \mathbb{P}^\xi(I^{\epsilon})$
 whenever $\epsilon\leq \epsilon_0$.  
In particular, there must exist an $\epsilon_0$ such that 
$0<\mathbb{P}^\xi(I^{\epsilon_0})$.  On the other hand, 
there is a constant $\delta>0$ such that $\delta< \epsilon_0/2$ and 
$\mathbb{P}^\xi(K^{\epsilon_0}_\delta)>0$
where $K^{\epsilon_0}_\delta\equiv I^{\epsilon_0}\cap J^{\epsilon_0/2}_\delta$. 
 This construction then yields
\[
(\forall i\in K^{\epsilon_0}_\delta)\qquad d(T_iy,x)\geq  d(T_ix,x) -  
d(T_iy,T_ix)\geq \frac{\epsilon_0}{2}>\delta
 \quad\forall y\in \overline{\mathbb{B}}(x,\delta).   
\]

Next, we claim that $X_k^{X_0}\sim\pi$ for all 
$k\in\Nbb$.  Indeed, for any $Y \sim \pi$, if 
  $\xi$ is independent of $Y$, then $T_{\xi} Y \sim \pi$.
To see this, note that 
For $A \in \mathcal{B}(G)$ 
Fubini's Theorem and disintegration 
yield   
  \begin{align*}
    \mathbb{P}(T_{\xi}Y \in A) &= \mathbb{E}[
    \cex{\1_{A}(T_{\xi}Y)}{\xi}] = \mathbb{E} \int \1_{A}(T_{\xi}y)
    \pi(\dd{y}) = \int \int \1_{A}(z) \mathbb{P}^{T_{\xi}y}(\dd{z})
    \pi(\dd{y}) \\ &= \int \int \1_{A}(z) p(y,\dd{z}) \pi(\dd{y}) =
    \pi\mathcal{P}(A) = \pi(A) = \mathbb{P}(Y \in A).
  \end{align*}
It follows that $\Supp \mathcal{L}(Y) =
\Supp \mathcal{L}(T_{\xi} Y)$, and since $\xi_k$ are i.i.d, 
 $X_{k}^{Y} \sim \pi$ for all $k \in \mathbb{N}$. This establishes
 the claim. 

The 
independence of $\xi_k$ and $X_0$ for all $k$ implies the independence of 
$\xi_k$ and $X_k^{X_0}$ for all $k$.  Moreover,  
$\mathbb{P}\paren{X_k^{X_0} \in  B_\delta} = \pi(B_\delta) > 0$
for all $k\in \Nbb$, where to avoid clutter we denote $B_\delta\equiv 
\mathbb{B}(x,\delta)$.  
This yields
\[
(\forall k\in \Nbb)\quad  \mathbb{P}\paren{X^{X_0}_k\in B_\delta, 
X^{X_0}_{k+1}\notin B_\delta} \geq 
 \mathbb{P}\paren{X^{X_0}_k\in B_\delta, \xi_k\in K^{\epsilon_0}_\delta} = 
 \pi(B_\delta)\mathbb{P}^\xi(K^{\epsilon_0}_\delta)>0.
\]
Thus, we also have 
\begin{eqnarray*}
(\forall k\in \Nbb)\quad  \mathbb{P}\paren{X^{X_0}_k\in B_\delta, 
X^{X_0}_{k+1}\in B_\delta} &=& 
\mathbb{P}\paren{X^{X_0}_k\in B_\delta}-\mathbb{P}\paren{X^{X_0}_k\in B_\delta, 
X^{X_0}_{k+1}\notin B_\delta} \\
&\leq&  \pi(B_\delta) -  \pi(B_\delta) \mathbb{P}^\xi(K^{\epsilon_0}_\delta)\\
&=&  \pi(B_\delta) (1-  \mathbb{P}^\xi(K^{\epsilon_0}_\delta)) < \pi(B_\delta).
\end{eqnarray*}
However, by assumption, $X_k^{X_0}\to X_*$ a.s. for some random variable $X_*$ with 
$\mathbb{P}\paren{X_* \in  B_{\delta}} = \pi(B_{\delta}) > 0.$ If  $X_* \in  B_{\delta}$
then due to the a.s. convergence there exists a (random) $k_*$ such that $X^{X_0}_{k}\in B_\delta$
for all $k\geq k_*.$ This implies that 
\[
  \mathbb{P}\paren{X^{X_0}_k\in B_\delta, X^{X_0}_{k+1}\in B_\delta, X_* \in  B_{\delta} }\to 
   \mathbb{P}\paren{X_* \in  B_{\delta}} = \pi(B_{\delta}).
\]
which is a contradiction since by the above
\[
  \mathbb{P}\paren{X^{X_0}_k\in B_\delta, X^{X_0}_{k+1}\in B_\delta, X_* \in  B_{\delta} }\leq 
   \mathbb{P}\paren{X^{X_0}_k\in B_\delta, X^{X_0}_{k+1}\in B_\delta}< \pi(B_\delta).
\]
So it must be true that $\mathbb{P}\paren{d(T_\xi 
x,x)>0}=0$.  In other 
words, $\mathbb{P}\paren{T_\xi x=x}=1$, hence $C\neq\emptyset$.    
Moreover, since the point $x$ was any arbitrary point in 
$\Supp\pi$, we conclude that 
$\Supp\pi\subset C$. 
\end{proof}

For inconsistent feasibility, or more general 
stochastic fixed point problems that are the aim of the RFI,  Algorithm \ref{algo:RFI}, 
we focus on {\em convergence in distribution}.
Let $(\nu_{k})$ be a sequence of probability measures on $G$.  The sequence 
$(\nu_{k})$ is said to converge in distribution to $\nu$ whenever $\nu \in
\mathscr{P}(G)$ and for all $f \in C_{b}(G)$ it holds that $\nu_{k} f
\to \nu f$ as $k \to \infty$, where $\nu f := \int f(x) \nu(\dd{x})$. 
Equivalently a sequence of random variables $(X_{k})$ is said to converge in 
distribution if their laws $(\mathcal{L}(X_{k}))$ do.

We now consider two modes of 
convergence in distribution for the corresponding sequence of measures 
$(\mathcal{L}(X_{k})))_{k \in \mathbb{N}}$ 
on 
$\mathscr{P}(G)$:
\begin{enumerate}
\item convergence in distribution of the Ces\`{a}ro averages of 
 the  sequence $(\mathcal{L}(X_{k}))$ to a probability measure 
  $\pi \in \mathscr{P}(G)$, i.e.\ for any $f \in C_{b}(G)$
  \begin{align*}
    \nu_{k}f:= \frac{1}{k} \sum_{j=1}^{k} \mathcal{L}(X_{j}) f =
    \mathbb{E}\left[\frac{1}{k}\sum_{j=1}^{k} f(X_{j})\right] \to \pi
    f, \qquad \text{as } k \to \infty; 
  \end{align*}
\item convergence in distribution 
of the sequence 
  $(\mathcal{L}(X_{k}))$ to a probability measure $\pi \in \mathscr{P}(G)$, i.e.\
  for any $f \in C_{b}(G)$
  \begin{align*}
    \mathcal{L}(X_{k})f = \mathbb{E}[f(X_{k})] \to \pi f, 
    \qquad \text{as } k \to \infty.
  \end{align*}
\end{enumerate}
Clearly, the second mode of convergence implies the first.  
This is used in Section \ref{sec:setCvgNE} and Section \ref{sec:cvgAverages}.

An elementary fact from the theory of Markov chains
(Proposition \ref{thm:construction_inv_meas}) is that, if the Markov operator 
$\mathcal{P}$ is Feller and $\pi$ is a cluster
point of $(\nu_k)$ with respect to convergence in distribution 
then $\pi$ is an invariant probability measure.  
Existence of invariant 
measures for a Markov operator then amounts to verifying that the 
operator is Feller (by Proposition \ref{thm:Feller}, automatic if the $T_i$ are 
continuous) and that cluster points exist (guaranteed by 
{\em tightness}  -- or compactness with respect to the topology of convergence in 
distribution -- of the sequence, see \cite[Section 5]{Billingsley}.
In particular, this means that 
there exists a convergent subsequence $(\nu_{k_{j}})$ with
\begin{align*}
  (\forall f\in C_b(G))\qquad \nu_{k_{j}} f=
  \mathbb{E}\left[\frac{1}{k_{j}}\sum_{i=1}^{k_{j}} f(X_{i})\right] \to \pi f,
  \qquad \text{as } j \to \infty.
\end{align*}
Convergence of the whole sequence, i.e.\ $\nu_{k} \to \pi$, amounts
then to showing that $\pi$ is the unique cluster point of $(\nu_{k})$
(see Proposition \ref{thm:cvg_subsequences}).  

Quantifying convergence is essential for establishing estimates for the distance 
of the iterates to the limit point, when this exists.  
\begin{definition}[R- and Q-linear convergence to points, Chapter 9 of 
\cite{OrtegaRheinboldt70}]\label{d:q-r-lc}
Let $(x_k)$ be a sequence in a metric space $(G,d)$.
\begin{enumerate}
\item $(x_k)_{k\in\Nbb}$ is said to \emph{converge R-linearly} to $\tilde{x}$ with rate $c\in [0,1)$
if there is a constant $\beta>0$ such that
\begin{align}\label{LinCon_Seq}
d(x_k,\tilde{x}) \le \beta c^k\quad \forall k\in \Nbb.
\end{align}
\item $(x_k)_{k\in\Nbb}$ is said to \emph{converge Q-linearly} to $\tilde{x}$ with rate $c\in [0,1)$ if
\[
d(x_{k+1},\tilde{x}) \le c d(x_{k},\tilde{x})\quad \forall k\in \Nbb.
\]
\end{enumerate}
\end{definition}
By definition, Q-linear convergence implies R-linear convergence with the same rate; 
the converse implication does not hold in general.  Q-linear convergence is 
encountered with contractive fixed point mappings, and this leads to 
a priori and a posteriori error estimates on the sequence.  This type of convergence is 
referred to as {\em geometric} or {\em exponential} convergence
in different communities.  The crucial distinction between R-linear and Q-linear  
convergence is that R-linear convergence  
permits neither a priori nor a posteriori error estimates.  

Common metrics for spaces of measures are the 
{\em Prokhorov-L\`evy distance} and the {\em Wasserstein metric}.  
  \begin{definition}[Prokhorov-L\`evy\&Wasserstein distance]\label{d:PL}
  Let $(G,d)$ be a separable complete metric space and let $\mu, \nu \in 
\mathscr{P}(G)$.
\begin{enumerate}[(i)]
\item The Prokhorov-L\`evy distance, denoted by $d_P$, is defined by
  \begin{equation}\label{eq:PL}
    d_{P}(\mu,\nu) = \inf\mysetc{\epsilon >0}{\mu(A) \le
      \nu(\mathbb{B}(A,\epsilon)) + \epsilon,\, \nu(A)\le
      \mu(\mathbb{B}(A,\epsilon))+\epsilon \quad \forall A \in
      \mathcal{B}(G)}.
  \end{equation} 
  \item  For $p \geq 1$ let 
  \begin{equation}\label{eq:p-probabiliy measures}
       \mathscr{P}_{p}(G) = \mysetc{\mu \in \mathscr{P}(G)}{ \exists\, x
      \in G \,:\, \int d^{p}(x,y) \mu(\dd{y}) < \infty}.
  \end{equation}
  The Wasserstein $p$-metric on $\mathscr{P}_{p}(G)$, denoted $W_p$, is defined by 
 \begin{equation}\label{eq:Wasserstein}
  W_p(\mu, \nu)\equiv \paren{\inf_{\gamma\in C(\mu, \nu)}\int_{G\times G} 
d^p(x,y)\gamma(dx, dy)}^{1/p}\quad (p\geq 1)
 \end{equation}
 where $C(\mu, \nu)$ is the set of couplings of $\mu$ and $\nu$ 
 (measures on the product space $G\times G$ whose marginals
 are $\mu$ and $\nu$ respectively - see \eqref{eq:couplingsDef}).  
 \end{enumerate}
  \end{definition}

\subsection{Regularity}
Our main results concern convergence of Markov chains under increasingly 
restrictive regularity assumptions on the mappings $\{T_{i}\}$.  The regularity
of $T_{i}$ is dictated by the application, and our primary interest is to follow
this through to the regularity of the corresponding Markov operator.  In 
\cite{LukNguTam18} a framework was developed for a quantitative 
convergence analysis of set-valued mappings $T_{i}$ that are {\em calm} 
(one-sided Lipschitz continuous -- in the sense of set-valued-mappings -- 
with Lipschitz constant greater than 1).  A set-valued self-mapping on a metric 
space $(G, d)$ is denoted $\mmap{T}{G}{G}$.  
 This setting includes, for instance, 
applications involving feasibility -- consistent and inconsistent -- as 
well as many randomized algorithms for large-scale optimization, convex and nonconvex.  
Studies concurrent with the present one define the regularity of fixed point mappings
in {\em $p$-uniformly convex spaces ($p\in(1,\infty)$) with parameter $c>0$} \cite{BLL, LauLuk21}.  
These are uniquely geodesic metric spaces $(G, d)$ 
for which  the following inequality holds \cite{NaoSil11}:
\begin{equation}\label{e:p-ucvx}
(\forall t\in [0,1])(\forall x,y,z\in G) \quad
d(z, (1-t)x\oplus ty)^p\leq (1-t)d(z,x)^p+td(z,y)^p - \tfrac{c}{2}t(1-t)d(x,y)^p
\end{equation}
where $w=(1-t)x\oplus ty$ for $t\in (0,1)$ denotes the point $w$ on the geodesic 
connecting $x$ and $y$ such that $d(w,x)=td(x,y)$.
The constant $c$ is tied to the curvature and the diameter of the space.  
When $p=c=2$, this inequality
defines a CAT($0$) space (Alexandrov \cite{Alexandrov} 
and Gromov \cite{Gromov}), the completion of which defines a Hadamard space. 
More generally, a CAT($\kappa$) space is a geodesic metric space with sufficiently small 
triangles possessing comparison triangles with sides the same length as the geodescic 
triangle but for which the distance between points on the geodesic triangle are less than 
or equal to the distance between corresponding points on the comparison triangle.
CAT($\kappa$) spaces are separable, but not complete, and locally $2$-uniformly convex 
with parameter 
$c$ approaching the value $2$ from below as the 
diameter of the local neighborhood vanishes \cite[Proposition 3.1]{Ohta07}.  

To keep the notation simple, we will restrict ourselves to CAT($\kappa$) spaces (that is, the case $p=2$), though 
we note that the exponents  in the definition below are a consequence of this choice of $p$.  
\begin{definition}[pointwise almost ($\alpha$-firmly) nonexpansive mappings in CAT($\kappa$) metric spaces]
\label{d:a-fne} 
Let $(G,d)$ be a CAT($\kappa$) metric space and $D\subset G$ and let 
$\mmap{F}{D}{G}$.
  \begin{enumerate}[(i)]
  \item The mapping $F$  is said to be \emph{pointwise  
  almost nonexpansive at $x_0\in D$ on $D$}, abbreviated {\em pointwise ane},
whenever  
\begin{equation}\label{eq:pane}
\exists \epsilon\in[0,1):\quad  d(x^+, x_0^+) \le \sqrt{1+\epsilon}\,d(x,x_0), 
\qquad \forall x \in D, \forall x^+\in Fx, x_0^+\in Fx_0.
\end{equation}
The {\em violation} is a value of $\epsilon$ for which \eqref{eq:pane} holds.
    When the above inequality holds for all $x_0\in D$ then $F$ is said to be 
    {\em almost nonexpansive on $D$} ({\em ane}).  When $\epsilon=0$ the 
    mapping $F$ is said to be 
    {\em (pointwise) nonexpansive}.  
  \item The mapping $F$ is said to be {\em pointwise almost $\alpha$-firmly 
nonexpansive  at $x_0\in D$ on $D$}, abbreviated {\em pointwise a$\alpha$-fne} 
  whenever
\begin{eqnarray}
&&
	\exists \epsilon\in[0,1)\mbox{ and }\alpha\in (0,1):\nonumber\\
	  \label{eq:paafne}&&	\quad	d^2(x^+,x_0^+) \le (1+\epsilon) d^2(x,x_0) - 
      \tfrac{1-\alpha}{\alpha}\psi_c(x,x_0, x^+, x_0^+)\\
&&      \qquad\qquad\qquad\qquad\qquad\quad \forall x \in D, \forall x^+\in Fx, \forall x_0^+\in Fx_0,
		\nonumber
\end{eqnarray}
where the {\em transport discrepancy} $\psi_c$ of  $F$ at $x, x_0$, $x^+\in Fx$ and $x_0^+\in Fx_0$
is defined by  
\begin{eqnarray}
&&\!\!\!\!\!\!\!\!\psi_c(x,x_0, x^+, x_0^+)\equiv \nonumber\\
\label{eq:delta}
&&\!\!\!\!\!\!\!\! \tfrac{c}{2}\left(d^2(x^+, x)+d^2(x_0^+, x_0) + d^2(x^+, x_0^+) + 
d^2(x, x_0)  - d^2(x^+, x_0)   - d^2(x,x_0^+)\right).
\end{eqnarray}
When the above inequality holds for all $x_0\in D$ then $F$ is said to be 
{\em almost $\alpha$-firmly nonexpansive on $D$}, ({\em a$\alpha$-fne}). 
The {\em violation} is the constant  
$\epsilon$ for which \eqref{eq:paafne} holds.
When $\epsilon=0$ the mapping $F$ is said to be 
    {\em (pointwise) $\alpha$-firmly nonexpansive}, abbreviated {\em (pointwise) $\alpha$-fne}.  
  \end{enumerate}
\end{definition}

The transport discrepancy $\psi_c$ is a central object for characterizing the regularity of mappings 
in metric spaces and ties the regularity of the mapping to the geometry of the space.  The parameter
$c$ is determined by the curvature and the diameter of the space, $(G,d)$.
The following lemma is derived from  
\cite[pp. 94]{Berdellima20}.
\begin{lemma}[$\psi_2$ is nonnegative in CAT($0$) spaces]\label{lem:psi_f nonneg}
Let $(G,d)$ be a CAT($0$) metric space and $\mmap{F}{D}{G}$ for $D\subset G$.  
Then the transport discrepancy defined by \eqref{eq:delta} is nonnegative for all $x,y\in D$, 
$x^+\in Fx$, $y^+\in Fy$. 
Moreover, if $F$ is pointwise a$\alpha$-fne at $x_0\in D$ 
with violation $\epsilon$ on $D$, then $F$ is pointwise ane at $x_0$ on $D$ with 
violation at most $\epsilon$.  
\end{lemma}
\begin{proof}
For any self-mapping on a CAT($0$) space, 
the following four-point inequality holds at {\em any} points $x, y, u, v$, 
\cite[Theorem 2.3.1]{Jost97}:
\begin{equation}\label{eq:Cauchy space}
    d^2(u, y)+d^2(x,v)-d^2(u, x)-d^2(v, y)
   \leq 2d(u,v)d(x,y).
\end{equation}
Thus $\psi_2(x, x_0, x^+, x_0^+)$ defined by \eqref{eq:delta} (with $c=2)$ is 
nonnegative in this setting.  It follows immediately 
from the definition \eqref{eq:paafne}, then, that in CAT($0$) spaces 
pointwise a$\alpha$-fne mappings are pointwise ane with at most the same violation.  
\end{proof}
In CAT($\kappa$) spaces the above statement does not hold.  It is well known, for example,  
that  in a CAT($\kappa$) metric space the projector onto a convex set is 
$\alpha$-fne with $\alpha=1/2$, but it is not nonexpansive \cite{AriLeuLop14}.  

The definition of pointwise a$\alpha$-fne
mappings generalizes the same notions developed in \cite{LukNguTam18}
for Euclidean spaces.  The notion of 
{\em averaged} mappings dates back to Mann, Krasnoselskii, and others
\cite{mann1953mean, krasnoselski1955, edelstein1966, Gubin67, BruckReich77}, 
while the name ``averaged'' seemed to stick from \cite{BaiBruRei78}.  We depart from this 
tradition because it does not fit with nonlinear spaces.    

In normed linear spaces, Baillon and Bruck \cite{Baillon1996} showed that 
nonexpansive mappings whose orbits are bounded under 
convex relaxations are {\em asymptotically regular} with a universal rate constant.  
Precisely: let $\mymap{T}{D}{D}$ be nonexpansive, where $D$ is a convex subset 
of a normed linear space, and define $x_m$ recursively by 
$x_m= T_\lambda x_{m-1}\equiv \left((1-\lambda)\Id + \lambda T\right)$
for $\lambda \in (0,1)$ and $x_0\in D$.   
If $\|x - TT_\lambda^k x\|\leq 1$ for any $\lambda\in(0,1)$ and for all $0\leq k\leq m$, then    
\cite[Main Result]{Baillon1996}
\begin{equation}\label{eq:universal 1}
\|x_m - Tx_m\|<\frac{\diam{D}}{\sqrt{\pi m\lambda(1-\lambda)}}. 
\end{equation}

Cominetti, Soto and Vaisman \cite{CominettiSotoVaisman} recently confirmed
a conjecture of Baillon and Bruck that a universal rate constant also holds for 
nonexpansive mappings with arbitrary relaxation in $(0,1)$ chosen at each iteration; 
in particular, that 
\[
\|x_m - T x_{m}\|\leq\frac{\diam{D}}{\sqrt{\pi\sum_{k=1}^m\lambda_k(1-\lambda_k)}},
\]
where $x_m$ is defined recursively by 
$x_m= T_{\lambda_m}x_{m-1}\equiv \left((1-\lambda_m)\Id + \lambda_m T\right)x_{m-1}$
for $\lambda_m\in (0,1)$ ($m=1,2, \dots$).  
The operators 
$T_{\lambda_m}$ all have the same set of fixed points (namely $\Fix T$), so these results are 
complementary to \cite{HerLukStu19a} where it was shown \cite[Theorem 3.5]{HerLukStu19a} 
that sequences of random variables on compact metric spaces  generated by Algorithm 
\ref{algo:RFI} with  {\em paracontractions} such as 
$T_{\lambda_m}$ above converge {\em almost surely} to a random variable in $\Fix T$, 
assuming that this is nonempty (see Proposition \ref{thm:Hermer} in this context). 
Necessary and sufficient conditions for linear 
convergence
of the {\em iterates} were also determined in a more limited setting in 
\cite[Theorems 3.11 and 3.15]{HerLukStu19a}.  
The results of \cite{Baillon1996,CominettiSotoVaisman,HerLukStu19a}, however, do not apply to 
inconsistent
stochastic feasibility considered here. 

Nonexpansive mappings have been explored in nonlinear 
metric spaces for instance in \cite{GoeRei84} and in Hadamard spaces recently in 
\cite{ReiSal16}. Our definition for $\alpha=1/2$ is equivalent (after some algebra) to the 
definition of {\em firmly contractive} mappings given in \cite[Definition 6]{Browder67} for 
Hilbert spaces (see \eqref{eq:paafne2} below).  The notion 
of $\lambda$-firmly nonexpansive mappings was defined in \cite{AriLeuLop14} in the context 
of $W$-hyperbolic spaces.  This nomenclature was appropriated in \cite{BLL} where 
it was shown that $\lambda$-firmly nonexpansive 
mappings are $\alpha$-fne, though the converse does not hold in general 
\cite[Proposition 4]{BLL}.   Rates of asymptotic regularity of compositions of 
firmly nonexpansive mappings on $p$-uniformly convex spaces have been established in 
\cite[Theorem 3.2]{RuiLopNic15}. 

The violation $\epsilon$ in \eqref{eq:pane} and \eqref{eq:paafne} is a recently introduced feature in the 
analysis of fixed point mappings, first appearing in this form in \cite{LukNguTam18}.  
It is interesting to note
that in the same article \cite{Baillon1996} where Baillon and Bruck showed \eqref{eq:universal 1} for 
nonexpansive mappings, 
they also observed that mappings with Lipschitz constant greater than one also 
behave nicely, and conjectured that something similar was also possible for this case.  Indeed, the analysis of 
\cite{LukNguTam18} shows how this works, though something like a universal constant has not been explored. 
Many  are familiar with mappings for which \eqref{eq:pane}  holds with $\epsilon<0$ at all $x_0\in G$ , 
i.e. contraction mappings.  In this case, 
the whole technology of pointwise a$\alpha$-fne mappings is not required since an 
appropriate application of Banach's fixed point theorem delivers existence of fixed points and convergence of 
fixed point iterations at a linear rate.  
We will have more to say about this later;  for the moment it suffices to note that the mappings associated 
with our target applications are expansive on all neighborhoods of fixed points and we will therefore require 
another property to guarantee convergence.    

When $\|\cdot\|$ is the norm induced by the inner product and 
$d(x,y)=\|x-y\|$, the transport discrepancy $\psi_2$ defined by \eqref{eq:delta} has 
the representation 
\begin{equation} \label{eq:nice ineq}
\psi_2(x,x_0, x^+, x_0^+) =  \|(x-x^+)-(x_0-x_0^+)\|^2.
\end{equation}
This representation 
shows the connection between our definition and more classical notions. 
Indeed, in a Hilbert space setting $(G, \|\cdot\|)$, 
a set-valued mapping $\mmap{F}{D}{G}$ $(D\subset G)$
is pointwise a$\alpha$-fne  at $x_0$ 
with constant $\alpha$ and violation at most $\epsilon$ on $D$ if and only 
if \cite[Proposition 2.1]{LukNguTam18}
 \begin{eqnarray}
&&\|x^+ - x_0^+\|^2 \le (1+\epsilon)\|x - x_0\|^2 -  
\tfrac{1-\alpha}{\alpha}\left\|((x-x^+) - (x_0-x_0^+) \right\|^2
\label{eq:paafne2}\\
 &&  \qquad\qquad\qquad\qquad\qquad\qquad\qquad\quad 
 \forall x \in D, \forall x^+\in Fx, \forall x_0^+\in Fx_0.\nonumber
 \end{eqnarray}

In the stochastic setting, to ease the notation and avoid certain technicalities,  
we will consider single-valued mappings $T_i$
that are only almost $\alpha$-firmly nonexpansive {\em in expectation}.
We can therefore write $x^+ = T_i x$ instead
of always taking some selection $x^+\in T_i x$ and verifying the desired properties
over all (measurable) selections, assuming these exist \cite{Wagner}.
The next definition uses the update function $\Phi$ defined 
in Assumption \ref{ass:1}(b).   
\begin{definition}[pointwise almost ($\alpha$-firmly) nonexpansive in
  expectation]\label{d:afne oa}
  Let $(G,d)$ be a $p$-uniformly convex metric space with constant $c$, let $\mymap{T_i}{G}{G}$ for
  $i\in I$, and let 
   $\mymap{\Phi}{G\times I}{G}$ be given by $\Phi(x,i) = T_{i}x$.
  Let $\psi_c$ be defined by \eqref{eq:delta} and   
let $\xi$ be an $I$-valued random variable.  
  \begin{enumerate}[(i)]
  \item   The mapping $\Phi$ is said to be \emph{pointwise
      almost nonexpansive in expectation at $x_0\in G$} on $G$, abbreviated 
      {\em pointwise ane in expectation}, whenever 
    \begin{align}\label{eq:panee}
		\exists \epsilon\in [0,1):\quad 
		\mathbb{E}\ecklam{d(\Phi(x,\xi), \Phi(x_{0},\xi))} \le \sqrt{1+\epsilon}\,
		d(x,x_0), \qquad \forall x \in G.
    \end{align}
    When the above inequality holds for all $x_0\in G$ then
    $\Phi$ is said to be {\em almost nonexpansive -- ane -- in expectation on
      $G$}.  As before, the violation is a value of  $\epsilon$ for which 
  \eqref{eq:panee} holds. When the violation is $0$, the qualifier ``almost'' is dropped. 
  \item The mapping $\Phi$ is said to be {\em pointwise almost $\alpha$-firmly 
nonexpansive  in expectation at $x_0\in G$} on $G$, abbreviated {\em pointwise a$\alpha$-fne
in expectation}, 
  whenever
  \begin{eqnarray}\label{eq:paafne i.e.}
&&\exists \epsilon\in [0,1), \alpha\in (0,1):\quad \forall x \in G,\\
&&\quad\mathbb{E}\ecklam{d^2(\Phi(x,\xi),\Phi(x_0,\xi))}\leq 
(1+\epsilon)d^2(x,x_0) - 
\tfrac{1-\alpha}{\alpha}\mathbb{E}\left[\psi_c(x,x_0, \Phi(x,\xi), \Phi(x_0,\xi))\right].
\nonumber
\end{eqnarray}
When the above inequality holds for all $x_0\in G$ then $\Phi$ is said to be 
{\em almost $\alpha$-firmly nonexpansive (a$\alpha$-fne) in expectation on $G$}.  The 
violation is a value of $\epsilon$ for which \eqref{eq:paafne i.e.} holds. 
When the violation is $0$, the qualifier ``almost'' is dropped and the abbreviation 
{\em $\alpha$-fne in expectation} is used. 
  \end{enumerate}
\end{definition}
\begin{prop}\label{r:nonneg psi_Phi}
Let $(G,d)$ be a CAT($0$) space.  The mapping 
$\mymap{\Phi}{G\times I}{G}$ given by $\Phi(x,i) = T_{i}x$
is pointwise a$\alpha$-fne in 
expectation  at $y$ on $G$ with constant $\alpha$ and violation at most $\epsilon$
and pointwise ane in expectation  at $y$ on $G$ 
with violation at most $\epsilon$ 
whenever $T_i$ is pointwise a$\alpha$-fne at $y$ on $G$ with constant $\alpha$ and 
violation no greater than $\epsilon$ for all $i$. 
\end{prop}
\begin{proof}
 By Lemma \ref{lem:psi_f nonneg},  whenever 
$(G,d)$ is a CAT($0$) space ($p$-uniformly convex with $p=2$ and $c=2$)
$\psi_2(x,y, \Phi(x,i), \Phi(y,i))\geq 0$ for all $i$ and for all $x,y\in G$, so 
the expectation $\mathbb{E}\left[\psi_2(x,y, \Phi(x,\xi), \Phi(y,\xi))\right]$ 
is well-defined and nonnegative for all $x,y\in G$ (the value $+\infty$ can be attained).  
This implies that, for all $i$, $T_i$ is pointwise ane at $y$ on $G$ with violation at most 
$\epsilon$ on $G$ whenever it is pointwise a$\alpha$-fne at $y$ with constant 
$\alpha$ on $G$ with violation at most 
$\epsilon$ on $G$ for all $i$.  It follows immediately from the definition, then, that  
$\Phi$ is pointwise a$\alpha$-fne in expectation at $y$ with constant 
$\alpha$ on $G$ with violation at most 
$\epsilon$ on $G$, and also pointwise ane in expectation at $y$ on $G$ with violation at most 
$\epsilon$ on $G$.  
\end{proof}

Lifting these notions of regularity to Markov operators yields an analogous definition on the 
space of measures which hinges on the update function  $\Phi$.  
To facilitate the discussion we denote 
the set of couplings where the distance 
$W_2(\mu_1, \mu_2)$ is attained by 
\begin{equation}\label{eq:Gamma opt}
 C_*(\mu_1,\mu_2)\equiv \left\{\gamma\in C(\mu_1, \mu_2)~\big{|}~ 
\int_{G\times G}d^2(x,y)\ \gamma(dx, dy) = W_2^2(\mu_1,\mu_2)\right\}.
\end{equation}
Note that by 
Lemma \ref{lemma:WassersteinMetric_prop}\eqref{lemma:WassersteinMetric_prop ii}
this set is nonempty when $W_2(\mu_1,\mu_2)$ is finite.

\begin{definition}[pointwise almost ($\alpha$-firmly) nonexpansive Markov operators]\label{d:afne rw}
  Let $(G,d)$ be a CAT($\kappa$) metric space, and let 
$\mathcal{P}$ be a Markov operator with transition kernel
\[
  (x\in G) (A\in
  \mathcal{B}(G)) \qquad p(x,A) \equiv \mathbb{P}(\Phi(x,\xi) \in A)
\]
where $\xi$ is an $I$-valued random variable and 
$\mymap{\Phi}{G\times  I}{G}$ is a measurable update function. 
  Let $\psi_c$ be defined by \eqref{eq:delta}.
  \begin{enumerate}[(i)]
  \item   The Markov operator is said to be \emph{pointwise
	  almost nonexpansive in measure at $\mu_0\in \mathscr{P}(G)$} on $\mathscr{P}(G)$, 
	  abbreviated {\em pointwise ane in measure}, whenever 
    \begin{align}\label{eq:paneim}
		\exists \epsilon\in [0,1):\quad W_2(\mu\Pcal, \mu_0\Pcal) \le \sqrt{1+\epsilon}\,
		W_2(\mu, \mu_0), \qquad \forall \mu\in \mathscr{P}(G).
    \end{align}
    When the above inequality holds for all $\mu_0\in \mathscr{P}(G)$ then
    $\Pcal$ is said to be {\em almost nonexpansive (ane)  in measure on
      $\mathscr{P}(G)$}.  As before, the violation is a value of  $\epsilon$ for which 
  \eqref{eq:paneim} holds. When the violation is $0$, the qualifier ``almost'' is dropped. 
  \item The Markov operator $\Pcal$ is said to be {\em pointwise almost $\alpha$-firmly 
nonexpansive  in measure at $\mu_0\in G$} on $\mathscr{P}(G)$, 
abbreviated {\em pointwise a$\alpha$-fne in measure}, 
  whenever
  \begin{eqnarray}
&&\exists \epsilon\in [0,1), \alpha\in (0,1): 
\qquad \forall\mu\in \mathscr{P}(G),\forall \gamma\in C_*(\mu, \mu_0)
\nonumber\\
&& W_2(\mu\Pcal, \mu_0\Pcal)^2\leq 
(1+\epsilon)W_2(\mu, \mu_0)^2 - \nonumber\\
&&\qquad \qquad\qquad \qquad
\tfrac{1-\alpha}{\alpha}\int_{G\times G}\mathbb{E}\left[\psi_c(x,y, \Phi(x,\xi), \Phi(y,\xi))\right] \gamma(dx, dy).
	  \label{eq:paafne i.m.}
\end{eqnarray}
When the above inequality holds for all $\mu_0\in \mathscr{P}(G)$ then $\Pcal$ is said to be 
{\em a$\alpha$-fne in measure on $\mathscr{P}(G)$}.  The 
violation is a value of $\epsilon$ for which \eqref{eq:paafne i.m.} holds. 
When the violation is $0$, the qualifier ``almost'' is dropped and the abbreviation 
{\em $\alpha$-fne in measure} is employed. 
  \end{enumerate}
\end{definition}

\begin{rem}\label{r:nonneg psi_Phi2}
By Lemma \ref{lem:psi_f nonneg},  when $(G, d)$ is a CAT($0$) space 
 the expectation on the right hand side of \eqref{eq:paafne i.m.} is nonnegative, and 
 the corresponding Markov operator is pointwise ane in measure 
 at $\mu_0$ whenever it is pointwise a$\alpha$-fne in measure at $\mu_0$
 (Proposition \ref{r:nonneg psi_Phi}). 
 In particular, when  
$\mu=\mu_0\in \inv\mathcal{P}$ the left hand side is zero and   
\[
\int_{G\times G} 
\mathbb{E}\left[\psi_2(x,y, T_\xi x, T_\xi y)\right]\ \gamma(dx, dy) = 0.
\]
Here the optimal coupling is the diagonal of the product space $G\times G$
and $\psi_2(x,x, T_\xi x, T_\xi x)=0$ for all $x\in G$.  
\end{rem}

\begin{prop}\label{thm:Tafne in exp 2 pafne of P}
  Let $(G,d)$ be a separable complete CAT($\kappa$) metric space and 
  $\mymap{T_i}{G}{G}$ for $i\in I$, let 
   $\mymap{\Phi}{G\times I}{G}$ be given by $\Phi(x,i) = T_{i}x$
   and let $\psi_c$ be defined by \eqref{eq:delta}. 
   Denote  by $\mathcal{P}$  the Markov operator with update function $\Phi$ and 
   transition kernel $p$ defined by
  \eqref{eq:trans kernel}.
  If $~\Phi$  is a$\alpha$-fne in expectation on $G$ with
  constant $\alpha\in (0,1)$ and violation $\epsilon\in [0,1)$, then the Markov operator 
  $\mathcal{P}$ is a$\alpha$-fne 
  in measure on $\mathscr{P}_2(G)$ 
  with constant $\alpha$ and violation at most $\epsilon$, that is, $\Pcal$ 
  satisfies 
  \begin{eqnarray}
W_2^{2}(\mu_1\mathcal{P}, \mu_2\mathcal{P})&\le&    
   (1+\epsilon)W_2^{2}(\mu_1, \mu_2) -  
   \tfrac{1-\alpha}{\alpha}  \int_{G\times G}
   \mathbb{E}\ecklam{\psi_c(x,y, \Phi(x,\xi), \Phi(y,\xi))}\ \gamma(dx, dy)
   \nonumber\\
\label{eq:alphfne meas}&&   
\qquad\qquad\qquad\qquad   
 \forall \mu_2, \mu_1\in \mathscr{P}_2(G), \ \forall \gamma\in C_*(\mu_1,\mu_2).  
  \end{eqnarray}
\end{prop}
\begin{proof}
If $W_2(\mu_1,\mu_2)=\infty$ the inequality holds trivially 
with the convention 
$+\infty-(+\infty)= +\infty$.    
So consider the case where $W_2(\mu_1,\mu_2)$ is finite.  Since $(G, d)$ is a separable, 
complete metric 
space, by 
Lemma \ref{lemma:WassersteinMetric_prop}\eqref{lemma:WassersteinMetric_prop ii}, 
the set of  optimal couplings $C_*(\mu_1, \mu_2)$ is nonempty.  
Since $\Phi$ is a$\alpha$-fne in expectation on $G$ with 
constant $\alpha$ and violation $\epsilon$, we have
\begin{eqnarray*}
&&\int_{G\times G}\mathbb{E}\left[d^2(\Phi(x, \xi),\Phi(y, \xi))\right]\ \hat{\gamma}(dx, dy)
\leq \\
&&\qquad\qquad\qquad\qquad 
\int_{G\times G} \paren{(1+\epsilon) d^2(x,y)    - \tfrac{1-\alpha}{\alpha} 
\mathbb{E}\left[\psi_c(x,y, \Phi(x, \xi), \Phi(y, \xi))\right]}
\ \hat{\gamma}(dx, dy),
  \end{eqnarray*}
  where $\hat{\gamma}$ is any coupling in $ C(\mu_1,\mu_2)$, not 
  necessarily optimal.  In particular, 
  since, for a random variable $X\sim \mu_1$, we have $\Phi(X, \xi)\sim \mu_1\mathcal{P}$,
  and for a random variable $Y\sim \mu_2$, 
  we have $\Phi(Y, \xi)\sim \mu_2\mathcal{P}$, 
  then, again for any optimal coupling $\gamma\in C_*(\mu_1,\mu_2)$,
   \begin{eqnarray*}
   W_2^2(\mu_1\mathcal{P},\mu_2\mathcal{P})&\leq& 
   \int_{G\times G}\mathbb{E}\left[d^2(\Phi(x, \xi),\Phi(y, \xi))\right]\ \gamma(dx, dy)
  \nonumber\\
  &\leq& 
\int_{G\times G} \paren{(1+\epsilon) d^2(x,y)    - \tfrac{1-\alpha}{\alpha} 
\mathbb{E}\left[\psi_c(x,y, \Phi(x, \xi), \Phi(y, \xi))\right]}
\ \gamma(dx, dy)\nonumber \\
&=&
(1+\epsilon) W_2^2(\mu_1,\mu_2) - \int_{G\times G}\tfrac{1-\alpha}{\alpha} 
\mathbb{E}\left[\psi_c(x,y, \Phi(x, \xi), \Phi(y, \xi))\right]\ \gamma(dx, dy).
  \end{eqnarray*}
  Since the measures $\mu_2, \mu_1\in\mathscr{P}_2(G)$ were arbitrary, as was the 
  optimal coupling $\gamma\in C_*(\mu_1,\mu_2)$, this completes the proof.  
\end{proof}

Contraction Markov operators have been studied in \cite{Oll09, JouOll10} using the 
parallel notion of the {\em coarse Ricci curvature} $\kappa(x, y)$ of the Markov operator $\mathcal{P}$ between two points 
$x$ and $y$:
\[
\kappa(x, y)\equiv 1 - \frac{W_1(\delta_x\Pcal, \delta_y\Pcal)}{d(x,y)}.
\]
Generalizing this definition to $W_p$ yields the coarse Ricci curvature with respect to $W_p$:
\[
\kappa_p(x, y)\equiv 1 - \frac{W_p^p(\delta_x\Pcal, \delta_y\Pcal)}{d(x,y)^p}.
\]
A few steps lead from this object for the 
Markov operator $\Pcal$ with  update function $\Phi(\cdot, \xi)=T_\xi$ 
and transition kernel defined by \eqref{eq:trans kernel}
to the violation $\epsilon$  in Proposition \ref{thm:Tafne in exp 2 pafne of P}.  
Indeed,  a formal adjustment of the proof of \cite[Proposition 2]{Oll09} establishes that the property
$\kappa_2(x,y)\geq \kappa \in \mathbb{R}$ for all $x, y\in G$ is equivalent to 
\[
W_2(\mu\Pcal, \mu'\Pcal)\leq \sqrt{1-\kappa}\,W_2(\mu, \mu')\quad\forall \mu, \mu'\in \mathscr{P}_2(G).
\]
When $\kappa>0$, i.e. when the coarse Ricci curvature is bounded below by a positive number, this 
characterizes contractivity of the Markov operator.  The negative of the violation in \eqref{eq:paneim}
is just a lower bound on the coarse Ricci curvature in $W_2$: 
$-\epsilon = \kappa\leq \kappa_2(x,y)$ for all $x,y\in G$.  
The consequences of Markov operators with Ricci curvature bounded below by a positive number 
have been extensively investigated.  Our approach extends this to {\em expansive} mappings, which 
allows one to treat our target application 
of electron density reconstructions from X-FEL experiments 
(see Section \ref{sec:incFeas}).  

In \cite{LukNguTam18} a general quantitative analysis for iterations of 
expansive
fixed point mappings is proposed consisting of two principle components: 
the constituent mappings are pointwise a$\alpha$-fne, and the transport 
discrepancy of the fixed point operator is 
{\em metrically subregular}.  
Recall that $\rho:[0,\infty) \to [0,\infty)$ is a \textit{gauge function} if 
$\rho$ is continuous, strictly increasing 
with $\rho(0)=0$, and $\lim_{t\to \infty}\rho(t)=\infty$. 
Recall also that, for any mapping $\mymap{\Psi}{A}{B}$,
the inverse mapping $\Psi^{-1}(y)\equiv \set{z\in A}{\Psi(z)=y}$, which 
clearly can be set-valued. 
\begin{definition}[metric subregularity]\label{d:(str)metric (sub)reg}
$~$ Let $(A, d_A)$ and $(B,d_B)$ be metric spaces and let $\mymap{\Psi}{A}{B}$.
The mapping $\Psi$ is called \emph{metrically subregular with respect to the metric $d_B$ 
for $y\in B$ relative 
to $\Lambda\subset A$ on $U\subset A$ with gauge $\rho$} whenever
\begin{equation}\label{e:metricregularity}
\inf_{z\in \Psi^{-1}(y)\cap \Lambda} d_A\paren{x, z}\leq \rho( d_B\paren{y, \Psi(x)})
\quad \forall x\in U\cap \Lambda.
\end{equation}
\end{definition}

Our definition is modelled after \cite{DonRoc14}, where the case where the 
gauge is just a linear function -- $\rho(t)=\kappa t$ -- is developed.  
In this case, metric subregularity
is one-sided Lipschitz continuity of the (set-valued) inverse mapping $\Psi^{-1}$.  
We will refer to the case when the gauge is linear to {\em linear metric subregularity}.
For connections of this notion to the concept of {\em transversality} in differential 
geometry and its use in variational analysis see \cite{Ioffe17}.
The main advantage of including the more general gauge function is 
to characterize sub-linear convergence rates of numerical 
methods.  We apply metric regularity to the Markov operator on $\mathscr{P}(G)$ 
with the Wasserstein metric.  In particular, 
the gauge of metric subregularity $\rho$  is constructed 
implicitly from another 
nonnegative function $\mymap{\theta}{[0,\infty)}{[0,\infty)}$ satisfying 
\begin{eqnarray}\label{eq:theta}
(i)~ \theta(0)=0; \quad (ii)~ 0<\theta(t)<t ~\forall t>0; 
\quad (iii)~\sum_{j=1}^\infty\theta^{(j)}(t)<\infty~\forall t\geq 0.
\end{eqnarray}
For a CAT(0) space the operative gauge of metric subregularity satisfies 
\begin{equation}\label{eq:gauge}
 \rho\paren{\paren{\frac{(1+\epsilon)t^2-\paren{\theta(t)}^2}{\tau}}^{1/2}}=
 t\quad\iff\quad
 \theta(t) = \paren{(1+\epsilon)t^2 - \tau\paren{\rho^{-1}(t)}^2}^{1/2}
\end{equation}
for $\tau>0$ fixed and $\theta$ satisfying \eqref{eq:theta}.  

In the case of linear metric subregularity on a CAT(0) space 
this becomes 
\[
\rho(t)=\kappa t\quad\iff\quad  
\theta(t)=\paren{(1+\epsilon)-\frac{\tau}{\kappa^2}}^{1/2}t\quad 
(\kappa\geq \sqrt{\tfrac{\tau}{(1+\epsilon)}}).
\]  
The condition $\kappa\geq \sqrt{\tfrac{\tau}{(1+\epsilon)}}$ is not a real restriction since, if 
\eqref{e:metricregularity} is satisfied for some $\kappa'>0$, then it is satisfied
for all $\kappa\geq \kappa'$. The conditions in \eqref{eq:theta} in this 
case simplify to $\theta(t)=\gamma t$ where 
\begin{equation}\label{eq:theta linear}
 0< \gamma\equiv 1+\epsilon-\frac{\tau}{\kappa^2}<1\quad\iff\quad 
\sqrt{\tfrac{\tau}{(1+\epsilon)}}\leq  \kappa\leq \sqrt{\tfrac{\tau}{\epsilon}}.
\end{equation}

Metric subregularity plays a central role in the implicit function paradigm for solution 
mappings \cite{CANO2, DonRoc14}.  
Linear metric subregularity was shown in \cite[Theorem 3.15]{HerLukStu19a} to be necessary 
 and sufficient for R-linear convergence in expectation of random function 
 iterations for consistent stochastic feasibility.  This result is a stochastic
 analog of the result \cite[Theorem 2]{LukTebTha18} in the deterministic setting.   
 
 We apply this to the Markov operator $\mathcal{P}$ on the metric space 
 $(\mathscr{P}_2(G), W_2)$ in the following manner.  Recall the transport discrepancy $\psi_c$ 
defined in \eqref{eq:delta}. 
 We construct the surrogate mapping 
$\mymap{\Psi}{\mathscr{P}(G)}{\mathbb{R}_+}\cup\{+\infty\}$ 
 defined by 
 \begin{equation}\label{eq:Psi}
\Psi(\mu)\equiv \inf_{\pi\in\inv\mathcal{P}}\inf_{\gamma\in C_*(\mu,\pi)}
\left(\int_{G\times G}
\mathbb{E}\left[\psi_c(x,y, T_\xi x, T_\xi y)\right]\ \gamma(dx, dy)\right)^{1/2}.
 \end{equation}
 We call this the {\em Markov transport discrepancy}.
 It is not guaranteed  that both $\inv\mathcal{P}$ and $C_*(\mu,\pi)$ are nonempty;  
 when at least one of these is empty, we define $\Psi(\mu)\equiv +\infty$. 
It is clear that $\Psi(\pi)=0$ for any 
$\pi \in \inv\Pcal$.  Whether $\Psi(\mu)=0$ only when $\mu\in \inv\Pcal$ is a property
of the {\em space} $(G, d)$.  Indeed, as noted in the discusson after 
Lemma \ref{lem:psi_f nonneg}, in CAT($\kappa$) spaces with $\kappa>0$
the transport discrepancy $\psi_c$ can be negative, and so 
by cancellation it could happen on such spaces that the Markov transport discrepancy
$\Psi(\mu)=0$ for $\mu\notin \inv\Pcal$. 
The regularity we require of $\Pcal$ is that 
the Markov transport discrepancy $\Psi$ takes the value $0$ at $\mu$ if and only if 
$\mu\in\inv\Pcal$, and is metrically subregular for $0$ relative to $\mathscr{P}_2(G)$ 
on $\mathscr{P}_2(G)$ 
defined in \eqref{eq:p-probabiliy measures}.  

Before moving to our main results, we put the more familiar contractive 
mappings into the present context.  
A survey of random function iterations for contractive mappings 
  in expectation can be found in \cite{Stenflo2012}.     An immediate
consequence of \cite[Theorem 1]{Stenflo2012} is the existence of a
unique invariant measure and linear convergence in the Wasserstein
metric from any initial distribution to the invariant measure.  See also 
Example \ref{ex:2linearSpaces_withInvMeas}.   Error estimates for 
Markov chain Monte Carlo methods under the assumption of positive Ricci curvature in $W_1$ 
(i.e. {\em negative} violation) are explored in \cite{JouOll10}.  Applications
to waiting queues, the Ornstein–Uhlenbeck process on $\Rn$ and Brownian
motion on positively curved manifolds, as well as demonstrations of how to verify the 
assumptions on the Ricci curvature are developed in \cite{Oll09}.  
The next result shows that update functions $\Phi$ that are contractions in expectation 
generate $\alpha$-fne Markov operators with metrically 
subregular Markov transport discrepancy.
\begin{thm}
\label{thm:contraInExpec}
   Let $(G,\|\cdot\|)$ be a Hilbert space, let 
   $\mymap{T_i}{G}{G}$ for $i\in I$ and let 
  $\mymap{\Phi}{G \times I }{G}$ be given by
  $\Phi(x,i)\equiv T_i(x)$.  
  Denote by  $\mathcal{P}$ the Markov
  operator with  update function $\Phi$ and transition kernel $p$ defined by 
  \eqref{eq:trans kernel}.  
  Suppose that $\Phi$ is a 
  \emph{contraction in expectation} with 
  constant $r<1$, i.e.\  
  $\mathbb{E}[\|\Phi(x,\xi) - \Phi(y,\xi)\|^2] \le r^2 \|x-y\|^2$ for all $x,y \in
  G$. 
  Suppose in addition  that  
  there exists $y \in G$ with
  $\mathbb{E}[\|\Phi(y,\xi) - y\|^2] < \infty$.  
Then the following hold. 
\begin{enumerate}[(i)]
 \item\label{thm:contraInExpec i} There exists a unique
  invariant measure $\pi \in \mathscr{P}_{2}(G)$ for $\mathcal{P}$ and 
  \begin{align*}
    W_2(\mu_0 \mathcal{P}^{n} , \pi) \le r^{n} W_2(\mu_0,\pi)
  \end{align*}
  for all $\mu_0 \in \mathscr{P}_{2}(G)$; that is, the sequence $(\mu_k)$ defined 
  by $\mu_{k+1}=\mu_k\mathcal{P}$ converges  to $\pi$ Q-linearly (geometrically) from any initial 
  measure $\mu_0\in \mathscr{P}_{2}(G)$.
\item\label{thm:contraInExpec ii}   $\Phi$ is $\alpha$-fne 
in expectation with constant $\alpha = (1+r)/2$, and the Markov operator $\mathcal{P}$ 
is $\alpha$-fne on $\mathscr{P}_{2}(G)$; that is,  $\mathcal{P}$
satisfies \eqref{eq:alphfne meas} 
with $\epsilon=0$ and constant $\alpha = (1+r)/2$ on $\mathscr{P}_{2}(G)$.
\item\label{thm:contraInExpec iii} If $\Psi$  defined by \eqref{eq:Psi} satisfies 
\begin{equation}\label{e:Hood}
\exists q>0:\quad \Psi(\mu) \geq qW_2(\mu\Pcal, \mu)\quad\forall \mu\in\mathscr{P}_{2}(G),
\end{equation}
then 
$\Psi$ is linearly metrically subregular for $0$ relative to 
$\mathscr{P}_{2}(G)$ on $\mathscr{P}_{2}(G)$ with gauge $\rho(t) = (q(1-r))^{-1} t$.
\end{enumerate}
\end{thm}
\begin{proof}
  Note that for any pair of distributions $\mu_1,\mu_2 \in
  \mathscr{P}_{2}(G)$ and an 
  optimal coupling $\gamma\in C_*(\mu_1,\mu_2)$ (possible by
  Lemma \ref{lemma:WassersteinMetric_prop}) it holds that 
  \begin{eqnarray*}
    W_2^2(\mu_1 \mathcal{P}, \mu_2 \mathcal{P}) &\le& 
    \int_{G\times G}\mathbb{E}[d^2(\Phi(x,\xi),\Phi(y,\xi))]\ \gamma(d x, d y) 
    \\
    &\le& r ^2\int_{G\times G}d^2(x,y)\ \gamma(d x, d y) = r^2 W_2^2(\mu_1,\mu_2),
  \end{eqnarray*}
where $\xi$ is independent of $\gamma$.
Moreover, $\mathcal{P}$ is a self-mapping on  ${\mathscr{P}_{2}(G)}$.  
To see this   
let  $\mu \in \mathscr{P}_{2}(G)$  independent of $\xi$
and let $y$ be a point in $G$ where $\mathbb{E}[\|\Phi(y,\xi) - y\|^2] < \infty$.   
Then by the triangle inequality and the contraction property
  \begin{eqnarray*}
    &&\int_G\mathbb{E}[\|\Phi(x,\xi)-y\|^2]\ \mu(dx) 
    \nonumber \\
    &&\qquad \le 4\paren{ \int_G\mathbb{E}[\|\Phi(x,\xi) - \Phi(y,\xi)\|^2]\ \mu(dx) +
   \mathbb{E}[\|\Phi(y,\xi) - y\|^2]}\\
    &&\qquad \le 4\paren{\int_G r^2 \|x-y\|^2\ \mu(dx) +
   \mathbb{E}[\|\Phi(y,\xi) - y\|^2]}<\infty.
  \end{eqnarray*}
Therefore $\mu \mathcal{P} \in \mathscr{P}_{2}(G)$.  Altogether, this 
 establishes  that $\mathcal{P}$ is a contraction on the separable complete 
metric space $(\mathscr{P}_{2}(G), W_2)$ and hence Banach's Fixed Point Theorem
  yields existence and uniqueness of $\inv\mathcal{P}$ and Q-linear convergence 
  of the fixed point sequence.  
  
  To see \eqref{thm:contraInExpec ii}, note that, by \eqref{eq:nice ineq},
  \begin{eqnarray}
  \mathbb{E}[\psi_2(x,y, T_\xi x, T_\xi y)]&=& 
  \mathbb{E}\left [\|(x - \Phi(x,\xi)) - (y - \Phi(y,\xi))\|^2\right]\nonumber\\
  &=&\|x-y\|^2+
  \mathbb{E}\left[\|\Phi(x,\xi)-\Phi(y,\xi)\|^2 - 2\langle x-y, \Phi(x,\xi)-\Phi(y,\xi)\rangle\right] 
  \nonumber\\
\label{eq:corona home}   &\leq& (1+ r)^2\|x-y\|^2,
  \end{eqnarray}
where the last inequality follows from the Cauchy-Schwarz inequality and 
the fact that $\Phi(\cdot,\xi)$ is a contraction in expectation.  Again using the 
contraction property and \eqref{eq:corona home} we have
\begin{eqnarray*}
 \mathbb{E}\left[\|\Phi(x,\xi)-\Phi(y,\xi)\|^2\right]&\leq& \|x-y\|^2 - (1-r^2)\|x-y\|^2 \\
 &\leq&\|x-y\|^2 - \tfrac{1-r^2}{(1+r)^2}\mathbb{E}[\psi_2(x,y, T_\xi x, T_\xi y)]. 
\end{eqnarray*}
The right hand side of this inequality is just the characterization 
\eqref{eq:paafne i.e.} of mappings that are $\alpha$-fne  in expectation 
with $\alpha = (1+r)/2$.
The rest of the statement follows from Proposition \ref{thm:Tafne in exp 2 pafne of P}.

\eqref{thm:contraInExpec iii}  The proof is modeled after the proof of \cite[Theorem 32]{BLL}. 
By the triangle inequality and part \eqref{thm:contraInExpec i} we have 
\begin{eqnarray}
W_2(\mu_{k+1},\mu_k) &\geq& W_2(\mu_k, \pi) - W_2(\mu_{k+1}, \pi) \nonumber\\
&\geq& (1-r)W_2(\mu_k, \pi) \quad  \forall k\in \Nbb.
\label{e:Robin}
\end{eqnarray}
On the other hand, \eqref{e:Hood} implies that $\Psi$ takes the value zero only at invariant 
measures 
so that by    
the uniqueness of invariant measures established in part \eqref{thm:contraInExpec i}
\[
\Psi^{-1}(0)\cap \mathscr{P}_{2}(G) = \inv \Pcal\cap \mathscr{P}_{2}(G) = \{\pi\}.
\]
Combining this with \eqref{e:Robin} and \eqref{e:Hood} then yields for all $k\in \Nbb$
\begin{eqnarray*}
|\Psi(\mu_k)-0| = \Psi(\mu_k)&\geq& qW_2(\mu_{k+1},\mu_k)\nonumber\\
&\geq& q(1-r)W_2(\mu_k, \Psi^{-1}(0)\cap \mathscr{P}_{2}(G)).  
 \end{eqnarray*}
In other words, 
 \begin{equation}\label{e:dumber}
(q(1 - r))^{-1}|\Psi(\mu_k)-0|\geq W_2(\mu_k, \Psi^{-1}(0)\cap \mathscr{P}_{2}(G)) 
\quad  \forall k\in \Nbb.
 \end{equation}
Since this holds for {\em any} sequence $(\mu_k)_{k\in\Nbb}$ initialized with 
$\mu_0\in \mathscr{P}_{2}(G)$,  
we conclude that $\Psi$ is metrically subregular for $0$ relative to $\mathscr{P}_{2}(G)$ with 
gauge $\rho(t)=(q(1 - r))^{-1}t$ on $\mathscr{P}_{2}(G)$, as claimed.  
\end{proof}

The simple example of a single Euclidean projector onto an affine subspace ($I=\{1\}$, and 
$T_1$ the orthogonal projection onto an affine subspace) shows that the statement of 
Theorem \ref{thm:contraInExpec} fails without the assumption of contractivity.

\subsection{Main Results}
\label{sec:mainres}

All of our main results concern Markov operators $\Pcal$ with
  update function $\Phi(x,i)=T_i(x)$ and transition kernel 
  $p$ given by \eqref{eq:trans kernel} for self mappings 
  $\mymap{T_i}{G}{G}$.
For any $\mu_0\in \mathscr{P}_2(G)$, we denote 
the distributions of the iterates of Algorithm \ref{algo:RFI} by 
$\mu_{k} =  \mu_0 \mathcal{P}^{k} = \mathcal{L}(X_{k})$, 
and we denote  $d_{W_2}\paren{\mu_k,\inv\mathcal{P}}\equiv 
\inf_{\pi'\in\inv\mathcal{P}}W_2\paren{\mu_{k},\, \pi'}$.

In most of our main results, it will be assumed that $\inv\mathcal{P}\neq\emptyset$.  
The existence theory is already well developed and is surveyed in Section \ref{sec:existence} 
below.  We show how existence is guaranteed when, for instance, the image is 
compact for some non-negligible
collection of operators $T_i$ (Proposition \ref{cor:finite_selection_existence}) or when 
the expectation of the 
random variables $X_k$ is finite (Proposition \ref{thm:ex_inRn}).

The main convergence result for nonexpansive mappings follows from a fundamental result 
of Worm \cite[Theorem 7.3.13]{WormPhd2010}. 
\begin{thm}[convergence of Ces\`{a}ro average in
  $\mathbb{R}^{n}$] \label{cor:cesaroConvergenceRn}
  Let   $\mymap{T_{i}}{\mathbb{R}^{n}}{\mathbb{R}^{n}}$ be nonexpansive 
  ($i \in I$) and assume
  $\inv\mathcal{P}\neq\emptyset$.  Let $\mu \in
  \mathscr{P}(\mathbb{R}^{n})$ and $\nu_{k} = \tfrac{1}{k}
  \sum_{j=1}^{k} \mu \mathcal{P}^{j}$, then this sequence converges in the 
  Prokhorov-L\`evy metric to 
  an invariant probability measure for $\mathcal{P}$, i.e.\ 
  for each $x \in \Supp \mu \subset \mathbb{R}^n$ the limit of the sequence
  $(\nu_{k}^{x})$, denoted $\pi^{x}$, exists and, more generally, 
  $\nu_{k}
  \to \pi^{\mu}$, an invariant measure, where 
  \begin{equation}\label{eq:rep Cesaro limit}
    \pi^{\mu} = \int_{\Supp \mu} \pi^{x} \mu(\dd{x}).
  \end{equation}
\end{thm}
\noindent When the mappings are $\alpha$-fne, we obtain the 
following stronger result. It is worth pointing interested readers to an
analogous metric space result of \cite[Theorem 27]{BLL} in which 
it is shown that, on $p$-uniformly convex spaces, sequences 
generated by fixed point iterations of compositions of pointwise 
$\alpha$-fne mappings $T_i$  
converge in a weak sense whenever
$\cap_i \Fix T_i$ is nonempty.  When the composition is boundedly compact, then 
the fixed point iterations converge strongly to a fixed point.
\begin{thm}[convergence for $\alpha$-firmly nonexpansive
  mappings on $\mathbb{R}^{n}$]%
  \label{thm:a-firm convergence Rn} Let
  $\mymap{T_{i}}{\mathbb{R}^{n}}{\mathbb{R}^{n}}$ be
  $\alpha$-fne with constant 
  $\alpha_{i}\le \alpha<1$ ($i\in I$).
  Assume $\inv\mathcal{P}\neq\emptyset$. For any initial distribution $\mu_0 \in
  \mathscr{P}(\mathbb{R}^{n})$ the distributions $\mu_k$ of the iterates
   generated by Algorithm \ref{algo:RFI} converge in the 
  Prokhorov-L\`evy metric to an invariant probability measure for
  $\mathcal{P}$.
\end{thm}
The proof of this result is very different than the strategy applied to the 
metric space result of \cite[Theorem 27]{BLL}.

Trading the weaker assumption that the mappings $T_i$ are only  
a$\alpha$-fne against the assumption of metric subregularity 
of the Markov transport discrepancy $\Psi$ yields rates of convergence 
by lifting the results of \cite[Corollary 2.3]{LukNguTam18} to the space of
probability measures.
\begin{thm}[convergence rates]\label{t:msr convergence} 
  Let $(H,d)$ be a separable Hadamard space and let $G\subset H$ be compact.  
  Let $\mymap{T_i}{G}{G}$ be continuous for all $i\in I$ and define
  $\mymap{\Psi}{\mathscr{P}_2(G)}{\mathbb{R}_+}\cup\{+\infty\}$ by
\eqref{eq:Psi}.
  Assume furthermore:
  \begin{enumerate}[(a)]
  \item \label{t:msr convergence c}  there is at least one 
  $\pi \in\inv\mathcal{P}\cap \mathscr{P}_2(G)$  where $\mathcal{P}$ is the Markov operator with
  update function $\Phi$ given by \eqref{eq:trans kernel};
  \item\label{t:msr convergence a} $\Phi$ is
    a$\alpha$-fne in expectation with constant
    $\alpha\in (0,1)$ and violation at most $\epsilon$; and
  \item\label{t:msr convergence b} 
    $\Psi$ takes the value $0$ only at points $\pi\in \inv\mathcal{P}$ and 
    is metrically subregular (in the $W_2$ metric) for $0$ relative to $\mathscr{P}_2(G)$ on
    $\mathscr{P}_2(G)$ with gauge $\rho$ given by \eqref{eq:gauge} where 
    $\tau=(1-\alpha)/\alpha$.
\end{enumerate}
Then for any $\mu_0\in \mathscr{P}_2(G)$ 
the distributions $\mu_k$ of the iterates of Algorithm \ref{algo:RFI} 
converge in the $W_2$ metric to 
some $\pi^{\mu_0}\in\inv\mathcal{P}\cap\mathscr{P}_2(G)$ with rate characterized by 
\begin{equation}\label{eq:gauge convergence}
d_{W_2}\paren{\mu_{k+1},\inv\mathcal{P}}
\leq \theta\paren{d_{W_2}\paren{\mu_k,\inv\mathcal{P}}} 
\quad \forall k \in \mathbb{N},
\end{equation}%
where $\theta$ given implicitly by \eqref{eq:gauge} satisfies \eqref{eq:theta}. 
\end{thm}
\begin{rem}\label{r:G compact}
The compactness assumption on $G$ can be dropped if 
$(H,d)$ is a Euclidean space.  
\end{rem}

An immediate corollary of this theorem is the following 
specialization to linear convergence. 
\begin{cor}[linear convergence rates]
\label{t:msr convergence - linear} 
Under the same assumptions as Theorem \ref{t:msr convergence},
if $\Psi$ is linearly metrically subregular (i.e. 
with gauge $\rho(t)=\kappa\cdot t$) for $0$  
with constant $\kappa$
satisfying $\sqrt{\frac{1-\alpha}{\alpha(1+\epsilon)}}\leq \kappa\
<\sqrt{\frac{1-\alpha}{\alpha\epsilon}}$, 
then the sequence of iterates $(\mu_k)$ converges 
R-linearly to 
some $\pi^{\mu_0}\in\inv\mathcal{P}\cap\mathscr{P}_2(G)$:  
\begin{equation}%
d_{W_2}\paren{\mu_{k+1},\inv\mathcal{P}}
\leq c\, d_{W_2}\paren{\mu_k,\inv\mathcal{P}}
\end{equation}%
where $c\equiv \sqrt{1+\epsilon -\paren{\tfrac{1-\alpha}{\kappa^2\alpha}}}<1$
and $\kappa\geq \kappa'$ satisfies $\kappa\geq\sqrt{(1-\alpha)/\alpha(1+\epsilon)}$.
If $\inv\mathcal{P}$ consists of a single point then convergence is 
Q-linear.  
\end{cor}

\section{Background Theory and Proofs}\label{sec:theory} 
In this section we prepare tools to prove the main results from
Section \ref{sec:mainres}. We start by establishing convergence results on the
supports of ergodic measures on a general Polish space $G$, and then,
for global convergence analysis, we restrict ourselves to
$\mathbb{R}^{n}$.  We begin with existence of invariant measures. 
We then analyze properties of (and convergence of the RFI on) subsets of 
$G$, called ergodic sets. Then we turn our attention to the
global convergence analysis.

\subsection{Existence of Invariant Measures}
\label{sec:existence}



A sequence of probability measures $(\nu_{k})$ 
is called \emph{tight}
if for any $\epsilon >0$ there exists a compact $K \subset G$ with
$\nu_{k} (K) > 1-\epsilon$ for all $k \in \mathbb{N}$.  By Prokhorov's theorem
(see, for instance, \cite{Billingsley}),  a sequence $(\nu_{k}) \subset
  \mathscr{P}(G)$, for $G$ a Polish space, 
  is tight if and only if
  $(\nu_{k})$ is compact in $\mathscr{P}(G)$, i.e.\ any
  subsequence of $(\nu_{k})$ has a  subsequence that converges in distribution
  (see, for instance \cite{Billingsley}). 
%
%

A basic building block is the existence of invariant measures proved by 
{Lasota} and T. {Szarek} \cite[Proposition 3.1]{LasSza06}.  Based on this, 
we show how existence  can be verified easily.  But first, we show how
to obtain existence constructively. 
\begin{prop}[construction of an invariant
  measure] \label{thm:construction_inv_meas} Let $\mu \in
  \mathscr{P}(G)$ and $\mathcal{P}$ be a Feller Markov operator.  Let
  $(\mu \mathcal{P}^{k})_{k \in \mathbb{N}}$ be a tight sequence of
  probability measures on a Polish space $G$, and let 
  $\nu_{k} = \tfrac{1}{k}\sum_{j=1}^{k} \mu\mathcal{P}^{j}$.  
  Any cluster point of the sequence $(\nu_{k})_{k\in\Nbb}$ is an invariant measure for
  $\mathcal{P}$.
\end{prop}
\begin{proof}
  Our proof follows \cite[Theorem 1.10]{Hairer2016}. 
  Tightness of the sequence
  $(\mu \mathcal{P}^{k})$ implies tightness of the sequence $(\nu_{k})$ and
  therefore  by Prokhorov's Theorem there exists a convergent subsequence $(\nu_{k_{j}})$ with
  limit $\pi \in \mathscr{P}(G)$. By the Feller
  property of $\mathcal{P}$ one has for any continuous and bounded
  $\mymap{f}{G}{\mathbb{R}}$ that also $\mathcal{P}f$ is continuous
  and bounded, and hence
  \begin{align*}
    \abs{(\pi \mathcal{P})f - \pi  f} &= \abs{\pi (\mathcal{P}f) - \pi
      f}  \\ &= \lim_{j} \abs{
      \nu_{k_{j}}(\mathcal{P} f)- \nu_{k_{j}} f} \\ &= \lim_{j}
    \frac{1}{k_{j}} \abs{ \mu\mathcal{P}^{k_{j}+1}f -
      \mu\mathcal{P}f} \\ &\le \lim_{j}
    \frac{2\norm{f}_{\infty}}{k_{j}} \\ &=0.
  \end{align*}
  Now, $\pi f = (\pi \mathcal{P})f$ for all $f \in C_{b}(G)$ implies
  that $\pi = \pi \mathcal{P}$.
\end{proof}

When a Feller Markov chain converges in distribution (i.e.\
$\mu\mathcal{P}^{k} \to \pi$), it does so to an invariant measure
(since $\mu\mathcal{P}^{k+1} \to \pi \mathcal{P}$).  A Markov operator
need not possess a unique invariant probability measure or any
invariant measure at all.  Indeed, consider the normed space 
$(\mathbb{R}^{n}, \norm{\cdot})$ for the
case that $T_{i}=P_{i}$, $i \in I$ is a projector onto a nonempty
closed and convex set $C_{i} \subset \mathbb{R}^{n}$. A sufficient
condition for the deterministic Alternating Projections Method to
converge in the inconsistent case to a limit cycle for convex sets is
that one of the sets is compact 
(this is an easy consequence of \cite[Theorem 4]{CheneyGoldstein59}).  
Translating this into the present 
setting,  a sufficient condition for the existence of an invariant
measure for $\mathcal{P}$ is the existence of a compact set 
$K\subset \mathbb{R}^{n}$ and  $\epsilon>0$ such that 
$p(x,K)\ge \epsilon$ 
for all $x \in\mathbb{R}^{n}$. This holds, for instance, when 
there are only finitely many sets with one of them, say
$C_{\ibar}$, compact and $\mathbb{P}(\xi = \ibar) = \epsilon$, since
$p(x,C_{\ibar}) = \mathbb{P}(P_{\xi}x \in C_{\ibar}) \ge \mathbb{P}(P_{\xi}x \in 
C_{\ibar}, \xi= \ibar) = \mathbb{P}(\xi = \ibar) = \epsilon$ for all $ x \in
\mathbb{R}^{n}$.  More generally, we have the following result.
\begin{prop}[existence of invariant measures for finite collections of
  continuous mappings]\label{cor:finite_selection_existence} Let
  $G$ be a Polish space and let $\mymap{T_{i}}{G}{G}$ be
  continuous for $i \in I$, where $I$ is a finite index set. If for one
  index $i \in I$ it holds that $\mathbb{P}(\xi = i)>0$ and $T_{i}(G)
  \subset K$, where $K \subset G$ is compact, then there exists an
  invariant measure for $\mathcal{P}$.
\end{prop}
\begin{proof}
  We have from $T_{i}(G)\subset K$ that $    \mathbb{P}(T_{\xi}x \in K ) \ge
  \mathbb{P}(\xi = i)$ and hence for the sequence $(X_{k})$ generated
  by Algorithm \ref{algo:RFI} for an arbitrary initial probability measure
  \begin{align*}
    \mathbb{P}(X_{k+1} \in K) = \mathbb{E}[\cpr{T_{\xi_{k}}X_{k}
      \in K}{X_{k}}] \ge \mathbb{P}(\xi = i) \qquad \forall k \in \mathbb{N}.
  \end{align*}
  The assertion follows now immediately from 
  \cite[Proposition 3.1]{LasSza06}
  since $\mathbb{P}(\xi = i)>0$ and
  $\mathcal{P}$ is Feller by continuity of $T_{j}$ for all $j \in I$.
\end{proof}

Next we mention an existence result which requires that the RFI sequence
$(X_{k})$ possess a uniformly bounded expectation.

\begin{prop}[existence in $\mathbb{R}^{n}$,
  RFI] \label{thm:ex_inRn} Let
  $\mymap{T_{i}}{\mathbb{R}^{n}}{\mathbb{R}^{n}}$ ($i \in I$) be
  continuous. Let $(X_{k})$ be the RFI sequence (generated by
  Algorithm \ref{algo:RFI}) for some initial measure. Suppose that for all $k
  \in \mathbb{N}$ it holds that $\mathbb{E}\left[ \norm{X_{k}}\right]
  \le M$ for some $M \ge 0$.  Then there exists an invariant measure
  for the RFI Markov operator $\mathcal{P}$ given by \eqref{eq:trans
    kernel}.
\end{prop}
\begin{proof}
  For any $\epsilon>M$ Markov's inequality implies that
  \begin{align*}
    \mathbb{P}(\norm{X_{k}} \ge \epsilon) \le
    \frac{\mathbb{E}\left[\norm{X_{k}}\right]}{\epsilon} \le
    \frac{M}{\epsilon} <1
  \end{align*}
  Hence,
  \begin{align*}
    \limsup_{k \to \infty} \mathbb{P}(\norm{X_{k}} \le \epsilon) \ge
    \limsup_{k \to \infty} \mathbb{P}(\norm{X_{k}} < \epsilon) \ge
    1-\frac{M}{\epsilon} >0.
  \end{align*}
  Existence of an invariant measure then follows from \cite[Proposition 3.1]{LasSza06}
  since closed balls in
  $\mathbb{R}^{n}$ with finite radius are compact, $\mathbb{P}(X_{k}
  \in \cdot) = \mu\mathcal{P}^{k}$ and continuity of $T_i$ yields the
  Feller property for $\mathcal{P}$.
\end{proof}

To conclude this section, we also establish that, for the setting considered 
here, the set of invariant measures is closed.
\begin{lemma}\label{lemma:invMeasuresClosed}
 Let $G$ be a Polish space and let $\mathcal{P}$ be a Feller
  Markov operator, which is in particular the case under Assumption \ref{ass:1}, 
if $T_{i}$ is continuous for all $i\in
  I$. Then the set of associated invariant measures $\inv\mathcal{P}$ is 
closed with respect to the topology of convergence in distribution.
\end{lemma}

\begin{proof}
Let $(\pi_n)_{n \in  \mathbb{N}}$ be a sequence of measures in $\inv\mathcal{P}$
that converges in distribution  to $\pi.$ Thus we have 
$\pi_n \mathcal{P}= \pi_n$ for all 
$n \in  \mathbb{N}$ and we need to establish this also for the limiting 
measure $\pi$.
For this it suffices to show 
that for all $f \in C_{b}(G)$,
  \begin{align*}
  \int f(x)  \pi(\dd{x}) =  \int f(x)  \pi \mathcal{P}(\dd{x}) 
   =   \int f(x) \int  p(y,\dd{x})  \pi(\dd{y}) 
   =  \int  \mathcal{P}f(y)   \pi(\dd{y}).
  \end{align*}
This means 
that for  $f \in C_{b}(G)$ we also have that $\mathcal{P}f \in C_{b}(G)$. 
(By Proposition \ref{thm:Feller} we have that 
$\mathcal{P}$ is Feller if Assumption \ref{ass:1} holds and $T_{i}$ is 
continuous for all $i\in I$.)
Hence, for those $f$ we have that 
$  \int f(x)  \pi_n(\dd{x}) \rightarrow \int f(x)  \pi(\dd{x})$ 
as well as 
$  \int \mathcal{P}f(x)  \pi_n(\dd{x}) \rightarrow 
\int \mathcal{P}f(x)  \pi(\dd{x})$
which establishes the claim.
\end{proof}

\subsection{Ergodic theory of general Markov Operators}
\label{sec:suppCvg}

We understand here under ergodic theory the analysis of the properties
of the RFI Markov chain when it is initialized by a distribution 
in the support of any ergodic
measure for the Markov operator $\mathcal{P}$. The convergence
properties for these points can be much stronger than the convergence
properties of Markov chains initialized by measures with support
outside the support of the ergodic measures.

The consistent stochastic feasibility problem was analyzed in
\cite{HerLukStu19a} without the need of the notion of convergence of
measures since, as shown in Proposition \ref{thm:Hermer}, for consistent stochastic feasibility 
convergence of sequences defined by \eqref{eq:X_RFI} is almost sure, if they converge at all.  
More general convergence of measures is more challenging
as the next example illustrates.
\begin{example}[nonexpansive mappings, negative
  result]\label{eg:convergence of average}
  For non-expansive mappings in general, one cannot expect that the
  sequence $(\mathcal{L}(X_{k}))_{k\in\mathbb{N}}$ converges to an invariant
  probability measure.  Consider the nonexpansive operator $T:=
  T_{1}x:= -x$ on $\mathbb{R}$ and set, in the RFI setup, $\xi= 1$ and
  $I=\{1\}$.  Then $X_{2k} = x$ and $X_{2k+1} = -x$ for all $k \in
  \mathbb{N}$, if $X_{0}\sim \delta_{x}$. This implies for $x\neq 0$
  that $(\mathcal{L}(X_{k}))$ does not converge to the invariant 
  distribution $\pi_{x} = \tfrac{1}{2}(\delta_{x} + \delta_{-x})$
  (depending on $x$), since $\mathbb{P}(X_{2k} \in B) = \delta_{x}(B)$
  and $\mathbb{P}(X_{2k+1} \in B) = \delta_{-x}(B)$ for $B \in
  \mathcal{B}(\mathbb{R})$. Nevertheless the Ces\`{a}ro average
  $\nu_{k}:= \tfrac{1}{k} \sum_{j=1}^{k} \mathbb{P}^{X_{j}}$ converges
  to $\pi_{x}$.
\end{example}

As Example \ref{eg:convergence of average} shows, meaningful notions of
ergodic convergence are possible (in our case, convergence of the
Ces\`{a}ro average) even when convergence in distribution can not be
expected. We start by collecting several general results for Markov
chains on Polish spaces.  In the next section we restrict ourselves to 
equicontinuous and Feller Markov operators.

An invariant probability measure $\pi$ of $\mathcal{P}$ is called
\emph{ergodic}, if any $p$\emph{-invariant set}, i.e.\ $A \in
\mathcal{B}(G)$ with $p(x,A)=1$ for all $x \in A$, has $\pi$-measure
$0$ or $1$.  Two measures $\pi_{1},\pi_{2}$ are called {\em mutually
  singular} when there is $A \in \mathcal{B}(G)$ with $\pi_{1}(A^{c})
= \pi_{2}(A)=0$.
The following decomposition theorem on Polish spaces is key to our
development.  For more detail see, for instance, \cite{Walters82}.  

\begin{prop}%
\label{thm:decomp_ergodic_stat_measures}
  Denote by $\mathcal{I}$ the set of all invariant probability
  measures for $\mathcal{P}$ and by $\mathcal{E} \subset \mathcal{I}$
  the set of all those that are ergodic. Then, $\mathcal{I}$ is convex
  and $\mathcal{E}$ is precisely the set of its extremal
  points. Furthermore, for every invariant measure $\pi \in
  \mathcal{I}$, there exists a probability measure $q_{\pi}$ on
  $\mathcal{E}$ such that
  \begin{align*}
    \pi(A) = \int_{\mathcal{E}} \nu(A) q_{\pi}(\dd{\nu}).
  \end{align*}
  In other words, every invariant measure is a convex combination of
  ergodic invariant measures. Finally, any two distinct elements of
  $\mathcal{E}$ are mutually singular.
\end{prop}
\begin{rem}
  If there exists only one invariant probability measure of
  $\mathcal{P}$, we know by Proposition \ref{thm:decomp_ergodic_stat_measures}
  that it is ergodic. If there exist more invariant probability
  measures, then there exist uncountably many invariant and at least
  two ergodic probability measures.
\end{rem}


\begin{prop}\label{cor:Birkhoff_conditional}
  Let $\pi$ be an ergodic invariant probability measure for
  $\mathcal{P}$,  let $(G,\mathcal{G})$ be a measurable space, and let 
  $\mymap{f}{G}{\mathbb{R}}$ be measurable, bounded and satisfy $\pi \abs{f}^{p} < \infty$
  for $p \in [1,\infty]$.  
  Then
  \begin{align*}
    \nu_{k}^{x}f := \frac{1}{k} \sum_{j=1}^{k} p^{j}(x,f) \to \pi f
    \qquad \text{as } k \to \infty \quad\text{ for } \pi\text{-a.e. }
    x \in G,
  \end{align*}
  where $p^{j}(x,f) := \delta_{x}\mathcal{P}^{j}f = \cex{f(X_{j})}{X_{0}=x}$
  for the sequence $(X_{k})$ generated by Algorithm \ref{algo:RFI} with
  $X_{0} \sim \pi$.
\end{prop}
\begin{proof}
 This is a direct consequence of Birkhoff's ergodic theorem, \cite[Theorem 9.6]{kallenberg1997}.
\end{proof}

For fixed $x$ in Proposition \ref{cor:Birkhoff_conditional}, we want the
assertion to be true for all $f \in C_{b}(G)$. This issue is addressed 
in the next section by restricting our attention to equicontinuous Markov operators.
The results above do not require any explicit structure on the
mappings $T_i$ that generate the transition kernel $p$ and hence the
Markov operator $\mathcal{P}$, however the assumption that the initial 
random variable $X_0$ has the same distribution as the invariant measure $\pi$
is very strong.  For Markov operators generated from discontinuous
mappings $T_i$, the support of an invariant measure 
may not be invariant under
$T_{\xi}$.  To see this, let
\begin{align*}
  Tx :=
  \begin{cases}
    x, & x \in \mathbb{R}\setminus \mathbb{Q}\\
    -1, & x \in \mathbb{Q}
  \end{cases}
\end{align*}
The transition kernel is then $p(x,A) = \1_{A}(T x)$ for $x \in
\mathbb{R}$ and $A \in \mathcal{B}(\mathbb{R})$. Let $\mu$ be the
uniform distribution on $[0,1]$, then, since $\lambda$-a.s.\ $T=\Id$
(where $\lambda$ is the Lebesgue measure on $\mathbb{R}$), we have
that $\mu \mathcal{P}^{k} = \mu$ for all $k \in
\mathbb{N}$. Consequently, $\pi = \mu$ is invariant and
$\Supp \pi=[0,1]$, but $T ([0,1]) = \{-1\} \cup [0,1]\cap
(\mathbb{R}\setminus \mathbb{Q})$, which is not contained in $[0,1]$.

The next result shows, however, that invariance of the the support of invariant 
measures under {\em continuous mappings} $T_i$ is guaranteed. 
\begin{lemma}[invariance of the support of invariant
  measures]\label{lemma:support_invariant_distr}
  Let $G$ be a Polish space and let $\mymap{T_{i}}{G}{G}$ be
  continuous for all $i \in I$. For any
  invariant probability measure $\pi \in \mathscr{P}(G)$ of
  $\mathcal{P}$ it holds that $T_{\xi} S_{\pi} \subset S_{\pi}$ a.s.
  where $S_{\pi} := \Supp \pi$.
\end{lemma}
\begin{proof}
By Fubini's Theorem,  for any $A \in \mathcal{B}(G)$  it holds that
  \begin{align*}
    \pi(A) = \int_{S_{\pi}} p(x,A) \pi(\dd{x}) &= \int_{\Omega}
    \int_{S_{\pi}} \1_{A}(T_{\xi}x) \pi(\dd{x}) \dd{\mathbb{P}} \\
    &= \int_{\Omega} \int_{S_{\pi}} \1_{T_{\xi(\omega)}^{-1} A}(x)
    \pi(\dd{x}) \mathbb{P}(\dd{\omega}) \\ &= \mathbb{E}\left[
      \pi(T_{\xi}^{-1}A \cap S_{\pi})\right] = \mathbb{E}\left[ 
\pi(T_{\xi}^{-1}A)\right].
  \end{align*}
  From $1= \pi(S_{\pi})= \mathbb{E}\left[ \pi(T_{\xi}^{-1}
    S_{\pi})\right]$ and $\pi(\cdot)\le 1$, it follows that
  $\pi(T_{\xi}^{-1}S_{\pi}) = 1$ a.s. 
  
  Note that $T_{i}^{-1}S_{\pi}$
  is closed for all $i \in I$ due to continuity of $T_{i}$ and
  closedness of $S_{\pi}$.  We show that $S_{\pi} \subset
  T_{\xi}^{-1}S_{\pi}$ a.s. which then yields the claim.  
  To see this, 
  let $S\subset G$ 
  be any closed set with $\pi(A\cap S) = \pi(A)$ for all $A\in\mathcal{B}(G)$, and
  let $x\in S_{\pi}$.  Then $\pi(\mathbb{B}(x,\epsilon) \cap S) >0$ 
  for all $\epsilon >0$, 
  i.e.\ $\mathbb{B}(x,\epsilon)\cap S \neq\emptyset$ for all $\epsilon>0$. 
  Now consider $x_{k} \in\mathbb{B}(x,\epsilon_{k}) \cap S$, 
  where $\epsilon_{k} \to 0$ as $k \to \infty$. Then since $S$ is closed, 
  $x_{k} \to x \in S$, from which we conclude that $S_{\pi}\subset S$.  
  Specifically, let $S=T^{-1}S_{\pi}$ and note that $T^{-1}S_{\pi} = D \setminus G$ 
  for some $D$ with
    $\pi(D)=0$. For any $A \in \mathcal{B}(G)$ it holds that $\pi(A \cap S) = 
    \pi(A) - \pi(A \cap D) = \pi(A)$. From the argument above, we conclude that 
    $S_{\pi}\subset S = T^{-1}S_{\pi}$ as claimed.  
\end{proof}

The above result means that, if the random variable $X_{k}$ enters $S_{\pi}$
for some $k$, then it will stay in $S_{\pi}$ forever. This can be
interpreted as a mode of convergence, i.e.\ convergence to the set
$S_{\pi}$, which is closed under application of $T_{\xi}$ a.s.
Equality $T_{\xi} S_{\pi} = S_{\pi}$ a.s.\ cannot be expected in
general.  For example,  let $I=\{1,2\}$, $G = \mathbb{R}$ and $T_{1}x= -1$,
$T_{2}x = 1$, $x\in\mathbb{R}$ and $\mathbb{P}(\xi = 1) = 0.5 =
\mathbb{P}(\xi = 2)$, then $\pi = \tfrac{1}{2}(\delta_{-1} +
\delta_{1})$ and $S_{\pi} = \{-1,1\}$. So $T_{1} S_{\pi} = \{-1\}$ and
$T_{2} S_{\pi} = \{1\}$.

%
%
%

\subsection{Ergodic convergence theory for equicontinuous Markov operators}
\label{sec:ergodicNonexp}

As shown by Szarek \cite{Szarek2006} and Worm \cite{WormPhd2010}, 
equicontinuity of Markov operators and
their generalizations give a nice structure to the set of ergodic
measures. We collect some results here which will be used heavily in the 
subsequent analysis.

\begin{definition}[equicontinuity]
  A Markov operator is called equicontinuous, if
  $(\mathcal{P}^{k}f)_{k \in \mathbb{N}}$ is equicontinuous for all
  bounded and Lipschitz continuous $\mymap{f}{G}{\mathbb{R}}$.
\end{definition}

In the following we consider the union of supports of all ergodic
measures defined by
\begin{align}\label{d:S}
  S := \bigcup_{\pi \in \mathcal{E}} \Supp \pi,
\end{align}
where $\mathcal{E} \subset \inv \mathcal{P}$ denotes the set of
ergodic measures.
%

\begin{prop}[tightness of
  $(\delta_{s}\mathcal{P}^{k})$] \label{thm:tighnessOfIterates} Let
  $G$ be a Polish space.  Let $\mathcal{P}$ be equicontinuous. Suppose there
  exists an invariant measure for $\mathcal{P}$. Then
the sequence  $(\delta_{s}\mathcal{P}^{k})_{k\in\Nbb}$ 
is tight for all $s \in S$ defined by
  \eqref{d:S}.
\end{prop}
\begin{proof}
  In the proof of \cite[Proposition 2.1]{Szarek2006} it is  shown
  that, under the assumption that $\mathcal{P}$ is equicontinuous and
  \begin{align}\label{eq:Kurtz}
    (\exists s,x\in G) \quad \limsup_{k \to \infty}
    \nu_{k}^{x}(\mathbb{B}(s,\epsilon)) > 0 \quad \forall\epsilon>0,
  \end{align}
  then the sequence $(\delta_{s}\mathcal{P}^{k})$ is tight.  
  It remains to demonstrate \eqref{eq:Kurtz}.  To see this, let $f =
  \1_{\mathbb{B}(s,\epsilon)}$ for some $s \in S_{\pi}$, where $\pi
  \in \mathcal{E}$ and $\epsilon>0$ in
  Proposition \ref{cor:Birkhoff_conditional}.  Then for $\pi\text{-a.e. } x \in
  G$ and $\nu_{k}^{x} := \tfrac{1}{k} \sum_{j=1}^{k} \delta_{x}
  \mathcal{P}^{j}$ we have
  \begin{align*}
    \limsup_{k \to \infty} \nu_{k}^{x}(\mathbb{B}(s,\epsilon)) =
    \lim_{k} \nu_{k}^{x}(\mathbb{B}(s,\epsilon)) =
    \pi(\mathbb{B}(s,\epsilon)) > 0.
  \end{align*}
This completes the proof.
\end{proof}
\begin{rem}[tightness of $(\nu_{k}^{s})$]\label{rem:tightnessMean}
  Note that the sequence $(\nu_{k}^{s})$ is tight for $s \in S$, since by
  Proposition \ref{thm:tighnessOfIterates}, for all $\epsilon>0$, there is a
  compact subset $K \subset G$ such that $p^{k}(s,K) > 1-\epsilon$ for
  all $k \in \mathbb{N}$, and hence also $\nu_{k}^{s}(K)>1-\epsilon$
  for all $k \in \mathbb{N}$.
\end{rem}

The next result due to Worm 
(Theorems 5.4.11 and 7.3.13 of \cite{WormPhd2010}))
concerns Ces\'{a}ro averages for
equicontinuous Markov operators.
\begin{prop}[convergence of Ces\`{a}ro averages 
{\cite{WormPhd2010}}]\label{thm:wormCesaro}
  Let $\mathcal{P}$ be Feller and equicontinuous, let $G$ be a Polish space
  and let $\mu \in \mathscr{P}(G)$.  Then the sequence 
  $(\nu_{k}^{\mu})$ is tight 
  ($\nu_{k}^{\mu} := \frac{1}{k} \sum_{j=1}^{k}\mu \mathcal{P}^{j}$)
  if and only if $(\nu_{k}^{\mu})$ converges to a 
  $\pi^{\mu} \in \inv \mathcal{P}$. In this case
  \begin{align*}
    \pi^{\mu} = \int_{\Supp \mu} \pi^{x} \mu(\dd{x}),
  \end{align*}
  where for each $x \in \Supp \mu \subset G$ there exists the limit of
  $(\nu_{n}^{x})$ and it is denoted by the invariant measure
  $\pi^{x}$.
\end{prop}

For the case the initial measure $\mu$ is supported in $\bigcup_{\pi
  \in \inv \mathcal{P}} \Supp {\pi}$, we have the following.
\begin{prop}[ergodic decomposition]\label{prop:ergodic_decomp}
  Let $G$ be a Polish space and let $\mathcal{P}$ be Feller and
  equicontinuous. Then
  \begin{align*}
    S = \bigcup_{\pi \in \inv \mathcal{P}} \Supp {\pi},
  \end{align*}
  where $S$ is defined in \eqref{d:S}. Moreover $S$ is closed, and
  for any $\mu \in \mathscr{P}(S)$ it holds that $\nu_{k}^{\mu} \to
  \pi^{\mu}$ as $k\to\infty$ with
  \begin{align*}
    \pi^{\mu} = \int_{S} \pi^{x} \mu(\dd{x}),
  \end{align*}
  where $\pi^{x}$ is the unique ergodic measure with $x \in
  \Supp {\pi^{x}}$.
\end{prop}
\begin{proof}
  This is a consequence of \cite[Theorem 7.3.4]{WormPhd2010} and
  \cite[Theorem 5.4.11]{WormPhd2010}, Remark \ref{rem:tightnessMean} and
  Proposition \ref{thm:wormCesaro}.
\end{proof}
  Proposition \ref{prop:ergodic_decomp} only establishes convergence of the
  Markov chain when it is initialized with a measure in the support of
  an invariant measure; moreover, it is only the average of the
  distributions of the iterates that converges.
\begin{rem}[convergence of $(\nu_{k}^{s})$ on Polish
  spaces] \label{cor:weakCvgMetricSpace} Let $\pi$
  be an ergodic invariant probability measure for $\mathcal{P}$. Then
  for all $s \in \Supp \pi$ the sequence $\nu_{k}^{s} \to
  \pi$ as $k \to \infty$, where $\nu_{k}^{s} = \tfrac{1}{k}
  \sum_{j=1}^{k} p^{j}(s,\cdot)$.
\end{rem}

By Proposition \ref{thm:decomp_ergodic_stat_measures} any invariant measure can
be decomposed into a convex combination of ergodic invariant measures;
in particular two ergodic measures $\pi_{1},\pi_{2}$ are mutually
singular. Note that it still could be the case that $\Supp \pi_{1}
\cap \Supp \pi_{2} \neq \emptyset$. But Remark \ref{cor:weakCvgMetricSpace}
above establishes that for a Feller and equicontinuous\ Markov operator
$\mathcal{P}$ this is not possible, so the singularity of ergodic
measures extends to their support. This leads to the following 
corollary. 
\begin{cor}\label{th:singular measure char}
  Under the assumptions of Proposition \ref{prop:ergodic_decomp} for two ergodic
  measures $\pi, \tilde \pi$, the intersection 
  $\paren{\Supp {\pi}} \cap \paren{\Supp {\tilde \pi}} = \emptyset$ 
  if and only if $\pi \neq \tilde \pi$.
\end{cor}
  
The next technical lemma implies that every point in the support of an
ergodic measure is reached infinitely often starting from any other
point in this support.
\begin{lemma}[positive transition probability for ergodic
  measures] \label{lemma:pos_transitionKernel_ergodic} Let $G$ be
  a Polish space and let $\mymap{T_{i}}{G}{G}$ be nonexpansive, $i \in
  I$. Let $\pi$ be an ergodic invariant probability measure for
  $\mathcal{P}$. Then for any $s,\tilde s \in \Supp {\pi}$ it holds that
  \begin{align*}
    \forall \epsilon>0 \, \exists \delta>0, \,\exists (k_{j})_{j \in \mathbb{N}} 
\subset
    \mathbb{N} \,:\, p^{k_{j}}(s, \mathbb{B}(\tilde s, \epsilon)) \ge
    \delta \quad \forall j \in \mathbb{N}.
  \end{align*}
\end{lemma}
\begin{proof}
  Given $\tilde s \in \Supp {\pi}$ and $\epsilon>0$, find a continuous and
  bounded function $\mymap{f=f_{\tilde s, \epsilon} }{G}{ [0,1]}$ with
  the property that $f = 1$ on $\mathbb{B}(\tilde s,
  \tfrac{\epsilon}{2})$ and $f =0$ outside $\mathbb{B}(\tilde s,
  \epsilon)$. For $s \in \Supp {\pi}$ let $X_{0} \sim \delta_{s}$ and
  $(X_{k})$ generated by Algorithm \ref{algo:RFI}. By
  Remark \ref{cor:weakCvgMetricSpace} the sequence $(\nu_{k})$ converges to $\pi$
  as $k\to\infty$, where
  $\nu_{k}:=\tfrac{1}{k} \sum_{j=1}^{k} p^{j}(s,\cdot)$. So in
  particular $\nu_{k}f \to \pi f \ge \pi(\mathbb{B}(\tilde s,
  \tfrac{\epsilon}{2})) >0$ as $k \to \infty$. Hence, for $k$ large
  enough there is $\delta >0$ with
  \begin{align*}
    \nu_{k}f = \tfrac{1}{k} \sum_{j=1}^{k} p^{j}(s,f) \ge \delta.
  \end{align*}
  Now, we can extract a sequence $(k_{j}) \subset \mathbb{N}$ with
  $p^{k_{j}}(s,f) \ge \delta$, $j\in \mathbb{N}$ and hence
  \begin{align*}
    p^{k_{j}}(s,\mathbb{B}(\tilde s, \epsilon)) \ge p^{k_{j}}(s,f) \ge
    \delta >0.&\qedhere
  \end{align*}
\end{proof}

\subsection{Ergodic theory for nonexpansive mappings}
\label{sec:nexp}

We now specialize to the case that the family of mappings
$\{T_{i}\}_{i \in I}$ are nonexpansive operators. 
\begin{lemma}\label{lemma:e.c.}
  Let $G$ be a Polish space. Let $\mymap{T_{i}}{G}{G}$ be
  nonexpansive, $i \in I$ and let $\Pcal $ denote the
  Markov operator that is induced by the transition
  kernel in \eqref{eq:trans kernel}. 
  \begin{enumerate}[(i)]
  \item $\mathcal{P}$ is Feller.
  \item $\mathcal{P}$ is equicontinuous.
  \end{enumerate}
\end{lemma}
\begin{proof}
  \begin{enumerate}[(i)]
  \item The mapping $T_{i}$ for $i \in I$ is 1-Lipschitz continuous,
    so in particular it is continuous. Proposition \ref{thm:Feller} yields the assertion.
  \item Let $\epsilon>0$ and $x,y \in G$ with $d(x,y)<
    \epsilon/\norm{f}_{\text{Lip}}$, then, using Jensen's inequality,
    Lipschitz continuity of $f$ and nonexpansivity of $T_{i}$, we get
    \begin{align*}
      \abs{\delta_{x}\mathcal{P}^{k}f-\delta_{y}\mathcal{P}^{k}f} &=
      \abs{\mathbb{E}[f(X_{k}^{x})] - \mathbb{E}[f(X_{k}^{y})]}\\ &\le
      \mathbb{E}[\abs{f(X_{k}^{x}) - f(X_{k}^{y})}] \\ &\le
      \norm{f}_{\text{Lip}} \mathbb{E}[d(X_{k}^{x},X_{k}^{y})] \\ &\le
      \norm{f}_{\text{Lip}} \mathbb{E}[d(x,y)] < \epsilon
    \end{align*}
    for all $k \in \mathbb{N}$. \qedhere
  \end{enumerate}
\end{proof}

A very helpful fact used later on is that the distance between the
supports of two ergodic measures is attained; moreover, any point in
the support of the one ergodic measure has a nearest neighbor in the
support of the other ergodic measure. 
\begin{lemma}[distance of supports is
  attained]\label{lemma:dist_supports_attained}
  Let $G$ be a Polish space and $\mymap{T_{i}}{G}{G}$ be
  nonexpansive, $i \in I$. Suppose $\pi,\tilde \pi$ are ergodic
  probability measures for $\mathcal{P}$. Denote the support of a 
  measure $\pi$ by $S_{\pi}\equiv \Supp {\pi}$.
  Then for all $s\in S_{\pi}$
  there exists $\tilde s \in S_{\tilde\pi}$ with $d(s,\tilde s) =
  \dist(s,S_{\tilde\pi}) = \dist(S_{\pi}, S_{\tilde\pi})$.
\end{lemma}
\begin{proof}
  First we show, that $\dist(S_{\pi}, S_{\tilde \pi}) = \dist(s,
  S_{\tilde \pi})$ for all $s \in S_{\pi}$. Therefore, recall the
  notation $X_{k}^{x} = T_{\xi_{k-1}}\cdots T_{\xi_{0}} x$ and note
  that by nonexpansivity of $T_{i}$, $i \in I$ and
  Lemma \ref{lemma:support_invariant_distr} it holds a.s.\ that
  \begin{align*}
    \dist(X_{k+1}^{x},S_{\pi}) \le \dist(X_{k+1}^{x},T_{\xi_{k}}
    S_{\pi}) = \inf_{s \in S_{\pi}} d(T_{\xi_{k}}X_{k}^{x},
    T_{\xi_{k}}s) \le \dist(X_{k}^{x},S_{\pi})
  \end{align*}
  for all $x \in G$, $\pi \in \inv \mathcal{P}$ and $k \in
  \mathbb{N}$.  Suppose now there would exist an $\hat s \in S_{\pi}$
  with $\dist(\hat s, S_{\tilde\pi}) < \dist(s, S_{\tilde\pi})$. Then
  by Lemma \ref{lemma:pos_transitionKernel_ergodic} for all $\epsilon >0$
  there is a $k \in \mathbb{N}$ with $\mathbb{P}( X_{k}^{\hat s} \in
  \mathbb{B}(s,\epsilon)) > 0$ and hence
  \begin{align*}
    \dist(s, S_{\tilde\pi}) \le d(s,X_{k}^{\hat s}) +
    \dist(X_{k}^{\hat s}, S_{\tilde\pi}) \le \epsilon + \dist(\hat s,
    S_{\tilde\pi})
  \end{align*}
  with positive probability for all $\epsilon >0$, which is a
  contradiction. So, it holds that $\dist(\hat s, S_{\tilde\pi}) = \dist(s,
  S_{\tilde\pi})$ for all $s, \hat s \in S_{\pi}$.

  For $s \in S_{\pi}$ let $(\tilde s_{m}) \subset S_{\tilde \pi}$ be a
  minimizing sequence for $\dist(s,S_{\tilde \pi})$, i.e.\ $\lim_{m}
  d(s,\tilde s_{m}) = \dist(s,S_{\tilde \pi})$.  Now define a
  probability measure $\gamma_{k}^{m}$ on $G\times G$ via
  \begin{align*}
    \gamma_{k}^{m} f := \mathbb{E}\left[ \frac{1}{k} \sum_{j=1}^{k}
      f(X_{j}^{s}, X_{j}^{\tilde s_{m}}) \right]
  \end{align*}
  for measurable $\mymap{f}{G\times G}{\mathbb{R}}$. Then
  $\gamma_{k}^{m} \in C(\nu_{k}^{s}, \nu_{k}^{\tilde s_{m}})$
  where $C(\nu_{k}^{s}, \nu_{k}^{\tilde s_{m}})$ is 
the set of all couplings for $\nu_{k}^{s}$ and $\nu_{k}^{\tilde s_{m}}$ 
(see \eqref{eq:couplingsDef}).
Also,  by
  Lemma \ref{lemma:weakCVG_productSpace} and Remark \ref{cor:weakCvgMetricSpace}
  the sequence $(\gamma_{k}^{m})_{k\in \Nbb}$ is tight for fixed 
  $m \in \mathbb{N}$ and
  there exists a cluster point $\gamma^{m} \in C(\pi, \tilde \pi)$. The
  sequence $(\gamma^{m}) \subset C(\pi, \tilde \pi)$ is again tight by
  Lemma \ref{lemma:weakCVG_productSpace}.  Thus for any cluster point
  $\gamma \in C(\pi,\tilde \pi)$ and the bounded and continuous
  function $(x,y) \mapsto f^{M}(x,y)=\min( M, d(x,y) )$ this yields
  \begin{align*}
    \gamma_{k}^{m}d = \gamma_{k}^{m} f^{M} \searrow \gamma^{m} f^{M}
    \qquad \text{as } k \to \infty
  \end{align*}
  for all $M \ge d(s,\tilde s_{m})$, $m\in \mathbb{N}$. Since by the
  Monotone Convergence Theorem $\gamma^{m}f^{M} \nearrow \gamma^{m}d$
  as $m \to \infty$, it follows that $\gamma^{m}f^{M} = \gamma^{m} d$ for
  all $M \ge d(s,\tilde s_{m})$. The same argument holds for $M \ge
  d(s,\tilde s_{1})$ and a subsequence $(\gamma^{m_{j}})$ with limit
  $\gamma$ such that $\gamma d = \gamma f^{M}$.  Hence,
  \begin{align*}
    \gamma d = \gamma f^{M} = \lim_{j} \gamma^{m_{j}} f^{M} = \lim_{j}
    \gamma^{m_{j}} d \le \lim_{j} d(s,\tilde s_{m_{j}}) =
    \dist(s,S_{\tilde \pi}).
  \end{align*}
  In particular for $\gamma$-a.e.\ $(x,y) \in S_{\pi} \times S_{\tilde
    \pi}$ it holds that $d(x,y)= \dist(S_{\pi}, S_{\tilde \pi})$, because
  $d(x,y)\ge \dist(S_{\pi}, S_{\tilde \pi})$ on $S_{\pi} \times
  S_{\tilde \pi}$. Taking the closure of these $(x,y)$ in $G\times G$,
  we see that for any $s \in S_{\pi}$ there is $\tilde s \in S_{\tilde
    \pi}$ with $d(s,\tilde s) = \dist(S_{\pi}, S_{\tilde \pi})$ by
  Lemma \ref{lemma:couplings}.
\end{proof}

\subsection{Convergence for nonexpansive mappings in $\mathbb{R}^{n}$}
\label{sec:setCvgNE}
By Proposition \ref{thm:wormCesaro} tightness of a sequence of Ces\`{a}ro averages 
is equivalent to convergence of said sequence.  So our focus is on tightness in the 
Euclidean space setting.

\begin{lemma}[tightness of $(\mu\mathcal{P}^{k})$ in
  $\mathbb{R}^{n}$] \label{lemma:tightnessOfNu_n} On the
  Euclidean space $(\mathbb{R}^{n}, \norm{\cdot})$ let
  $\mymap{T_{i}}{\mathbb{R}^{n}}{\mathbb{R}^{n}}$ be nonexpansive 
  for all 
  $i \in I$, and let $\inv
  \mathcal{P} \neq \emptyset$ for the corresponding Markov operator.
  The sequence $(\mu\mathcal{P}^{k})_{k\in\Nbb}$ is tight for any $\mu \in
  \mathscr{P}(\mathbb{R}^{n})$.
\end{lemma}
\begin{proof}
  First, let $\mu=\delta_{x}$ for $x \in \mathbb{R}^{n}$. We know that 
  the sequence  $(\delta_{s} \mathcal{P}^{k})$ is tight for $s \in S$ by
  Proposition \ref{thm:tighnessOfIterates}. So for $\epsilon>0$ there is a
  compact $K \subset \mathbb{R}^{n}$ with $p^{k}(s,K)\ge 1-\epsilon$
  for all $k \in \mathbb{N}$. Recall the definition of $X_{k}^{x}$ in
  \eqref{eq:X_RFI}.  Since a.s.\ $\norm{X_{k}^{x}-X_{k}^{s}} \le \norm{x-s}$,
  we have that $p^{k}(x,\cb(K,\norm{x-s})) = \mathbb{P}( X_{k}^{x} \in
  \cb(K,\norm{x-s})) \ge p^{k}(s,K)\ge 1-\epsilon$ for all $k \in
  \mathbb{N}$. Hence $(\delta_{x} \mathcal{P}^{k})$ is tight.

  Now consider  the initial random variable $X_0\sim \mu$ for any 
  $\mu \in \mathscr{P}(\mathbb{R}^{n})$. For given $\epsilon >0$
  there is a compact $K_{\epsilon}^{\mu} \subset \mathbb{R}^{n}$ with
  $\mu(K_{\epsilon}^{\mu}) > 1-\epsilon$. From the special case
  established above, there exists a compact $K_{\epsilon} \subset
  \mathbb{R}^{n}$ with $p^{k}(0, K_{\epsilon}) >1-\epsilon$ for all $k
  \in \mathbb{N}$. Let $M>0$ such that $K_{\epsilon}^{\mu} \subset
  \cb(0,M)$ and let $x \in \cb(0,M)$. We have that
  $p^{k}(x,\cb(K_{\epsilon},M)) > 1-\epsilon$ for all $x \in
  \cb(0,M)$, since $\norm{X_{k}^{x} - X_{k}^{X_0}} \le \norm{x} \le
  M$. Hence $\mu\mathcal{P}^{k}(\cb(K_{\epsilon},M)) >
  (1-\epsilon)^{2}$, which implies tightness of the sequence
  $(\mu\mathcal{P}^{k})$.
\end{proof}
\begin{rem}[tightness of $(\nu_{k}^{\mu})$ in
  $\mathbb{R}^{n}$] \label{rem:tightnessNu_nInRn} The tightness of
 the sequence $(\nu_{k}^{\mu})$ for any $\mu \in \mathscr{P}(\mathbb{R}^{n})$
  follows immediately from tightness of $(\mu\mathcal{P}^{k})$ as in
  Remark \ref{rem:tightnessMean}.
\end{rem}

We are now in a position to prove the first main result.\\
\textbf{ Proof of Theorem \ref{cor:cesaroConvergenceRn}}.  
 By Lemma \ref{lemma:e.c.} the Markov operator $\mathcal{P}$ is Feller and 
 equicontinuous.  By Lemma \ref{lemma:tightnessOfNu_n}  the sequence 
 $\left(\mu\mathcal{P}^k\right)$ is tight, and so the sequence of 
 of Ces\'aro averages $(\nu_k^\mu)$ is also tight 
 (see Remark \ref{rem:tightnessNu_nInRn}).  
Hence by Proposition \ref{thm:wormCesaro} $\nu_k^\mu\to \pi^\mu$
with $\pi^\mu$ given by \eqref{eq:rep Cesaro limit}.\hfill$\Box$

\subsection{More properties of the RFI for nonexpansive mappings}
\label{sec:furtherprops}

This section is devoted to the preparation of some tools used in
Section \ref{sec:cvgAverages} to prove convergence of the distributions of
the iterates of the RFI. When the Markov chain is initialized with a
point not supported in $S$, i.e.\ when $\Supp \mu \setminus S \neq
\emptyset$, the convergence results on general Polish spaces are much
weaker than for the ergodic case in the previous section. One 
problem is that the sequences $(\nu_{k}^{x})_{k \in \mathbb{N}}$ for 
$x \in G \setminus S$ need not be tight anymore.  The right-shift operator 
$\mathcal{R}$ on
$l^{2}$, for example, with the initial distribution $\delta_{e_{1}}$,
generates the sequence $\mathcal{R}^{k} e_{1} = e_{k}$, $k=1,2,\dots$.  
Examples of spaces on which
we can always guarantee tightness are, of course, 
Euclidean spaces as seen in the
previous section, 
and compact metric spaces --
since then $(\mathscr{P}(G), d_{P})$ is compact.

For the case that the sequence of Ces\`{a}ro
averages does not necessarily converge, we have the following result.
\begin{lemma}[convergence of with  nonexpansive
  mappings] \label{thm:invMeas_nonexpans_map} Let $(G,d)$ be a 
  separable complete metric space and 
  let $\mymap{T_{i}}{G}{G}$ be nonexpansive for all $i \in I$. 
  Suppose $\inv \mathcal{P} \neq \emptyset$.  
  Let $X_{0} \sim \mu \in
  \mathscr{P}(G)$ and let $(X_{k})$ be the sequence generated by
  Algorithm \ref{algo:RFI}.  Denote the support of any measure $\mu$ by 
  $S_\mu$, and denote $\nu_{k} := \tfrac{1}{k}\sum_{j=1}^{k} \mu \mathcal{P}^{j}$.
  \begin{enumerate}[(i)]
  \item\label{item:nonExpProp_1}  
    \[\forall \pi \in \inv \mathcal{P}, \quad \dist(X_{k+1},S_{\pi}) \le \dist(X_{k},S_{\pi})~ a.s.
    \quad \forall k\in \mathbb{N}.\]
  \item\label{item:nonExpProp_2top} If the sequence $(\nu_{k})$ has a cluster point $\pi\in \inv\mathcal{P}$, then,
  \begin{enumerate}[(a)]
  \item\label{item:nonExpProp_2}$\dist(X_{k},S_{\pi}) \to 0$ a.s.\ as
    $k\to\infty $;
  \item\label{item:nonExpProp_3} all cluster points of the sequence $(\nu_{k})$ have the
    same support;
  \item \label{item:nonExpProp_4} cluster points of the sequence $(\mu
    \mathcal{P}^{k})$ have support in $S_{\pi}$ (if they exist).
  \end{enumerate}
\end{enumerate}
\end{lemma}

\begin{proof}
\eqref{item:nonExpProp_1}. By Lemma \ref{lemma:support_invariant_distr}, the sets 
  on which $T_{\xi_{k}} S_{\pi}$ is not a subset of
    $S_{\pi}$ are $\mathbb{P}$-null sets and their union
    is also a $\mathbb{P}$-null set.
    This yields
    \begin{align*}
(\forall s\in S_{\pi})\quad       
\dist(X_{k+1}, S_{\pi}) \le d(X_{k+1}, T_{\xi_{k}} s) =
      d(T_{\xi_{k}} X_{k},T_{\xi_{k}}s) \le d(X_{k}, s)
            \quad\text{a.s.},
    \end{align*}
    and hence
    \begin{align*}
      \dist(X_{k+1},S_{\pi}) \le \dist(X_{k},S_{\pi}) 
      \quad\text{a.s.}
    \end{align*}
    
\eqref{item:nonExpProp_2} Define the function $f = \min(M,\dist(\cdot,S_{\pi}))$ 
    for some $M>0$.  Since  this is bounded and continuous, we have for a subsequence
    $(\nu_{k_{j}})$  converging to $\pi$, that
    $\nu_{k_{j}} f = \tfrac{1}{k_{j}} \sum_{n=1}^{k_{j}}
    \mu\mathcal{P}^{n}f \to \pi f = 0$ as $j \to \infty$. Now
    part \eqref{item:nonExpProp_1} and the identity
    \[ \mu\mathcal{P}^{n+1}f =
    \mathbb{E}[\min(M,\dist(X_{n+1},S_{\pi}))] \le
    \mathbb{E}[\min(M,\dist(X_{n},S_{\pi}))] = \mu\mathcal{P}^{n}f
    \]
    yield $\mu\mathcal{P}^{n}f =
    \mathbb{E}[\min(M,\dist(X_{n},S_{\pi}))] \to 0$ as $n \to \infty$.
    Again by part \eqref{item:nonExpProp_1}
    \[
    Y:=\lim_{n\to\infty} \min(M,\dist(X_{n},S_{\pi}))
    \]
    exists and is nonnegative; so by Lebesgue's dominated convergence
    theorem it follows that $Y=0$ a.s., since otherwise $\mathbb{E}[Y]
    > 0 = \lim_{n\to\infty} \mu \mathcal{P}^{n} f$ would yield a
    contradiction.

\eqref{item:nonExpProp_3} Let $\pi_{1},\pi_{2}$ be two cluster points of $(\nu_{k})$ with
    support $S_{1}, S_{2}$ respectively, then these probability
    measures are invariant for $\mathcal{P}$ by
    Proposition \ref{thm:construction_inv_meas}. By Corollary \ref{th:singular measure
      char} the intersection $S_{1}\cap S_{2}$ must be nonempty. 
      Suppose now w.l.o.g. $\exists y
    \in S_{1}\setminus S_{2}$.  Then there is an $\epsilon>0$ with
    $\mathbb{B}(y,2\epsilon) \cap S_{2} = \emptyset$. Let
    $\mymap{f}{G}{[0,1]}$ be a continuous function that takes the
    value $1$ on $\mathbb{B}(y,\tfrac{\epsilon}{2})$ and $0$ outside
    of $\mathbb{B}(y,\epsilon)$. Then $\pi_{1}f >0$ and
    $\pi_{2}f=0$. But there are two subsequences of $(\nu_{k})$ with
    $\nu_{k_{j}} f \to \pi_{1}f$ and $\nu_{\tilde k_{j}}f \to
    \pi_{2}f$ as $j \to \infty$. For the former sequence we have, for
    $j$ large enough,
    \begin{align*}
      \exists \delta >0: \frac{1}{k_{j}} \sum_{n=1}^{k_{j}} \mu
      \mathcal{P}^{n} f \ge \delta >0.
    \end{align*}
    So, one can from this extract a sequence $(m_{k})_{k\in\Nbb} \subset
    \mathbb{N}$ with $\mu \mathcal{P}^{m_{k}} f \ge \delta$, $k\in
    \mathbb{N}$.  Note that $\mathbb{P}(X_{m_{k}} \in
    \mathbb{B}(y,\epsilon)) \ge \mu \mathcal{P}^{m_{k}}f \ge \delta
    >0$. This implies $\dist(X_{m_{k}},S_{2}) \ge \epsilon$ with
    $\mathbb{P}\ge \delta$ and hence $\mathbb{E}[\dist(X_{m_{k}},
    S_{2})] \ge \delta\epsilon$, in contradiction to
    \eqref{item:nonExpProp_2}. So there cannot be such $y$ which
    yields $S_{1} = S_{2}$, as claimed.

\eqref{item:nonExpProp_4} 
    Let $\nu$ be a cluster point of the sequence $(\mu\mathcal{P}^{k})$, which
    is assumed to exist, and assume there is $s \in \Supp \nu \setminus
    S_{\pi}$ and $\epsilon>0$ such that $\dist(s,S_{\pi})> 2
    \epsilon$. Let $\mymap{f}{G}{[0,1]}$ be a continuous function,
    that takes the value $1$ on $\mathbb{B}(s,\tfrac{\epsilon}{2})$
    and $0$ outside of $\mathbb{B}(s,\epsilon)$. With
    \eqref{item:nonExpProp_2} we find, that
    \begin{align*}
      0<\nu f = \lim_{j} \mathbb{P}^{X_{k_{j}}}f \le \lim_{j}
      \mathbb{P}(X_{k_{j}} \in \mathbb{B}(s,\epsilon)) = 0.
    \end{align*}
    Were $\mathbb{P}(X_{k_{j}} \in \mathbb{B}(s,\epsilon)) \ge
    \delta>0$ for $j$ large enough, then this would imply that
    \[
    \mathbb{E}[\dist(X_{k_{j}}, S_{\pi})] \ge \delta \epsilon
    \]
    for $j$ large enough, which is a contradiction.  We conclude that
    there is no such $s$, which completes the proof. \qedhere
\end{proof}

We now prepare some tools to handle convergence of the distributions
of the iterates of the RFI for $\alpha$-fne mappings in
Section \ref{sec:cvgAverages}. 
We restrict ourselves to Polish spaces with finite dimensional metric
(see Definition \ref{def:finiteDimMetric}) in order to apply a differentiation
theorem.  We begin with the next technical fact. 
\begin{lemma}[characterization of balls in
  $(\mathcal{E},d_{P})$]\label{lem:char_ballsEW}
  Let $(G,d)$ be a separable complete metric space and 
  $\mymap{T_{i}}{G}{G}$ be 
  nonexpansive,
  $i \in I$. Let $\mathcal{E}$ denote the the (convex) set of ergodic
measures associated to the Markov operator $\Pcal$ , which is induced 
by the family of mappings $\{T_i\}$  $i∈I$ and the marginal probability 
law of the random variables $\xi_k$.
  Let $\pi, \tilde\pi \in \mathcal{E}$ and denote the support of 
  the measure $\pi$ by $S_\pi$ (and similarly for $\tilde\pi$). Then
  \begin{align*}
    \tilde \pi \in \cb(\pi,\epsilon) \qquad \Longleftrightarrow \qquad
    S_{\tilde \pi} \subset \cb(S_{\pi},\epsilon)
  \end{align*}
  for $\epsilon \in (0,1)$, where $\cb(\pi,\epsilon)$ is the closed $\epsilon$-ball 
  with respect to the Prokhorov-L\`evy metric $d_{P}$.
\end{lemma}
\begin{proof}
  By Lemma \ref{lemma:dist_supports_attained} there exist $s \in S_{\pi}$
  and $\tilde s \in S_{\tilde \pi}$ such that $d(s,\tilde s ) =
  \dist(S_{\pi}, S_{\tilde \pi})$. First note that, if $\pi\neq \tilde
  \pi$, then $S_{\pi} \cap S_{\tilde\pi} = \emptyset$ by
  Corollary \ref{th:singular measure char}, and hence $d(s,\tilde s )
  = \dist(S_{\pi}, S_{\tilde \pi}) > 0$.\\
  Recall the notation $X_{k}^{x} := T_{\xi_{k-1}} \cdots T_{\xi_{0}}
  x$ for $x \in G$ and note that by
  Lemma \ref{lemma:couplings}\eqref{item:couplings1} and
  Lemma \ref{lemma:support_invariant_distr}, $\Supp \mathcal{L}(X_{k}^{s})
  \subset S_{\pi}$ and $\Supp \mathcal{L}(X_{k}^{\tilde s}) \subset
  S_{\tilde \pi}$. So it holds that $d(X_{k}^{s},X_{k}^{\tilde s}) \ge
  \dist(S_{\pi},S_{\tilde \pi})$ a.s.\ for all $k\in \mathbb{N}$.
  Since $T_{i}$ ($i \in I$) is nonexpansive  we have that
  $d(X_{k}^{s},X_{k}^{\tilde s}) \leq d(s,\tilde s)$ a.s.\ for all $k
  \in \mathbb{N}$. So, both inequalities together imply the equality
  \begin{align}\label{eq:const_dist1}
    d(X_{k}^{s},X_{k}^{\tilde s}) = d(s,\tilde s) \qquad \text{a.s. }
    \forall k \in \mathbb{N}.
  \end{align}
  Now, letting $c:= \min(1,d(s,\tilde s))$, we show that
  $d_{P}(\pi,\tilde \pi) = c$, where $d_{P}$ denotes the
  Prokhorov-L\`evy metric (see
  Lemma \ref{lemma:prokhorovDist_properties}). Indeed,  take $(X,Y) \in
  C(\mathcal{L}(X_{k}^{s}), \mathcal{L}(X_{k}^{\tilde s}))$. Again, by
  Lemma \ref{lemma:couplings}\eqref{item:couplings1} and
  Lemma \ref{lemma:support_invariant_distr} $\Supp \mathcal{L}(X) \subset
  S_{\pi}$ and $\Supp \mathcal{L}(Y) \subset S_{\tilde \pi}$ and hence
  $d(X,Y) \ge \dist(S_{\pi}, S_{\tilde \pi}) = d(s,\tilde s)$ a.s. We
  have, thus
  \begin{align*}
    \mathbb{P}(d(X,Y)> c - \delta) \ge \mathbb{P}(d(X,Y) > d(s,\tilde
    s) - \delta) = 1 \qquad \forall \delta >0,
  \end{align*}
  which implies $d_{P}(\mathcal{L}(X_{k}^{s}),
  \mathcal{L}(X_{k}^{\tilde s})) \ge c$ by
  Lemma \ref{lemma:prokhorovDist_properties}\eqref{item:prokLeviRep}. In
  particular, for $c=1$ it follows that $d_{P}(\mathcal{L}(X_{k}^{s}),
  \mathcal{L}(X_{k}^{\tilde s})) =1$, since $d_{P}$ is bounded by
  $1$. Now, let $c<1$, i.e.\ $c=d(s,\tilde s) < 1$. We have by
  \eqref{eq:const_dist1}
  \begin{align*}
    \inf_{(X,Y) \in C(\mathcal{L}(X_{k}^{s}),
      \mathcal{L}(X_{k}^{\tilde s}))} \mathbb{P}\left(d(X,Y) >
      c\right) \le \mathbb{P}\left(d(X_{k}^{s}, X_{k}^{\tilde s}) >
      c\right) = 0 \le c.
  \end{align*}
  Altogether we find that $d_{P}(\mathcal{L}(X_{k}^{s}),
  \mathcal{L}(X_{k}^{\tilde s})) = c$, again by
  Lemma \ref{lemma:prokhorovDist_properties}\eqref{item:prokLeviRep}. Since
  also $\Supp \nu_{k}^{s} \subset S_{\pi}$ and $\Supp \nu_{k}^{\tilde
    s} \subset S_{\tilde \pi}$, where $\nu_{k}^{x} = \tfrac{1}{k}
  \sum_{j=1}^{k}\mathcal{L}(X_{j}^{x})$ for any $x\in G$, it follows
  that
  \begin{align}\label{eq:charBalls}
    c \le d_{P}(\nu_{k}^{s},\nu_{k}^{\tilde s}) \le \max_{j=1,\ldots,k}
    d_{P}(\mathcal{L}(X_{j}^{s}), \mathcal{L}(X_{j}^{\tilde s})) =c
  \end{align}
  by Lemma \ref{lemma:prokhorovDist_properties}\eqref{item:prokhorov5}.
  Now taking the limit $k \to \infty$ of \eqref{eq:charBalls} and using
  Remark \ref{cor:weakCvgMetricSpace}, it follows that $d_{P}(\pi,\tilde\pi) = c$.
This proves the assertion.
\end{proof}

\begin{definition}[Besicovitch family]
  A family $\mathcal{B}$ of closed balls $B = \cb(x_{B},\epsilon_{B})$
  with $x_{B} \in G$ and $\epsilon_{B} > 0$ on the metric space
  $(G,d)$ is called a \emph{Besicovitch family} of balls if
  \begin{enumerate}[(i)]
  \item for every $B \in \mathcal{B}$ one has $x_{B} \not \in B' \in
    \mathcal{B}$ for all $B' \neq B$, and
  \item $\bigcap_{B \in \mathcal{B}} B \neq \emptyset$.
  \end{enumerate}
\end{definition}

\begin{definition}[$\sigma$-finite dimensional 
metric]\label{def:finiteDimMetric}
  Let $(G, d)$ be a metric space. We say that $d$ is \emph{finite
    dimensional} on a subset $D \subset G$ if there exist constants $K
  \ge 1$ and $0 < r \le \infty $ such that $\Card \mathcal{B} \le K$
  for every Besicovitch family $\mathcal{B}$ of balls in $(G, d)$
  centered on $D$ with radius $< r$.  We say that $d$ is
  \emph{$\sigma$-finite dimensional} if $G$ can be written as a
  countable union of subsets on which $d$ is finite dimensional.
\end{definition}

\begin{prop}[differentiation theorem, \cite{Preiss81}]
\label{th:diff thm}
  Let $(G, d)$ be a  separable complete metric space. For every locally
  finite Borel regular measure $\lambda$ over $(G, d)$, it holds that
  \begin{equation}\label{eq:diff thm}
    \lim_{r \to 0} \frac{1}{\lambda(\cb(x,r))} \int_{\cb(x,r)} f(y)
    \lambda(\dd{y}) = f(x) \qquad \text{for } \lambda\text{-a.e. } x
    \in G,~\forall f \in L_{\mathrm{loc}}^{1}(G,\lambda)
  \end{equation}
  if and only if $d$
  is $\sigma$-finite dimensional.
\end{prop}

\begin{prop}[Besicovitch covering property in 
$\mathcal{E}$]\label{prop:weak_besicovitch_Rn}
  Let $(G,d)$ be separable complete metric space with finite dimensional metric 
$d$ and
  let $\mymap{T_{i}}{G}{G}$ be nonexpansive, $i \in I$. The
  cardinality of any Besicovitch family of balls in
  $(\mathcal{E},d_{P})$ is bounded by the same constant that bounds
  the cardinality of Besicovitch families in $G$.
\end{prop}
\begin{proof}
  Let $\mathcal{B}$ be a Besicovitch family of closed balls $B=
  \cb(\pi_{B},\epsilon_{B})$ in $(\mathcal{E},d_{P})$, where $\pi_{B}
  \in \mathcal{E}$ and $\epsilon_{B} > 0$. Note that if $\epsilon_{B}
  \ge 1$, then $\abs{\mathcal{B}} = 1$, since in that case $B=
  \mathcal{E}$ since $d_{P}$ is bounded by $1$. So let
  $\abs{\mathcal{B}} >1$, that implies $\epsilon_{B}<1$ for all $B \in
  \mathcal{B}$.

  The defining properties of a Besicovitch family translate then with
  help of Lemma \ref{lem:char_ballsEW} into
  \begin{align} \label{eq:cvg_average_besicovitch1} \pi_{B} \not \in
    B', \quad \forall B' \in \mathcal{B}\setminus \{B\}
    &&\Longleftrightarrow&& S_{\pi_{B}} \cap \cb(
    S_{\pi_{B'}},\epsilon_{B'}) = \emptyset, \qquad \forall B' \in
    \mathcal{B}\setminus \{B\}, 
  \end{align}
  and
  \begin{align}
    \label{eq:cvg_average_besicovitch2}
    \bigcap_{B \in \mathcal{B}} B \neq \emptyset
    &&\Longleftrightarrow&& \bigcap_{B \in \mathcal{B}}
    \cb(S_{\pi_{B}},\epsilon_{B}) \neq \emptyset.
  \end{align}
  Now fix $\pi$ in the latter intersection in
  \eqref{eq:cvg_average_besicovitch2} and let $s \in S_{\pi}$. Also
  fix for each $B \in \mathcal{B}$ a point $s_{B} \in S_{\pi_{B}}$
  with the property that $s_{B} \in \argmin_{\tilde s \in S_{\pi_{B}}}
  d(s,\tilde s)$ (possible by Lemma \ref{lemma:dist_supports_attained}).
  Then the family $\mathcal{C}$ of balls $\cb(s_{B},\epsilon_{B})
  \subset G$, $B \in \mathcal{B}$ is also a Besicovitch family: We
  have $s_{B} \not \in B'$ for $B \neq B'$ due to
  \eqref{eq:cvg_average_besicovitch1} and by the choice of $s_{B}$ one
  has $s \in \bigcap_{B \in \mathcal{C}} B$.
  Since the cardinality of any Besicovitch family in $G$ is bounded by
  a uniform constant, it follows, that also the cardinality of
  $\mathcal{B}$ is uniformly bounded.
\end{proof}
\begin{rem}[Euclidean metric on $\mathbb{R}^{n}$ is finite
  dimensional] \label{rem:norm_fin_dim} The cardinality of any
  Besicovitch family in $\mathbb{R}^{n}$ is uniformly bounded
  depending on $n$ \cite[Lemma 2.6]{mattila1995geometry}.
\end{rem}

\begin{lemma}[equality around support of ergodic measures implies
  equality of measures]\label{lemma:equalityBallsimpliesEqOfMeasures}
  Let $(G,d)$ be a separable complete metric space with the finite dimensional 
  metric $d$ and let $\mymap{T_{i}}{G}{G}$ be nonexpansive ($i \in I$).  If
  $\pi_{1},\pi_{2} \in \inv \mathcal{P}$ satisfy
  \begin{align}\label{eq:equalityBallsimpliesEqOfMeasures}
    \pi_{1}(\cb(S_{\pi}, \epsilon)) = \pi_{2}(\cb(S_{\pi}, \epsilon))
  \end{align}
  for all $\epsilon>0$ and all $\pi \in \mathcal{E}$, then $\pi_{1} =
  \pi_{2}$.
\end{lemma}
\begin{proof}
  From Proposition \ref{thm:decomp_ergodic_stat_measures} follows the existence
  of probability measures $q_{1},q_{2}$ on the set $\mathcal{E}$ of
  ergodic measures for $\mathcal{P}$ such that one has
  \begin{align*}
    \pi_{j}(A) = \int_{\mathcal{E}} \pi(A) q_{j}(\dd{\pi}), \qquad A
    \in \mathcal{B}(G),\, j=1,2.
  \end{align*}
  If we set $q = \tfrac{1}{2}(q_{1}+q_{2})$, then by the Radon-Nikodym 
  theorem, there are densities $f_{1},f_{2} \ge 0$ on $\mathcal{E}$
  with $q_{j} = f_{j} \cdot q$ and hence
  \begin{align*}
    \pi_{j}(A) = \int_{\mathcal{E}} \pi(A) f_{j}(\pi) q(\dd{\pi}),
    \qquad A \in \mathcal{B}(G),\, j=1,2.
  \end{align*}
  For $q$-measurable\ subsets $E \subset \mathcal{E}$, one can define a
  probability measure on $\mathcal{E}$ via
  \begin{align}\label{eq:pi char}
    \tilde \pi_{j} (E) := \int_{\mathcal{E}} \1_{E}(\pi) f_{j}(\pi)
    q(\dd{\pi}), \qquad j=1,2.
  \end{align}
  One then has for $\epsilon >0$ and $\pi \in \mathcal{E}$ that
  \begin{align}
    \label{eq:cvg_average_pi_equals_piTilde}
    \pi_{j}(\cb(S_{\pi},\epsilon)) = \tilde
    \pi_{j}(\cb(\pi,\epsilon)), \qquad j=1,2,
  \end{align}
  where $\cb(\pi,\epsilon) := \mysetc{\tilde \pi \in \mathcal{E}}{
    d_{P}(\tilde\pi,\pi) \le\epsilon}$. This is due to
  Lemma \ref{lem:char_ballsEW}, from which follows
  \begin{align*}
    \tilde\pi(\cb(S_{\pi},\epsilon)) =
    \begin{cases}
      1, & \tilde \pi \in \cb(\pi,\epsilon)\\
      0, & \text{else}
    \end{cases}.
  \end{align*}
  With the above characterizations of $\pi_j$ and $\tilde \pi_j$, we
  can use Proposition \ref{th:diff thm} to show that $f_{1}=f_{2}$ $q$-a.s.,
  which, together with \eqref{eq:pi char}, would imply that $\pi_{1} =
  \pi_{2}$, as claimed.  To apply Proposition \ref{th:diff thm} we require that
  $d_{P}$ is finite dimensional.  But this follows from
  Proposition \ref{prop:weak_besicovitch_Rn}.  So Proposition \ref{th:diff thm} applied to
  $\tilde \pi_{j}$ with respect to $q$ then gives $q$-a.s.
  \begin{equation}\label{eq:rapsberry}
    \lim_{\epsilon \to 0} \frac{\tilde \pi_{j}(\cb(\pi,
      \epsilon))}{q(\cb(\pi,\epsilon))} = f_{j}(\pi), \quad j=1,2.
  \end{equation}
  And since $\tilde\pi_{1}(\cb(\pi,\epsilon)) =
  \tilde\pi_{2}(\cb(\pi,\epsilon))$ by
  \eqref{eq:cvg_average_pi_equals_piTilde} and assumption
  \eqref{eq:equalityBallsimpliesEqOfMeasures}, we have $f_{1}=f_{2}$
  $q$-a.s., which completes the proof.
\end{proof}
\begin{rem}\label{rem:equality}
  In the assertion of Lemma \ref{lemma:equalityBallsimpliesEqOfMeasures},
  it is enough to claim the existence of a sequence
  $(\epsilon_{k}^{\pi})_{k \in \mathbb{N}} \subset \mathbb{R}_{+}$
  with $\epsilon_{k}^{\pi} \to 0$ as $k \to \infty$ satisfying
  \begin{align*}
    \pi_{1}(\cb(S_{\pi},\epsilon_{k}^{\pi})) =
    \pi_{2}(\cb(S_{\pi},\epsilon_{k}^{\pi})) \qquad \forall \pi\in
    \mathcal{E}, \, \forall k \in \mathbb{N},
  \end{align*}
  because from Proposition \ref{th:diff thm} one has the existence of the limit
  in \eqref{eq:rapsberry} $q$-a.s.
\end{rem}

\subsection{Convergence theory for $\alpha$-firmly nonexpansive 
mappings}
\label{sec:cvgAverages}

Continuing the development of the convergence theory under
greater regularity assumptions on the mappings $T_{i}$ ($i \in I$),
in this section we examine what is achievable under the assumption
that the mappings $T_i$ are $\alpha$-fne (Definition \ref{d:a-fne}). 
We restrict ourselves to the Euclidean space
$(\mathbb{R}^{n}, \norm{\cdot})$, and begin with a technical lemma that describes properties of sequences
whose relative expected distances are invariant under $T_{\xi}$.
\begin{lemma}[constant expected separation]\label{lemma:const_dist}
  Let $\mymap{T_{i}}{\mathbb{R}^{n}}{\mathbb{R}^{n}}$ be $\alpha$-fne
with $\alpha_{i}\le \alpha<1$, $i\in I$. Let
  $\mu, \nu \in \mathscr{P}(\mathbb{R}^{n})$ and $X \sim \mu$, $Y \sim
  \nu$ independent of $(\xi_{k})$ satisfy
  \begin{align*}
    \mathbb{E}\left[\norm{X_{k}^{X} - X_{k}^{Y}}^{2}\right] =
    \mathbb{E}\left[\norm{X-Y}^{2}\right] \qquad \forall k\in
    \mathbb{N},
  \end{align*}
  where $X_{k}^{x}:= T_{\xi_{k-1}}\!\!\!\cdots T_{\xi_{0}}x$ for $x
  \in \mathbb{R}^{n}$ is the RFI sequence started at $x$. Then for
  $\mathbb{P}^{(X,Y)}$-a.e.\ $(x,y) \in \mathbb{R}^{n}\times
  \mathbb{R}^{n}$ we have $X_{k}^{x}-X_{k}^{y}= x-y$
  $\mathbb{P}$-a.s.\ for all $k \in \mathbb{N}$. Moreover, if there
  exists an invariant measure for $\mathcal{P}$, then
  \begin{align*}
    \pi^{x}(\cdot) = \pi^{y}(\cdot-(x-y)) \qquad
    \mathbb{P}^{(X,Y)}\mbox{-a.s.}
  \end{align*}
  for the limiting invariant measures $\pi^{x}$ of the Ces\`{a}ro
  average of $(\delta_{x}\mathcal{P}^{k})$ and $\pi^{y}$ of the
  Ces\`{a}ro average of $(\delta_{y}\mathcal{P}^{k})$.
\end{lemma}
\begin{proof}
  By the Hilbert space characterization of $\alpha$-fne mappings 
  \eqref{eq:paafne2}, one has
  \begin{align*}
    \mathbb{E}\left[\norm{X-Y}^{2}\right] &\ge \mathbb{E}\left[ 
\norm{T_{\xi_{0}} X -
      T_{\xi_{0}} Y}^{2}\right] + \tfrac{1-\alpha}{\alpha}
    \mathbb{E}\left[\norm{(X-T_{\xi_{0}}X) - (Y-T_{\xi_{0}}Y)}^{2}\right] \\
    &\ge \dots \\
    & \ge \mathbb{E}\left[ \norm{T_{\xi_{k-1}}\cdots T_{\xi_{0}} X -
        T_{\xi_{k-1}}\cdots T_{\xi_{0}} Y}^{2}\right] \\ &  +
    \tfrac{1-\alpha}{\alpha} \sum_{j=0}^{k-1}\mathbb{E}\left[
    \norm{(T_{\xi_{j-1}}\cdots T_{\xi_{-1}} X-T_{\xi_{j}}\cdots
      T_{\xi_{0}}X) - (T_{\xi_{j-1}}\cdots
      T_{\xi_{-1}}Y-T_{\xi_{j}}\cdots T_{\xi_{0}}Y)}^{2} \right],
  \end{align*}
  where we used $T_{\xi_{-1}}:= \id$ for a simpler representation of
  the sum. We will denote $X_{k}^{x} = T_{\xi_{k-1}}\cdots T_{\xi_{0}}
  x$.  The assumption $\mathbb{E}\left[\norm{X_{k}^{X} - X_{k}^{Y}}^{2}\right] =
  \mathbb{E}\left[\norm{X-Y}^{2}\right]$ for all $k \in \mathbb{N}$ then 
implies,
  that for $j=1,\dots,k$ $\mathbb{P}$-a.s.\
  \begin{align*}
    X_{k}^{X}-X_{k-1}^{X} = X_{k}^{Y}-X_{k-1}^{Y} 
    \quad (k \in\mathbb{N}),
  \end{align*}
  and hence by induction
  \begin{align*}
    X_{k}^{X} - X_{k}^{Y} = X-Y.
  \end{align*}
  By disintegrating
  and using $(X,Y)
  \indep (\xi_{k})$ we have $\mathbb{P}$-a.s.\
  \begin{align*}
    0 &= \cex{\norm{(X-X_{k}^{X}) - (Y-X_{k}^{Y} )}^{2}}{X,Y} \\ &=
    \int_{I^{k+1}} \norm{(X-T_{i_{k}} \cdots T_{i_{0}}X) - (Y-T_{i_{k}}
      \cdots T_{i_{0}}Y)}^{2} \mathbb{P}^{\xi}(\dd{i_{k}}) \cdots
    \mathbb{P}^{\xi}(\dd{i_{0}}).
  \end{align*}
 Consequently, for $\mathbb{P}^{(X,Y)}$-a.e.\ $(x,y) \in \mathbb{R}^{n}
  \times \mathbb{R}^{n}$, we have 
  \begin{align*}
    X_{k}^{x} - X_{k}^{y} = x-y \quad \forall k \in \mathbb{N}
    \quad \mathbb{P}-\mbox{a.s.}
  \end{align*}
  So in particular for any $A \in \mathcal{B}(\mathbb{R}^{n})$
  \begin{align*}
    p^{k}(x,A) = \mathbb{P}(X_{k}^{x} \in A) = \mathbb{P}(X_{k}^{y}
    \in A-(x-y)) = p^{k}(y,A-(x-y))
  \end{align*}
  and hence, denoting $f_{h}=f(\cdot + h)$ and $\nu_{k}^{x} =
  \tfrac{1}{k} \sum_{j=1}^{k} p^{j}(x,\cdot)$, one also has for $f \in
  C_{b}(\mathbb{R}^{n})$ by Theorem \ref{cor:cesaroConvergenceRn}
  \begin{equation*}
    \nu_{k}^{y}f_{x-y} \to \pi^{y} f_{x-y} = \pi_{x-y}^{y} f \quad{and}\quad
    \nu_{k}^{x} f \to \pi^{x}f\mbox{ as }k\to\infty,
  \end{equation*}
  where $ \pi_{x-y} ^{y}:= \pi^{y} (\cdot -
  (x-y))$. So from $\nu_{k}^{y}f_{x-y} = \nu_{k}^{x} f$ for any $f \in
  C_{b}(\mathbb{R}^{n})$ and $k\in\mathbb{N}$ it follows 
  that $\pi_{x-y}^{y} =\pi^{x}$.
\end{proof}

We can now give the proof of the second main result.  
For a given $\mymap{h}{\mathbb{R}^{n} \times \mathbb{R}^{n}}{\mathbb{R}}$
  we will define sequences of functions $(\hhbar_{k})$ on $\mathbb{R}^{n} \times
  \mathbb{R}^{n}$ via
  \begin{align*}
    \hhbar_{k}(x,y) \equiv \mathbb{E}\left[h(X_{k}^{x}, X_{k}^{y}) \right], 
    \quad X_{k}^{z}:= T_{\xi_{k-1}} \cdots T_{\xi_{0}}z \mbox{ for any }z \in
  \mathbb{R}^{n}\quad (k \in \mathbb{N}).
  \end{align*}
    Note that, by continuity of
  $T_{i}$, $i \in I$ and Lebesgue's dominated convergence theorem,  
  $\hhbar_{k} \in C_{b}(\mathbb{R}^{n}\times \mathbb{R}^{n})$ for all
  $k \in \mathbb{N}$ whenever $h\in C_{b}(\mathbb{R}^{n}\times \mathbb{R}^{n})$.

\noindent 
\textbf{Proof of Theorem \ref{thm:a-firm convergence Rn}.}
  Let $x,y \in \mathbb{R}^{n}$, define $F(x,y)\equiv\norm{x-y}^2$ and the 
  corresponding sequence of functions 
  \begin{align*}
    \Fbar_{k}(x,y) \equiv \mathbb{E}\left[F(X_{k}^{x}, X_{k}^{y}) \right], 
    \quad X_{k}^{z}:= T_{\xi_{k-1}} \cdots T_{\xi_{0}}z \mbox{ for any }z \in
  \mathbb{R}^{n}\quad (k \in \mathbb{N}).
  \end{align*}
    By the remarks preceding this proof, 
  $\Fbar_{k} \in C_{b}(\mathbb{R}^{n}\times \mathbb{R}^n)$ for all $k \in \mathbb{N}$. 
  From the regularity of $T_{i}$, $i \in I$ and the 
  characterization \eqref{eq:paafne2}, we get that a.s.\ for all
  $k \in \mathbb{N}$
  \begin{align}
    \label{eq:averadness_suppPI}
    \norm{X_{k}^{x} - X_{k}^{y}}^{2} \ge \norm{X_{k+1}^{x} -
      X_{k+1}^{y}}^{2} + \tfrac{1-\alpha}{\alpha} \norm{(X_{k}^{x} -
      X_{k+1}^{x}) - (X_{k}^{y} - X_{k+1}^{y})}^{2}.
  \end{align}
  After computing the expectation, this is the same as
  \begin{align*}
    \Fbar_{k}(x,y) \ge \Fbar_{k+1}(x,y) + \tfrac{1-\alpha}{\alpha}
    \mathbb{E}\left[\norm{(X_{k}^{x} - X_{k+1}^{x}) - (X_{k}^{y} -
        X_{k+1}^{y})}^{2}\right].
  \end{align*}
  We conclude that $(\Fbar_{k}(x,y))$ is a monotonically nonincreasing
  sequence for any $x,y\in G$.\\
  Recall the notation $S_{\pi}\equiv \Supp\pi$ for some measure $\pi$.  
  Let $s, \tilde s \in S_{\pi}$ for the ergodic invariant measure $\pi \in \mathcal{E}$ and define
  the sequence of measures $\gamma_k$ by
  \begin{align*}
    \gamma_{k} f := \mathbb{E}\left[ f(X_{k}^{s}, X_{k}^{\tilde s})\right]
  \end{align*}
  for any measurable function $\mymap{f}{\mathbb{R}^{n} \times
    \mathbb{R}^{n}}{\mathbb{R}}$. 
   Note that due to nonexansiveness  the pair $(X_{k}^{s}, X_{k}^{\tilde s})$ a.s. takes
   values in $G_r:=\{(x,y) \in \mathbb{R}^{n} \times
    \mathbb{R}^{n}: ||x-y||^2 \leq r\}$ for $r =  ||s-\tilde{s}||^2$, so that  
  $\gamma_k$ is concentrated on this set. Since $(X_{k}^{s})$ is a tight
  sequence by Lemma \ref{lemma:tightnessOfNu_n}, and likewise for $(X_{k}^{\tilde{s}})$, 
  we know from  Lemma \ref{lemma:weakCVG_productSpace} that the sequence 
  $(\gamma_k)$ is tight as well.
  Let $\gamma$ be a cluster point of $(\gamma_{k}),$ which is again concentrated 
on $G_{ ||s-\tilde{s}||^2},$   and consider a subsequence
  $(\gamma_{k_j})$ such that $\gamma_{k_j} \rightarrow \gamma.$ By 
   Lemma \ref{lemma:weakCVG_productSpace} we also know that $\gamma \in C(\nu_{1}, 
\nu_{2})$
   where $\nu_1$ and $\nu_2$ are the distributions of the limit 
in convergence in distribution of $(X_{k_j}^{s})$ 
   and $(X_{k_j}^{\tilde{s}}).$
 For any $f \in C_b(\mathbb{R}^{n}\times \mathbb{R}^n)$ we have 
 $\gamma_{k_j} f \rightarrow \gamma f.$  So consider the case $f=F^M$ where 
 $F^{M} := \min(M,F)$ for  $M \in \mathbb{R}$.  
 Since $\norm{x-y}^2=F(x,y)=F^M(x,y)$ almost surely (with respect to $\gamma_{k_j}$ and $\gamma$) 
 for $M \ge \norm{s-\tilde s}^{2}$, we have  
 \begin{align*}
    \gamma_{k_j}F = \gamma_{k_j} F^{M} \rightarrow \gamma F^{M}= \gamma F.
  \end{align*}
However,  by the monotonicity in \eqref{eq:averadness_suppPI} we now also 
obtain convergence for the entire sequence:
 \begin{align*}
    \gamma_{k}F = \gamma_{k} F^{M}  \searrow \gamma F^{M}=\gamma F.
  \end{align*}
  Let $(X,Y) \sim \gamma$ and $(\tilde \xi_{k}) \indep (\xi_{k})$ be
  another i.i.d.\ sequence with $(X,Y) \indep (\tilde \xi_{k}),
  (\xi_{k})$. We use the notation 
  $\tilde X_{k}^{x} := T_{\tilde \xi_{k-1}} \cdots T_{\tilde\xi_{0}} x$, $x \in \mathbb{R}^{n}.$ 
  Define the sequence of functions 
  \[
   \Fbar^M_k(x,y) \equiv \mathbb{E}\left[F^M(X_{k}^{x}, X_{k}^{y}) \right] \quad (k \in \mathbb{N}),
  \]
  and note that  $\Fbar^{M}_{k}  \in C_b(\mathbb{R}^{n}\times \mathbb{R}^{n})$.  
  When $M \ge \norm{s-\tilde s}^{2}$ this yields 
  \begin{align*}
 \gamma \Fbar_k = \gamma \Fbar^{M}_{k} &= 
 \mathbb{E}\left[ \min \left(M,\norm{\tilde X_{k}^{X} - \tilde X_{k}^{Y}}^{2}\right)\right] 
 = \lim_{j \to \infty} \gamma_{ k_{j}} \Fbar^{M}_{k} \\ 
 &=  \lim_{j \to \infty} \mathbb{E}\left[ 
      \min\left(M,\norm{\tilde
            X_{k}^{X_{ k_{j}}^{s}} - 
            \tilde X_{k}^{X_{k_{j}}^{\tilde s}}}^{2}%
            \right) 
                \right] \\
    &= \lim_{j \to \infty} \mathbb{E}\left[ \min\left(M,\norm{
          X_{k+ k_{j}}^{s} - X_{k+ k_{j}}^{\tilde s}}^{2}
      \right) \right] \\ 
      &= \lim_{j \to \infty} \gamma_{k+ k_{j}} F^{M} = 
      \gamma F^{M} = \gamma F.
  \end{align*}
This means that  for all $k \in \mathbb{N},$
  \begin{align*}
    \mathbb{E}\left[\norm{X_{k}^{X} - X_{k}^{Y}}^{2}\right] =
    \mathbb{E}[\norm{X-Y}^{2}]. 
  \end{align*}
  For $\mathbb{P}^{(X,Y)}$-a.e.\ $(x,y)$ we have $x,y \in S_{\pi}$ and thus 
  $\pi^x=\pi^y=\pi$ where $\pi^x$ is the unique ergodic measure with $x\in 
S_{\pi^{x}}$
 (see Remark \ref{cor:weakCvgMetricSpace}).
 An application of Lemma \ref{lemma:const_dist} then yields  
 $\pi(\cdot) = \pi(\cdot-(x-y))$, i.e.\ $x=y$. Hence $X=Y$ a.s. implying
  $\nu_{1}=\nu_{2}=:\nu$ and $\gamma F = 0$. That means
  \begin{align*}
    \gamma_{k}F = \mathbb{E}\left[\norm{X_{k}^{s}
        - X_{k}^{\tilde{s}}}^{2}\right] \to 0 \qquad \text{as } k \to \infty.
  \end{align*}
  Now Lemma \ref{lemma:prokhorovDist_properties} yields
  \begin{align*}
    \mathbb{P}\left(\norm{X_{k}^{s} - X_{k}^{\tilde{s}}} >
    \epsilon\right) \le \frac{\mathbb{E}\left[\norm{X_{k}^{s} -
    X_{k}^{\tilde{s}}}\right]}{\epsilon} \le \frac{\mathbb{E}\left[\sqrt{\norm{X_{k}^{s} -
    X_{k}^{\tilde{s}}}^{2}}\right]}{\epsilon} \to 0
  \end{align*}
  as $k \to \infty$ for any $\epsilon > 0$;  so this yields convergence 
  of the corresponding probability  measures $\delta_{s} \mathcal{P}^{k}$ and 
  $\delta_{\tilde{s}} \mathcal{P}^{k}$  in the Prokhorov metric:
  \begin{align*}
    d_P(\delta_{s} \mathcal{P}^{k},\delta_{\tilde{s}} \mathcal{P}^{k}) 
    \rightarrow 0.
  \end{align*}
By the triangle inequality, therefore, if 
  $\delta_{s} \mathcal{P}^{k_j} \rightarrow \nu$,  then also
  $\delta_{\tilde{s}} \mathcal{P}^{k_j} \rightarrow \nu$ for any
  $\tilde{s} \in S_{\pi}$.  Hence 
  \begin{align*}
    d_{P}(\delta_{\tilde{s}}\mathcal{P}^{k_{j}},\nu) &\le
    d_{P}(\delta_{s}\mathcal{P}^{k_{j}},
    \delta_{\tilde{s}}\mathcal{P}^{k_{j}}) +
    d_{P}(\delta_{s}\mathcal{P}^{k_{j}}, \nu) \to 0,\quad \text{as
    } j \to \infty.
  \end{align*}
  By Lebesgue's dominated convergence theorem we conclude 
  that, for any $f \in C_{b}(\mathbb{R}^{n})$ and $\mu \in
  \mathscr{P}(S_{\pi})$,
  \begin{align*}
    \mu \mathcal{P}^{k_{j}} f = \int_{S_{\pi}}
    \delta_s \mathcal{P}^{k_{j}}f \mu(\dd{s}) \to \nu f, \quad
    \text{as } j \to \infty.
  \end{align*}
In particular, 
  $\mu \mathcal{P}^{k_{j}} \to \nu$ and taking $\mu=\pi$ yields  
  $\nu = \pi$.  Thus, all cluster points of
  $(\delta_{s}\mathcal{P}^{k})$ for all $s \in S_{\pi}$ have the same
  distribution $\pi$ and hence, because the sequence is tight, 
  $\delta_{s} \mathcal{P}^{k}=p^{k}(x,\cdot) \to \pi$.

  Now, let $\mu \in \mathscr{P}(S)$, where $S = \bigcup_{\pi \in
    \mathcal{E}} S_{\pi}$. By what  we have just shown we have for $x\in
  \Supp \mu$, that $p^{k}(x,\cdot) \to \pi^{x}$, where $\pi^{x}$ is
  unique ergodic measure with $x\in S_{\pi^{x}}$. Then, again by Lebesgue's
  dominated convergence theorem, one has for any $f \in C_{b}(\mathbb{R}^{n}),$
  \begin{align}
  \label{eq:muconv}
    \mu \mathcal{P}^{k} f = \int f(y) p^{k}(x, \dd{y}) \mu(\dd{x}) \to
    \int f(y) \pi^{x}(\dd{y}) \mu(\dd{x}) =:  \pi^{\mu} f   
    \mbox{ as }k \to \infty,  
  \end{align}
and the measure $\pi^{\mu} $ is again invariant
  for $\mathcal{P}$ by invariance of $\pi^{x}$ for all $x \in S.$
  Now, let $\mu = \delta_{x}$, $x \in \mathbb{R}^{n} \setminus S$. We
  obtain the tightness of $(\delta_{x}\mathcal{P}^{k})$ from the
  tightness of $(\delta_{s}\mathcal{P}^{k})$ for $s \in S$.  Indeed,  for
  $\epsilon >0$ there exists a compact
  $K_{\epsilon} \subset \mathbb{R}^{n}$ with
  $p^{k}(s,K_{\epsilon})> 1-\epsilon$ for all $k \in \mathbb{N}.$
  This together with the fact that $T_{i}$, $i \in I$ is nonexpansive implies that 
  $\norm{X_{k}^{x}-X_{k}^{s}}\le \norm{x-s}$ for all
  $k \in \mathbb{N}$ hence
  $p^{k}(x,\cb(K_{\epsilon},\norm{x-s}))>1-\epsilon$, where $p$ is the 
  transition kernel defined by \eqref{eq:trans kernel}. 
  Tightness  implies the existence of a cluster point $\nu$ of the sequence
  $(\delta_{x}\mathcal{P}^{k})$. From Theorem \ref{cor:cesaroConvergenceRn}
  we know that $\nu_{k}^{x} = \tfrac{1}{k} \sum_{j=1}^{k}
  \delta_{x}\mathcal{P}^{j} \to \pi^{x}$ for some $\pi^{x} \in \inv
  \mathcal{P}$ with $S_{\pi^x} \subset S.$  Furthermore, we have $\nu \in
  \mathscr{P}(S_{\pi^{x}}) \subset \mathscr{P}(S) $ by
  Lemma \ref{thm:invMeas_nonexpans_map}\eqref{item:nonExpProp_4}. So by 
  \eqref{eq:muconv} there exists $\pi^{\nu} \in \inv \mathcal{P}$ with
  $\nu\mathcal{P}^{k} \to \pi^{\nu}$.
  
  In order to complete the proof we have to show that $\nu = \pi^{x}$,
  i.e.\ $\pi^{x}$ is the unique cluster point of
  $(\delta_{x}\mathcal{P}^{k})$ and hence convergence follows by
  Proposition \ref{thm:cvg_subsequences}.  It suffices to  show that
  $\pi^{\nu} = \pi^{x}$, since then, as $k \to \infty$
  \begin{align*}
    d_{P}(\nu, \pi^{x}) = \lim_{k} d_{P}(\delta_{x}\mathcal{P}^{k},
    \pi^{x}) = \lim_{k} d_{P}(\delta_{x}\mathcal{P}^{k+j}, \pi^{x}) =
    d_{P}(\nu \mathcal{P}^{j}, \pi^{x}) = d_{P}(\nu \mathcal{P}^{j}, \pi^{\nu}) 
\to 0.
  \end{align*}
To begin, fix  $\pi \in
  \inv\mathcal{P}$.  For any $\epsilon>0$ let $A_{k} := \{X_{k}^{x} \in 
\cb(S_{\pi},
  \epsilon)\}.$  By nonexpansivity $A_{k} \subset A_{k+1}$ for $k \in
  \mathbb{N}$, since we have by Lemma \ref{lemma:support_invariant_distr}
  a.s.
  \begin{align*}
    \dist(X_{k+1}^{x}, S_{\pi}) \le \dist(X_{k+1}^{x}, T_{\xi_{k}}
    S_{\pi}) \le \dist(X_{k}^{x},S_{\pi}). 
  \end{align*}
  Hence $(p^{k}(x, \cb(S_{\pi}, \epsilon))) = (\mathbb{P}(A_{k}))$ is
  a monotonically increasing sequence and bounded from above and
  therefore the sequence converges to some $b_{\epsilon}^{x} \in
  [0,1]$ as $k\to \infty$. It
  follows 
  \begin{align}
  \label{eq:b_expconv}
  b_{\epsilon}^{x} =  \lim_{k} p^{k}(x,\cb(S_{\pi},\epsilon)) = \lim_{k}
    \frac{1}{k}\sum_{j=1}^{k} p^{j}(x,\cb(S_{\pi},\epsilon)).
  \end{align}
  and thus
  $\nu(\cb(S_{\pi},\epsilon)) = \pi^{x}(\cb(S_{\pi},\epsilon))$ 
  for all $\epsilon$, which make $\cb(S_{\pi},\epsilon)$ both
  $\nu$- and
  $\pi^{x}$-continuous. Note that there are at most countably many
  $\epsilon>0$ for which this may fail, see \cite[Chapter 3, Example
  1.3]{kuipers1974uniform}). 
 
  With the same argument used for  \eqref{eq:b_expconv} we also obtain for any 
  $k \in \mathbb{N}$ that $\nu \mathcal{P}^{k}(\cb(S_{\pi},\epsilon)) =  
\pi^{x}(\cb(S_{\pi},\epsilon))$
  with only countably many $\epsilon$ excluded, 
 and so 
  \begin{align*}
    \pi^{\nu}(\cb(S_{\pi},\epsilon)) = \pi^{x}(\cb(S_{\pi},\epsilon))
  \end{align*}
  also needs to hold for all except countably many $\epsilon$.  
  Since
  $\pi^{\nu} \in \inv \mathcal{P}$,  this implies 
  that $\pi^{\nu} = \pi^{x}$ by
  Lemma \ref{lemma:equalityBallsimpliesEqOfMeasures} combined with
  Remark \ref{rem:equality}.
  For a general initial measure $\mu_0 \in \mathscr{P}(\mathbb{R}^{n})$,
  one has, yet again by Lebesgue's dominated convergence theorem, that
  \begin{align*}
    \mu_0 \mathcal{P}^{k} f = \int f(y) p^{k}(x, \dd{y}) \mu_0(\dd{x}) \to
    \int f(y) \pi^{x}(\dd{y}) \mu_0(\dd{x}) =: \pi^{\mu_0} f,
  \end{align*}
  where $\pi^{x}$ denotes the limit of $(\delta_{x} \mathcal{P}^{k})$
  and the measure $ \pi^{\mu_0}$ is again invariant for $\mathcal{P}$.
  This completes the proof.
\hfill $\Box$

\begin{rem}[a.s.\ convergence]\label{rem:asCvgSequence}
  If we were to choose $X$ and $Y$ in \eqref{eq:averadness_suppPI} such that
  $\mathcal{L}(X),\mathcal{L}(Y) \in \mathscr{P}(S_{\pi})$, where $\pi
  \in \mathcal{E}$, then still $\gamma_{k} F \to \gamma F = 0$,
  where $\gamma \in C(\pi,\pi)$. For $(W,Z) \sim \gamma$ it
  still holds that $W=Z$ and hence
  \begin{align*}
    \norm{X_{k}^{X} - X_{k}^{Y}} \to 0 \qquad \text{a.s.}
  \end{align*}
  by monotonicity of $(\gamma_{k}F)$.
\end{rem}

\subsubsection{Structure of ergodic measures for $\alpha$-firmly 
nonexpansive mappings}
\label{sec:structure}

\begin{prop}[structure of ergodic measures]\label{thm:struOfErgMeas}
  Let $\mymap{T_{i}}{\mathbb{R}^{n}}{\mathbb{R}^{n}}$ be
  $\alpha$-fne with constant $\alpha_{i}\le \alpha<1$ ($i\in I$) and
  assume there exists an invariant probability distribution for
  $\mathcal{P}$.  Any two ergodic measures $\pi,\tilde \pi$ are
  shifted versions of each other, i.e.\ there exist $s \in \Supp {\pi}$
  and $\tilde s \in \Supp {\tilde \pi}$ with $\pi = \tilde \pi(\cdot -
  (s-\tilde s))$.
\end{prop}
\begin{proof}
  Denote $\Supp \pi = S_{\pi}$.  Since we can find for any $s \in S_{\pi}$ a closest point $\tilde s
  \in S_{\tilde \pi}$, i.e.\ $\dist(S_{\pi}, S_{\tilde \pi}) =
  d(s,\tilde s)$, by Lemma \ref{lemma:dist_supports_attained}, the
  assertion follows from
  \begin{align}\label{eq:structure1}
    \dist(S_{\pi}, S_{\tilde \pi}) \le
    \sqrt{\mathbb{E}[\norm{X_{k}^{s} - X_{k}^{\tilde s}}^{2}]} \le
    \norm{s - \tilde s} \qquad \forall k \in \mathbb{N},
  \end{align}
  where we also used that $\Supp \mathcal{L}(X_{k}^{s}) \subset
  S_{\pi}$, and $\Supp \mathcal{L}(X_{k}^{\tilde s}) \subset S_{\tilde
    \pi}$.
\end{proof}

\begin{prop}[specialization to projectors]
  Let the mappings
  $\mymap{T_{i}=P_{i}}{\mathbb{R}^{n}}{\mathbb{R}^{n}}$ be projectors
  onto nonempty closed and convex sets ($i \in I$).  If there exist
  two ergodic measures $\pi_{1},\pi_{2}$, then there exist infinitely
  many ergodic measures $\pi_{\lambda}$ with $\pi_{\lambda} \equiv
  \pi_{2}(\cdot-\lambda a)$ for all $\lambda \in [0,1]$, where $a$ is
  the shift such that $\pi_1=\pi_2(\cdot-a)$.
\end{prop}
\begin{proof}
  For any pair $(s_{1},s_{2}) \in \Supp {\pi_{1}} \times \Supp {\pi_{2}}$ of
  closest neighbors it holds that $a = s_{1}-s_{2}$ by
  Proposition \ref{thm:struOfErgMeas}. 
  Lemma \ref{lemma:const_dist} and \eqref{eq:structure1} yield 
  $P_{\xi}s_{1} = P_{\xi}s_{2} + a$
  a.s. Hence, $a \perp (s_{i}-P_{\xi}s_{i})$, $i=1,2$, and then
  $P_{\xi}(s_{2} + \lambda a) = P_{\xi}s_{2} + \lambda a$ for $\lambda
  \in [0,1]$. Hence $X_{k}^{s_{2} + \lambda a} = X_{k}^{s_{2}} +
  \lambda a$ and $\lim_{k} \mathcal{L}(X_{k}^{s_{2} + \lambda a}) =
  \pi_{2}(\cdot - \lambda a)$. Note, that if $\mathcal{P}$ is Feller
  and the sequence $(\mu \mathcal{P}^{k})$ converges for some $\mu \in
  \mathscr{\mathbb{R}^{n}}$, then the limit is also an invariant measure.
\end{proof}

\subsection{Rates of Convergence}\label{sec:rates}
We now prove the third main result of this paper.  \\
\textbf{Proof of Theorem \ref{t:msr convergence}}. 
  First note that since $G$ is compact and $\inv\Pcal$ is nonempty, 
  there is at least one 
  $\pi \in\inv\mathcal{P}\cap\mathscr{P}_2(G)$ and one $\mu\in \mathscr{P}_{2}(G)$ with 
  $W_2(\mu,\pi)<\infty$ and  $\mu\mathcal{P}\in\mathscr{P}_{2}(G)$.
  The Markov operator $\mathcal{P}$ is therefore a
self-mapping on $\mathscr{P}_{2}(G)$, hence
$W_2(\mu,\mu\mathcal{P}) < \infty$, and for any 
$\mu_1, \mu_2\in \mathscr{P}_{2}(G)$ the set of optimal couplings
$C_*(\mu_1, \mu_2)$ is nonempty 
(see Lemma \ref{lemma:WassersteinMetric_prop}).  
Since $(H, d)$ is 
a Hadamard space and $G\subset H$, the function $\Psi(\mu)$ defined by 
\eqref{eq:Psi}
is extended real-valued, nonnegative (see Lemma \ref{lem:psi_f nonneg}), 
and finite since $C_*(\mu, \pi)$ and 
$\inv\mathcal{P}$ are nonempty.
Moreover, by assumption 
\eqref{t:msr convergence b} and the definition of metric subregularity
(Definition \ref{d:(str)metric (sub)reg}) this  satisfies 
$\Psi(\pi)=0\iff \pi\in\inv\mathcal{P}$, hence $\Psi^{-1}(0)=\inv\mathcal{P}$
 and $\Psi(\pi)=0$ for all $\pi\in\inv\mathcal{P}$, and for all
$\mu\in\mathscr{P}_{2}(G)$ 
\begin{eqnarray}
\inf_{\pi\in\inv\mathcal{P}} W_2^{2}(\mu, \pi) &=& 
  \inf_{\pi\in\Psi^{-1}(0)}W_2^{2}(\mu, \pi)
  \nonumber\\
 &\le& 
\paren{\rho\paren{d_{\mathbb{R}}(0,\Psi(\mu))}}^2
 = \paren{\rho(\Psi(\mu))}^2.\nonumber
\end{eqnarray}
Rewriting this for the next step yields
\begin{equation}\label{eq:rate step 1}
   \tfrac{1-\alpha}{\alpha}
   \paren{\rho^{-1}\paren{%
   \inf_{\pi\in\inv\mathcal{P}}W_2(\mu, \pi)}}^{2}
   \leq \tfrac{1-\alpha}{\alpha}\Psi^2(\mu).
\end{equation}
On the other hand, by assumption 
\eqref{t:msr convergence a} and 
Proposition \ref{thm:Tafne in exp 2 pafne of P} 
(which applies because we are on a separable Hadamard space) we
have 
\begin{eqnarray}
  \tfrac{1-\alpha}{\alpha}\Psi^2(\mu)
   &\leq& 
   \int_{G\times G}\mathbb{E} \left[\psi_2(x,y, T_\xi x, T_\xi y)\right]\ \gamma(dx, dy)
   \nonumber\\
   &\leq&
     (1+\epsilon) W_2^{2}(\mu,\pi) - W_2^{2}(\mu\mathcal{P},\pi)
      \quad\forall \pi\in\inv\mathcal{P}, 
      \forall \mu\in\mathscr{P}_{2}(G).
      \label{eq:rate step 2}
\end{eqnarray}
Incorporating \eqref{eq:rate step 1} into 
\eqref{eq:rate step 2} and rearranging the inequality yields
\begin{eqnarray*}
W_2^{2}(\mu\mathcal{P},\pi)
   \!&\leq&\! 
      (1+\epsilon) W_2^{2}(\mu,\pi) - 
         \tfrac{1-\alpha}{\alpha}
   \paren{\rho^{-1}\paren{%
   \inf_{\pi'\in\inv\mathcal{P}}W_2(\mu, \pi')}}^{2}
      \quad\forall \pi\in\inv\mathcal{P}, 
      \forall \mu\in\mathscr{P}_{2}(G).
\end{eqnarray*}
Since this holds at {\em any} $\mu\in\mathscr{P}_2(G)$, it certainly 
holds at the iterates $\mu_k$ with initial distribution $\mu_0\in\mathscr{P}_2(G)$
since $\mathcal{P}$ is a 
self-mapping on $\mathscr{P}_2(G)$.   Therefore 
\begin{eqnarray}\label{eq:gauge convergence 0}
&& W_2\paren{\mu_{k+1},\, \pi}
\leq\\ 
&&\qquad \sqrt{(1+\epsilon) W_2^2\paren{\mu_{k},\, \pi} - 
\frac{1-\alpha}{\alpha}
\paren{\rho^{-1}\paren{%
\inf_{\pi'\in\inv\mathcal{P}}W_2\paren{\mu_{k},\, \pi'}
}
}^2
} 
\quad\forall \pi\in\inv\mathcal{P}, ~
\forall k \in \mathbb{N}.\nonumber
\end{eqnarray}%

Equation \eqref{eq:gauge convergence 0} simplifies.  
Indeed, by Lemma \ref{lemma:invMeasuresClosed}, $\inv \mathcal{P}$ is closed 
with respect to convergence 
in distribution.  
Moreover, since $G$ is assumed to be compact, 
$\mathscr{P}_2(G)$ is locally compact (\cite[Remark 7.19]{AmbGigSav2005}
so, for every $k\in\Nbb$  the infimum in  \eqref{eq:gauge convergence 0}
is attained  at some $\pi_k$. 
This yields
\begin{equation}
W_2^{2}(\mu_{k+1},\pi_{k+1} ) \leq 
W_2^{2}(\mu_{k+1},\pi_{k} ) \leq 
(1+\epsilon) W_2^{2}(\mu_k, \pi_k)  - 
\tfrac{1-\alpha}{\alpha}\paren{\rho^{-1}
\paren{ W_2(\mu_k, \pi_k)}}^2
\quad\forall k \in \mathbb{N}.
\label{eq:gauge convergence intermed}
\end{equation}
Taking the square root and recalling \eqref{eq:theta} and \eqref{eq:gauge}
yields \eqref{eq:gauge convergence}.

To obtain convergence, note that for $\mu_0\in\mathscr{P}_{2}(G)$
satisfying $W_2(\mu_0,\pi)<\infty$ and  $\mu_0\mathcal{P}\in\mathscr{P}_{2}(G)$
(exists by compactness of $G$
), the triangle inequality and  
\eqref{eq:gauge convergence intermed} yield
\begin{eqnarray*}
 W_2(\mu_{k+1},\mu_k )&\leq&  W_2(\mu_{k+1},\pi_{k} )
 + W_2(\mu_{k},\pi_{k} )\\
 &\leq &\theta\paren{W_2\paren{\mu_k, \pi_k}}+ W_2(\mu_{k},\pi_{k} ).
\end{eqnarray*}
Using \eqref{eq:gauge convergence} 
and continuing by backwards induction yields
\[
 W_2(\mu_{k+1},\mu_k )\leq \theta^{k+1}\paren{ d_0}+ 
 \theta^k\paren{ d_0}
\]
where $d_0\equiv \inf_{\pi\in\inv\mathcal{P}} W_2\paren{\mu_0,\pi}$.
Repeating this argument, for any $k<m$ 
\[
 W_2(\mu_{m},\mu_k )\leq \theta^{m}\paren{ d_0} + 
 2\sum_{j=k+1}^{m-1}\theta^j\paren{ d_0}
 +\theta^k\paren{ d_0}.
\]
By assumption, $\theta$ satisfies \eqref{eq:theta}, so for any $\delta>0$
\begin{eqnarray*}
 W_2(\mu_{m},\mu_k )&\leq& 
 \theta^{m}\paren{ d_0} + 
 2\sum_{j=k+1}^{m-1}\theta^j\paren{ d_0}
 +\theta^k\paren{ d_0}\nonumber\\
& \leq& 2\sum_{j=k+1}^{\infty}\theta^j\paren{ d_0}
+ \theta^k\paren{ d_0}
 <\delta
\end{eqnarray*}
for all $k, m$ large enough;  that is the sequence $(\mu_k)_{k\in\Nbb}$ is a Cauchy sequence
in $(\mathscr{P}_{2}(G), W_2)$ -- a separable complete metric space 
(Lemma \ref{lemma:WassersteinMetric_prop}
\eqref{lemma:WassersteinMetric_prop iii}) -- 
and therefore convergent to some probability measure 
$\pi^{\mu_0}\in \mathscr{P}_{2}(G)$. 
By Proposition \ref{thm:Feller} the Markov operator $\mathcal{P}$ is Feller since $T_i$ is continuous, and 
by Proposition \ref{thm:construction_inv_meas} when a Feller Markov chain converges
in distribution, it does so to an invariant measure:  $\pi^{\mu_0} \in \inv \mathcal{P}$.
\hfill $\Box$

\textbf{Proof of Corollary \ref{t:msr convergence - linear}}. 
In the case that the gauge $\rho$ is linear with constant $\kappa'$, 
then $\theta(t)$ is linear with constant 
\[
c=\sqrt{1+ \epsilon  - \frac{1-\alpha}{\kappa^2\alpha}}<1,  
\]
where $\kappa\geq \kappa'$ satisfies $\kappa^2\geq (1-\alpha)/(\alpha(1+\epsilon))$.
Specializing the argument in the proof above to this particular $\theta$ 
shows that, for any
$k$ and $m$ with $k<m$, we have
\begin{equation}\label{DR.3}
    \begin{aligned}
W_2\paren{\mu_m,\mu_{k}} &\,\le\, 
d_0 c^m + 2d_0\sum_{j=k+1}^{m-1}c^j+ d_0c^k.\\
    \end{aligned}
\end{equation}
Letting $m\to \infty$ in \eqref{DR.3} yields R-linear convergence
(Definition \ref{d:q-r-lc})
with rate $c$ given above and leading constant 
$\beta= \frac{1+c}{1-c}d_0$.

If, in addition, $\inv\mathcal{P}$ is a singleton, then 
$\{\pi^{\mu_0}\} = \inv\mathcal{P}$
in the above and convergence is actually Q-linear, which completes the proof.
\hfill $\Box$

%

\section{Examples: Stochastic Optimization and Inconsistent Nonconvex Feasibility}
\label{sec:incFeas}

To fix our attention we focus on the following optimization problem
  \begin{equation}\label{eq:opt prob}
  \underset{\mu\in \mathscr{P}_2(\Rn)}{\mbox{minimize}}
  \int_{\Rn}\mathbb{E}_{\xi}[f_{\xi^f}(x) + g_{\xi^g}(x)]\mu(dx). 
  \end{equation}
It is assumed throughout that $\mymap{f_{i}}{\mathbb{R}^{n}}{\mathbb{R}}$ is continuously 
  differentiable for all $i \in I_f$ and that $\mymap{g_{i}}{\mathbb{R}^{n}}{\mathbb{R}}$
  is proper and lower semi-continuous for all $i \in I_g$.   
The random variable with values on $I_f\times I_g$ will be denoted $\xi = (\xi^f, \xi^g)$. 
This model covers deterministic composite optimization as a special case: $I_f$ and $I_g$ consist
of single elements and the measure $\mu$ is a point mass.  

The algorithms reviewed in this section rely on resolvents of the functions
$f_i$ and $g_i$, denoted $\Jcal_{f_i}$ and $\Jcal_{g_i}$.   The resolvent of a 
subdifferentially regular function 
$\mymap{f}{G\subset\Rn}{\mathbb{R}}$ (the epigraph of $f$ is Clarke regular \cite{VA}) 
is defined by $\Jcal_f(x)\equiv \paren{\Id+\partial f}^{-1}(x)\equiv \set{z\in G}{x=z+\partial f(z)}$.  
For proper, lower semicontinuous convex functions $\mymap{f}{\Rn}{\mathbb{R}\cup\{+\infty\}}$, 
this is equivalent to the proximal mapping \cite{Moreau65} 
defined by
\begin{equation}\label{eq:prox}
 \prox_{f}(x)\equiv\argmin_{y}\{f(y)+ \tfrac{1}{2}\norm{y-x}^2\}.
\end{equation}
In general one has
\begin{equation}\label{eq:resolvent}
 \prox_{f}(x) \subset \Jcal_f(x)
\end{equation}
whenever the subdifferential is defined.

\subsection{Stochastic (nonconvex) forward-backward splitting}\label{ex:spg ncvx}
We begin with a general prescription of the forward-backward splitting algorithm together with 
abstract properties of the corresponding fixed point mapping, and then specialize this to 
more concrete instances.   
\begin{algorithm}    
\SetKwInOut{Output}{Initialization}
  \Output{Set $X_{0} \sim \mu_0 \in \mathscr{P}_2(G)$, 
  $X_0\sim \mu$,  ${t}>0$, 
   and $(\xi_{k})_{k\in\Nbb}$ 
  another i.i.d.\ sequence with values on $I_f\times I_g$ and  
  $X_0 \indep (\xi_{k})$.}
    \For{$k=0,1,2,\ldots$}{
            { 
            \begin{equation}\label{eq:spcd}
                X_{k+1}= T^{FB}_{\xi_k}X_k\equiv \Jcal_{g_{\xi^g_k}}\paren{X_{k}-{t}\nabla f_{\xi^f_k}(X_{k})}
            \end{equation}
            }\\
    }
  \caption{Stochastic Forward-Backward Splitting}\label{algo:sfb}
\end{algorithm}

When $f_{\xi^f}(x) = f(x) + \xi^f\cdot x$ and $g_{\xi^g}$ is the zero function, then this is just 
 steepest descents with linear noise discussed in Section \ref{sec:consist RFI}.  More generally, 
 \eqref{eq:spcd} with $g_{\xi^g}$ the zero function models stochastic gradient 
 descents, which is a central algorithmic template in many applications.  We show how the 
 approach developed above opens the door to an analysis of this basic algorithmic paradigm for 
 {\em nonconvex} problems.  

\begin{prop}\label{t:sfb}
On the Euclidean space $(\Rn, \|\cdot\|)$ suppose the following hold:
\begin{enumerate}[(a)]
 \item for all $i\in I_f$, $\nabla f_i$ is Lipschitz continuous with constant $L$ on $G\subset\Rn$ and
 {\em hypomonotone} on $G$ with violation $\tau_f>0$  on $G\subset\Rn$:
\begin{equation}
 \label{e:hypomonotone} 
    -\tau_f\norm{x-y}^2\leq \ip{\nabla f_i(x)-\nabla f_i(y)}{x-y}
 \qquad \forall x, y\in G.
\end{equation}
\item there is a $\tau_g$ such that for all $i\in I_g$, the (limiting) 
subdifferential $\partial  g_i$ satisfies
\begin{equation}\label{e:submonotone}
 -\tfrac{\tau_g}{2}\norm{(x^++z)-(y^++w)}^{2}\leq \ip{z-w}{x^+-y^+}.
\end{equation}%
at all points $(x^+, z)\in \gph\partial  g_i$ and  $(y^+, w)\in \gph\partial  g_i$
where $z = x-x^+$ for $\{x^+\}= \Jcal_{g_i}(x)$ for any 
$x\in \bigcup_{i\in I_f}\paren{\Id - t\nabla f_i}(G)$ and 
where $w = y-y^+$ for $\{y^+\}= \Jcal_{g_i}(y)$ for any 
$y\in \bigcup_{i\in I_f}\paren{\Id - t\nabla f_i}(G)$.
\item $T^{FB}_i$ is a self-mapping on $G\subset\Rn$ for all $i$. 
\end{enumerate}
 Then  the following hold.
 \begin{enumerate}[(i)]
  \item\label{ex:spg ncvx i} $T^{FB}_i$ is a$\alpha$-fne on $G$ with constant 
  $\alpha = 2/3$ and violation at most 
  \begin{equation}\label{eq:sfb violation}
    \epsilon = \max\{0, (1+2\tau_g)\paren{1+t(2\tau_f+2tL^2)} - 1\}
  \end{equation}
  for all $i\in I$.
  \item\label{ex:spg ncvx ii} $\Phi(x, i):=T_{i}x$ is a$\alpha$-fne in expectation 
  on $G$ with 
  constant $\alpha = 2/3$ and violation at most $\epsilon$ given in \eqref{eq:sfb violation}.
  \item\label{ex:spg ncvx iv} The Markov operator $\mathcal{P}$ corresponding to 
  \eqref{eq:spcd} is a$\alpha$-fne in 
  measure on $\mathscr{P}_2(G)$ with constant $\alpha=2/3$ and violation no 
  greater than $\epsilon$ given in \eqref{eq:sfb violation}, 
  i.e. it  satisfies \eqref{eq:alphfne meas}.
  \item\label{ex:spg ncvx iii} Suppose that 
  assumption (a) holds with condition \eqref{e:hypomonotone} being 
  satisfied for  
  $\tau_f<0$ (that is, $\nabla f_i$ is strongly monotone for all $i$), and that  
  condition \eqref{e:submonotone} holds with $\tau_g=0$ (for instance, when $g_i$ is convex).  
  Then,  whenever there exists an invariant measure 
  for the Markov operator $\mathcal{P}$ corresponding to \eqref{eq:spcd} 
  for all step lengths $t\in \left(0,\tfrac{|\tau_f|}{L^2}\right]$, 
  the distributions of the sequences of random variables
  converge to an invariant measure in the Prokhorov-L\`evy metric.  
  \item\label{ex:spg ncvx v} Let $G$ be compact and  $G\cap\inv\mathcal{P}\ne \emptyset$.  
  If  $\Psi$ given by \eqref{eq:Psi} 
 takes the value $0$ only at points in $\inv\mathcal{P}$ and 
 is metrically subregular for 
  $0$ on $\mathscr{P}_2(G)$ with gauge $\rho$ given by \eqref{eq:gauge} 
  with $\tau=1/2$, $\epsilon$ satisfying 
  \eqref{eq:sfb violation}, and  $\theta$
  satisfying \eqref{eq:theta},
 then the Markov chain converges to an invariant distribution with rate given by \eqref{eq:gauge convergence}.  
 \end{enumerate} 
\end{prop}
Before proving the statement, some background for conditions \eqref{e:hypomonotone} and \eqref{e:submonotone} 
might be helpful.  The inequality \eqref{e:hypomonotone} is satisfied by functions $f$ that are 
{\em prox-regular} \cite{PolRock96a}.  This traces back to Federer's study of 
curvature measures \cite{Federer59} where such functions would be called functions whose epigraphs have 
{\em positive reach}. Inequality \eqref{e:submonotone} is equivalent to the property 
that $\Jcal_{g_i}$ is a$\alpha$-fne with constant $\alpha_i=1/2$ and 
violation $\tau_g$ on $G$ \cite[Proposition 2.3]{LukNguTam18}.  Any differentiable function $g_i$ with 
gradient satisfying \eqref{e:hypomonotone} with constant $\tau_g/(2(1+\tau_g))$ 
will satisfy \eqref{e:submonotone} with constant $\tau_g$.  
In the present setting, if 
$g_i$ is {\em prox-regular} on $G$, then  
$\partial {g_i}$ is hypomonotone on $G$ and therefore 
satisfies \eqref{e:submonotone} \cite{LukNguTam18}.  Of course, convex functions 
are trivially hypomonotone with constant $\tau=0$.

\begin{proof}
\eqref{ex:spg ncvx i}.  This is \cite[Proposition 3.7]{LukNguTam18}. 

\eqref{ex:spg ncvx ii}. This follows immediately from Part \eqref{ex:spg ncvx i} above and 
Proposition \ref{r:nonneg psi_Phi}.

\eqref{ex:spg ncvx iv}.  This follows immediately from Part \eqref{ex:spg ncvx ii} above and 
Proposition \ref{thm:Tafne in exp 2 pafne of P}.

\eqref{ex:spg ncvx iii}.  Inserting the assumptions with 
their corresponding constants into the expression for the violation \eqref{eq:sfb violation} shows that 
the $\epsilon$ is zero for all step-lengths $t\in \left(0,\tfrac{|\tau_f|}{L^2}\right]$.  So by part 
\eqref{ex:spg ncvx i} we have that, for all step lengths small
enough, the mappings $T^{FB}_i$ are $\alpha$-fne with constant $\alpha=2/3$.  
(Steps sizes up to twice the upper bound considered here can also be taken, but then the 
constant $\alpha$ approaches $1$.)
If the corresponding 
Markov operator possesses invariant measures, then convergence in distribution of the 
corresponding Markov chain follows from Theorem \ref{thm:a-firm convergence Rn}. 

\eqref{ex:spg ncvx v}.  This follows from Part \eqref{ex:spg ncvx iv} and Theorem \ref{t:msr convergence}.
\end{proof}
The assumptions of Proposition \ref{t:sfb}\eqref{ex:spg ncvx iii} are not unusual.  
What is new is the generality of global convergence.  The compactness assumption on 
$G$ in part \eqref{ex:spg ncvx v} is just to permit the application of Theorem \ref{t:msr convergence}.  
As noted in Remark \ref{r:G compact} this assumption can be dropped for mappings $T_i$ on 
Euclidian space. 
The result narrows the work of proving convergence of stochastic forward-backward 
algorithms to verifying existence of $\inv \Pcal$.
The next corollary shows how this
is done for the special case of stochastic gradient descent. 
\begin{cor}[stochastic gradient descent] 
  In problem \eqref{eq:opt prob} let  $g_i(x)\equiv 0$ for all $i$ at each $x$.
  In addition to assumptions of Proposition \ref{t:sfb}, assume that 
  \begin{enumerate}[(i)]
   \item the expectation $\mathbb{E}\ecklam{f_{\xi }(x)}$ 
  attains a minimum at $\xbar\in G\subset \Rn$ with value 
  $\mathbb{E}\ecklam{f_{\xi }(\xbar)}= \bar p$;
  \item $\mathbb{E}\left[\norm{X_{0} -\bar x}^{2}\right]$ exists where $X_0$ is 
  a random variable with distribution $\mu_0\in \mathscr{P}_2(G)$;
  \item $\nabla f_i$ is strongly 
    monotone with constant $|\tau_f|$ for all $i$ 
    (that is condition \eqref{e:hypomonotone} is satisfied with 
    $\tau_f<0$).  
  \end{enumerate}
  Then Algorithm \ref{algo:sfb} is the stochastic gradient descent algorithm and, 
  for a fixed step length 
  $t\in \left(0,\tfrac{|\tau_f|}{L^2}\right]$, when initialized with $\mu_0$,
  the distributions of the iterates coverge in the Prokhorov-L\`evy 
  metric to an invariant measure of the 
  corresponding Markov operator.  Moreover, whenever $\Psi$ given by \eqref{eq:Psi} 
  takes the value $0$ only at points in $\inv\mathcal{P}$ and 
  is metrically subregular for 
  $0$ on $\mathscr{P}_2(G)$ with gauge $\rho$ given by \eqref{eq:gauge} 
  with $\tau=1/2$, $\epsilon_i\leq \epsilon$ for all $i$ with $\epsilon$ 
  given by \eqref{eq:sfb violation} and  $\theta$
  satisfying \eqref{eq:theta},
  then the Markov chain converges to a point in $\inv\Pcal$ 
  with rate given by \eqref{eq:gauge convergence}.  
\end{cor}
\begin{proof}
  In this case $T^{FB}_i\equiv \Id - t\nabla f_i$.   
  If we can show that the 
  corresponding Markov operator posseses invariant measures, then the statement 
  follows from  
  Proposition \ref{t:sfb}.
  
  To establish existence of invariant distributions, 
  note that  
  \begin{align*}
    \norm{X_{k+1} - \bar x}^{2} =& 
    \norm{X_{k} - {t} \nabla f_{\xi _{k}}(X_{k}) - \bar x}^{2} - 
      \norm{X_{k+1} - 
    X_{k} - {t} \nabla f_{\xi_{k}}(X_{k})}^{2} \\
    &= \norm{X_{k}-\bar x}^{2} - \norm{X_{k+1}-X_{k}}^2 
    -2{t}\act{\nabla f_{\xi_{k}}(X_{k}), X_{k+1} -X_{k}+X_{k}-\bar x}.
  \end{align*}
  For functions with Lipschitz continuous gradients the following {\em growth condition}
  holds
  \begin{align*}
    \act{\nabla f_{\xi_{k}}(X_{k}),X_{k+1}-X_{k}} \ge
    f_{\xi_{k}}(X_{k+1})-f_{\xi_{k}}(X_{k}) 
-\frac{L}{2}\norm{X_{k+1}-X_{k}}^{2}.
  \end{align*}
  Interchanging the gradient and the expectation in $\mathbb{E}\ecklam{\nabla f_{\xi}(x)}$ 
  together with strong monotonicity of the gradients $\nabla f_i$ yields
  \begin{align*}
    \act{\nabla \mathbb{E}\ecklam{f_{\xi}(X_{k}), X_{k} - \bar x}} \ge 
    \mathbb{E}\ecklam{f_{\xi}(X_{k})}-\bar p + 
    \frac{|\tau_f|}{2}\norm{X_{k}-\bar x}^{2}.
  \end{align*}
  It follows that 
  \begin{align*}
    \mathbb{E}\left[\norm{X_{k+1} - \bar x}^{2}\right] &\le
    (1-{t}|\tau_f|)\mathbb{E}\left[\norm{X_{k} - \bar x}^{2} \right] -
    (1-{t} L) \mathbb{E}\left[\norm{X_{k+1} - X_{k}}^{2}\right] -
    2{t} (\mathbb{E}[f_{\xi_{k}}(X_{k+1})] - \bar p \\ 
    & \le     (1-{t}|\tau_f|)\mathbb{E}\left[\norm{X_{k} - \bar x}^{2}
    \right] +2 {t} \bar p.
  \end{align*}
  This yields
  \begin{align*}
    \mathbb{E}\left[\norm{X_{k} - \bar x}^{2} \right]
    \le (1-{t}|\tau_f|)^{k} \mathbb{E}\left[\norm{X_{0} - \bar x}^{2}
    \right] + 2|\tau_f| \bar p  \sum_{i=0}^{k-1} (1-{t} |\tau_f|)^{i} \le
    \mathbb{E}\left[\norm{X_{0} - \bar x}^{2} \right] + 
    \frac{2\bar p }{|\tau_f|}.
  \end{align*}
  So the sequence $\left( \mathbb{E}\left[\norm{X_{k} - \bar x}^{2} \right]
  \right)_{k\in\Nbb}$ is bounded if $\mathbb{E}\left[\norm{X_{0} -
      \bar x}^{2}\right]$ exists and hence the existence of an
  invariant measure follows from Theorem \ref{thm:ex_inRn}.
\end{proof}
  Note that the step size in the stochastic gradient could be large enough that the 
  gradient descent mapping is {\em expansive}, even though the gradient is 
  assumed to be strongly monotone.  This explains the frequent observation that  
  descent methods perform well even for step sizes larger than the usual analysis
  recommends.

\subsection{Stochastic Douglas-Rachford and X-FEL Imaging}  
Another prevalent algorithm for nonconvex problems is the 
Douglas-Rachford algorithm \cite{LionsMercier79}.  This is
based on compositions of {\em reflected resolvents}:
\begin{equation}\label{eq:Rprox}
 R_{f}\equiv 2\Jcal_f -\Id.
\end{equation}

  \begin{algorithm}    
\SetKwInOut{Output}{Initialization}
  \Output{Set $X_{0} \sim \mu_0 \in \mathscr{P}_2(G)$, 
  $X_0\sim \mu$,   
   and $(\xi_{k})_{k\in\Nbb}$ 
  another i.i.d.\ sequence with $\xi_{k} = (\xi^f_{k}, \xi^g_{k})$ 
  taking values on $I_f\times I_g$ and  
  $X_0 \indep (\xi_{k})$.}
    \For{$k=0,1,2,\ldots$}{
            { 
            \begin{equation}\label{eq:sdr}
                X_{k+1}= T^{DR}_{\xi_k}X_k\equiv 
                \frac{1}{2}\paren{R_{f_{\xi^f_k}}\circ R_{g_{\xi^g_k}} + \Id}(X_{k})
            \end{equation}
            }\\
    }
  \caption{Stochastic Douglas-Rachford Splitting}\label{algo:sdr}
\end{algorithm}
Algorithm \ref{algo:sdr} has been studied for solving large-scale, convex optimization 
and monotone inclusions (see for example \cite{Pesquet19, Cevher2018}). The result below 
opens the analysis to nonconvex, nonmonotone problems. 
\begin{prop}\label{t:sdr}
On the Euclidean space $(\Rn, \|\cdot\|)$, suppose the following hold:
\begin{enumerate}[(a)]
 \item there is a $\tau_g$ such that for all $i\in I_g$, the (limiting) 
subdifferential $\partial  g_i$ satisfies
\begin{equation}\label{e:hypomonotone g}
 -\tfrac{\tau_g}{2}\norm{(x^++z)-(y^++w)}^{2}\leq \ip{z-w}{x^+-y^+}.
\end{equation}%
at all points $(x^+, z)\in \gph\partial  g_i$ and  $(y^+, w)\in \gph\partial  g_i$
where $z = x-x^+$ for $\{x^+\}= \Jcal_{g_i}(x)$ for any $x\in G\subset \Rn$ and 
where $w = y-y^+$ for $\{y^+\}= \Jcal_{g_i}(y)$ for any $y\in G$.
 \item there is a $\tau_f$ such that for all $i\in I_f$, the (limiting) 
subdifferential $\partial  f_i$ satisfies
\begin{equation}\label{e:hypomonotone f}
 -\tfrac{\tau_f}{2}\norm{(x^++z)-(y^++w)}^{2}\leq \ip{z-w}{x^+-y^+}.
\end{equation}%
at all points $(x^+, z)\in \gph\partial  f_{i}$ and  $(y^+, w)\in \gph\partial  f_{i}$
where $z = x-x^+$ for $\{x^+\}= \Jcal_{f_{i}}(x)$ for any $x\in \bigcup_{j\in I_g}\{\Jcal_{g_{j}}(G)\}$ and 
where $w = y-y^+$ for $\{y^+\}= \Jcal_{f_{i}}(y)$ for any $y\in \bigcup_{j\in I_g}\Jcal_{g_{j}}(G)$.
\item $T^{DR}_i$ is a self-mapping on $G\subset\Rn$ for all $i$. 
\end{enumerate}
 Then  the following hold.
 \begin{enumerate}[(i)]
  \item\label{ex:sdr ncvx i} For all $i\in I_f\times I_g$ the mapping $T^{DR}_i$ 
  defined by \eqref{eq:sdr} 
  is a$\alpha$-fne on $G$ with constant 
  $\alpha = 1/2$ and violation at most 
  \begin{equation}\label{eq:sdr violation}
    \epsilon = \tfrac{1}{2}\paren{(1+ 2\tau_g)(1+2\tau_f) - 1}
  \end{equation}
  on $G$.
  \item\label{ex:sdr ncvx ii} $\Phi(x, i):=T^{DR}_{i}x$ is a$\alpha$-fne in expectation with 
  constant $\alpha = 1/2$ and violation at most $\epsilon$ given by \eqref{eq:sdr violation}.
  \item\label{ex:sdr ncvx iv} The Markov operator $\mathcal{P}$ corresponding to 
  \eqref{eq:sdr} is a$\alpha$-fne in 
  measure with constant $\alpha=1/2$ and violation no greater than $\epsilon$ given by 
  \eqref{eq:sdr violation}, 
  i.e. it  satisfies \eqref{eq:alphfne meas}.
  \item\label{ex:sdr ncvx iii} Suppose that 
  assumptions (a) and (b) hold with conditions \eqref{e:hypomonotone g} and \eqref{e:hypomonotone f} 
  being satisfied for  
  $\tau_g=\tau_f=0$ (i.e., when $f_i$ and $g_i$ are convex for all $i$).  
  Then,  whenever there exists an invariant measure 
  for the Markov operator $\mathcal{P}$ corresponding to \eqref{eq:sdr}, 
  the distributions of the sequences of random variables
  converge to an invariant measure in the Prokhorov-L\`evy metric.  
  \item\label{ex:sdr ncvx v} Let $G$ be compact and  $G\cap\inv\mathcal{P}\ne \emptyset$.
  If $\Psi$ given by \eqref{eq:Psi} 
 takes the value $0$ only at points in $\inv\mathcal{P}$ and 
 is metrically subregular for 
  $0$ on $\mathscr{P}_2(G)$ with gauge $\rho$ given by \eqref{eq:gauge} 
  with $\tau=1/2$, $\epsilon$ satisfying 
  \eqref{eq:sdr violation}, and  $\theta$
  satisfying \eqref{eq:theta},
 then the Markov chain converges to an invariant distribution with rate given by \eqref{eq:gauge convergence}.  
 \end{enumerate} 
\end{prop}
\begin{proof}
\eqref{ex:sdr ncvx i}.  By \cite[Proposition 3.7]{LukNguTam18} for all $j\in I_g$, 
$\Jcal_{g_j}$ is a$\alpha$-fne 
with constant $\alpha=1/2$ and violation $\epsilon_g = 2\tau_g$ on $G$.  
Likewise, for all $i\in I_f$,  $\Jcal_{f_i}$  
is a$\alpha$-fne 
with constant $\alpha=1/2$ and violation $\epsilon_f = 2\tau_f$ on 
$\bigcup_{j\in I_g}\{\Jcal_{g_{j}}(G)\}$.
By \cite[Propositions 2.3-2.4]{LukNguTam18}, for all 
$i\in I_f\times I_g$ the Douglas-Rachford mapping $T^{DR}_i$ is therefore 
a$\alpha$-fne with constant $\alpha=1/2$ and 
violation at most $\tfrac{1}{2}\paren{(1+2\tau_g)(1+2\tau_f) - 1}$ 
on $G$.  

\eqref{ex:sdr ncvx ii} - \eqref{ex:sdr ncvx v} follow in the same way as their 
conterparts in Proposition \ref{t:sfb}.
\end{proof}
Here as in Proposition \ref{t:sfb} the compactness assumption on $G$ 
in part \eqref{ex:sdr ncvx v} can be dropped
since $T_i$ is a mapping on $\Rn$ (see Remark \ref{r:G compact}). 

\subsubsection{Application to X-FEL Imaging}
For nonconvex problems, the Douglas-Rachford algorithm is popular because  
its set of fixed points is often smaller than other 
popular algorithms  \cite[Theorem 3.13]{LukMar20}.  We briefly discuss its application to 
the problem of X-ray free electron laser imaging, for which the analytical 
framework establish here is intended.  

Here, a high-energy X-ray pulse illuminates molecules suspended in fluid.  A 
two dimensional, low-count diffraction image is recorded for each pulse.  
The goal is to reconstruct the three-dimensional electron density of the target molecules 
from the observed two-dimensional diffraction images (on the order of $10^9$).  This is a stochastic 
tomography problem with a nonlinear model for the data - stochastic because the molecule 
orientations are random, and uniformly distributed on SO(3).  
Computed tomography with random orientations has been studied for more than two decades 
\cite{Bresler2000a,Bresler2000b} and 
been successfully applied for inverting the Radon transform (a linear operator) with unknown orientations 
\cite{Panaretos09, SingerWu13}.  
The model for the data in X-FEL imaging is nonlinear and nonconvex: Fraunhoffer diffraction with missing phase 
\cite{Born, ArdGru20}.  The problem of recovering a 2-dimensional slice of the object from diffraction 
intensity data is the optical {\em phase retrieval} problem \cite{Rayleigh92, RobSal20}.  The most successful 
and widely applied methods for solving this problem are fixed point algorithms where the 
fixed point mappings consist of compositions and averages of projection mappings onto nonconvex sets 
\cite{LukSabTeb18}.  A theoretical framework for unifying and extending the first proofs of local convergence 
of these methods, with rates, was established in \cite{LukNguTam18}.  This analysis accommodates mappings
that not only are not contractions, but are actually {\em expansive}.  Moreover, unlike many other approaches, 
the framework does not require that the constituent mappings have common fixed points.  
This has been applied to prove, for the first time, local linear convergence of 
a wide variety of fundamental algorithms for phase retrieval \cite{LukMar20, Thao18, HesseLuke13, LukNguTam18}. 

The goal of X-FEL imaging is to determine the {\em electron density} of a molecule from 
experimental samples of its scattering probability distribution.  
The physical model for the experiment is
	\begin{equation} \label{eq:physical model}
		\left\|\paren{\Fcal(\rho)}_{i}\right\| = \phi_{i}, \quad 
		\forall \,\, i = 1 , 2 , \ldots , n.
	\end{equation}
	Here $\mymap{\Fcal}{\mathbb{C}^{n}}{\mathbb{C}^{n}}$ is a unitary linear operator accounting for the propagation 
	of an electromagnetic wave, $\rho \in \mathbb{C}^{n}$ is the unknown electron density that interacts with the 
	wave at one end (the {\em pupil or object plane}) of the instrument, and 
	$\phi_{i} \in \mathbb{R}_{+}$ is the probability of observing a scattered photon in the 
	$i$'th {\em voxel} ($i = 1 , 2 ,\ldots , n$) of the imaging volume. The problem is to determine 
	$\rho$ from $\phi_{\cdot}$.  The set of possible vectors satisfying such measurements 
	is given by
	\begin{equation}\label{eq:C}
		C \equiv \set{\rho \in \mathbb{C}^{n}}{\left\|\paren{\Fcal(\rho)}_{i}\right\| = 
		\phi_{i}, \quad \forall \,\, i = 1 , 2 , \ldots , n}.
	\end{equation}
	Although this set is nonconvex, it is prox-regular \cite{Luke12}. 
	To give an idea of the size of this problem, 
	in a typical experiment $n=O(10^9)$.  
	
	For this model, we have measurements $Y_\xi$ where $\xi$ is a uniformly distributed random variable 
	which takes values on SO(3).  In an X-FEL experiment, 
	$Y_\xi$ is a two-dimensional measurement of photon counts (so, real and nonnegative) 
	on a plane $H_\xi$ passing through the origin in the domain of  $\phi$ with orientation $\xi$.  
	The value of $\xi$ is in fact not observable, however, by observing three-electron correlations, 
	this can be estimated \cite{ArdGru, ArdGru20}.  The set $C$ above is then replaced by the random set 
    \begin{equation}\label{eq:CYhat}
     C(\xi)\equiv  \set{\rho \in \mathbb{C}^{n}}{\left\|\paren{\Fcal(\rho)}_{i}\right\| = 
     (Y_\xi)_{i}, \quad \mbox{at voxels $i$ intersecting } H_\xi}.
    \end{equation}
	In addition to the random sets generated by the data, there are certain a priori qualitative constraints that can 
	(and should) be added depending on the type of experiment that has been conducted. Often these are 
	support constraints, or real-valuedness, or nonnegativity.  All of these are convex constraints for which
	we reserve the set $C_{0}$ for the qualitative constraints.
	
    The problem is a specialization of \eqref{eq:opt prob} where $I_f=\{1\}$, 
    $I_g=SO(3)$, $\xi^f_k=1$ for all $k$, $\xi^g_k$ is a 
	uniformly distributed random variable on SO(3) for all $k$, and 
	\begin{eqnarray} 
	 (\forall k)\quad f_{\xi^f_k}(\rho)&\equiv& \tfrac{\lambda}{2(1-\lambda)}\dist^2(\rho, C_0)\nonumber\\
	 g_{\xi^g_k}(\rho)&\equiv& \iota_{C(\xi^g_k)}(\rho)\equiv \begin{cases} 0&\mbox{ if } \rho\in C(\xi^g_k)\\
	                                                     +\infty& \mbox{ otherwise}.
	                                                    \end{cases}
	\end{eqnarray}

	The algorithm we propose for this problem is Algorithm \ref{algo:sdr}.	
	Assumptions (a) and (b) of Proposition \ref{t:sdr} are easily verified, and in fact, for this 
	application $\tau_f=0$ since $C_0$ is convex.  In \cite{LukMar20} the fixed points of the 
	deterministic version of Algorithm \ref{algo:sdr} for the simpler, two-dimensional 
	{\em phase retrieval problem} have been characterized, and metric subregularity of the
	transport discrepancy \eqref{eq:delta} has been determined for geometries applicable 
	to {\em cone and sphere} problems \cite{LukSabTeb18} such as this.  So 
	for a majority of relevant instances, there is good reason to expect that 
	Proposition \ref{t:sdr} can be  
	applied provably to X-FEL measurements.  The determination of the domain $G$ in condition (c) 
	of Proposition \ref{t:sdr}
	is therefore key.   
	There are some unresolved cases, however, that 
	are relevant for optical phase retrieval (see \cite[Example 5.4]{LukMar20}), 
	and this needs further study.  
	
%
%

\subsection{Inconsistent set feasibility}\label{ex:2linearSpaces_withInvMeas}
We conclude this study with our explanation for the numerical behavior
observed in Fig. \ref{fig:Axb}.  This is an affine feasibility 
problem:
\begin{equation}\label{eq:feas}
 \mbox{Find}\quad x\in L\equiv \cap_{j\in I} \{x~|~\ip{a_j}{x}=b_j\}.
\end{equation}
When the intersection is empty
we say that the problem is {\em inconsistent}.  
Consistent or not, we apply the method of cyclic projections \eqref{eq:CP}. 
Even though the projectors onto the corresponding problems have an analytic 
expression, this representation can only be evaluated to finite precision.  The 
trick here is to view the algorithm not as inexact cyclic projections onto
deterministic hyperplanes, but rather as {\em exact} projections onto 
randomly selected hyperplanes.  

Indeed, consider the following generalized affine noise model for a single
  affine subspace: $H^{(\xi,\zeta)}_\xbar = 
 \mysetc{x \in \mathbb{R}^{n}}{\act{a+\xi, x-\xbar} = \zeta}$, 
 where $a \in \mathbb{R}^n$  
    and $\xbar$ satisfies $A\xbar=b$ for a given $b \in \mathbb{R}$ and noise 
  $(\xi,\zeta) \in \mathbb{R}^{n}\times\mathbb{R}$ is independent.  
  The key conceptual distinction is that 
  the analysis proceeds with {\em exact} projections onto randomly selected
  hyperplanes $H^{(\xi,\zeta)}_\xbar$, rather than working with inexact projections 
  onto deterministic hyperplanes.    

\begin{prop}\label{t:linear convergence affine feas}
Given $a\in \Rn$, $b\in \mathbb{R}$, define the hyperplane $H = \set{y}{\ip{a}{y}=b}$ and  
fix $\xbar\in H$.  Define the random mapping 
   $\mymap{T_{(\xi,\zeta)}}{\Rn}{\Rn}$ by 
     \begin{align*}
    T_{(\xi,\zeta)}x\equiv P_{H^{(\xi,\zeta)}_\xbar}x &= x - 
    \frac{\act{a+\xi,x-\xbar} - \zeta}{\norm{a+\xi}^{2}} (a+\xi)
  \end{align*}
  where $(\xi,\zeta) \in \mathbb{R}^{n}\times\mathbb{R}$  
  is a vector of independent random variables satisfying 
\begin{subequations}\label{eq:noise assump}
  \begin{align}
    d&:= \mathbb{E} \left[
      \frac{(b+\zeta)^{2}}{\norm{a + \xi}^{2}}\right] < \infty, 
      \label{eq:noise assump zeta}\\
    c&:= \inf_{\substack{z \in \mathbb{S}}}
    \mathbb{E}\left[\frac{\act{a+\xi, z}^{2}}{\|a+\xi\|^2} \right] > 0
    \label{eq:noise assump xi}
  \end{align}
  where $\mathbb{S}$ is the set of unit vectors in $\Rn$. 
\end{subequations}
Algorithm \eqref{algo:RFI} with this random function initialized with any 
$\Rn$-valued random variable $X_0$ with distribution 
   $\mu^0\in\mathscr{P}(\Rn)$ converges Q-linearly to a unique invariant distribution.  
\end{prop}

\begin{proof}
  Each mapping $T_{(\xi,\zeta)}$ is the orthogonal projector onto the 
  hyperplane $H^{(\xi,\zeta)}_\xbar$, and so 
  is $\alpha$-fne with constant $\alpha=1/2$ (no violation).   
  It follows immediately from the definition, then, that this is both 
  nonexpansive in expectation and   $\alpha$-fne in 
  expectation with $\alpha=1/2$.  
  By Proposition \ref{thm:Tafne in exp 2 pafne of P}
  the corresponding Markov operator $\mathcal{P}$ satisfies 
  \eqref{eq:alphfne meas}, provided $\inv\mathcal{P}\neq\emptyset$.
  We will show below, that there do indeed exist invariant measures, but 
  for the moment, let us just assume this holds.
  
Since the randomly selected projectors are 
mappings on a Hilbert space, using the identity 
\eqref{eq:nice ineq} the surrogate function $\Psi$
on the space of measures $\mathscr{P}(\Rn)$ defined by 
\eqref{eq:Psi} can be written as
\[
 \Psi(\mu) = \inf_{\pi\in\inv\mathcal{P}}\inf_{\gamma\in C_*(\mu,\pi)}
 \paren{\int_{\Rn\times \Rn}\mathbb{E}_{(\xi,\zeta)}
 \left[ \norm{(x-T_{(\xi,\zeta)x})-(y-T_{(\xi,\zeta)}y)}^2 \right]
 \gamma(dx, dy)}^{1/2}. 
\]
where, recall,  
$C_*(\mu,\pi_\mu)$ is the set of optimal couplings, that is the set of couplings where  
$W_2(\mu,\pi_\mu)$ is attained.  An elementary calculation shows that 
\begin{equation}\label{eq:cosine}
 \norm{(x-T_{(\xi,\zeta)}x)-(y-T_{(\xi,\zeta)}y)}^2 = \left\langle \tfrac{x-y}{\|x-y\|}, \tfrac{a+\xi}{\|a+\xi\|}\right\rangle^2\|x-y\|^2.
\end{equation}
Now, we use the assumptions on the random variables $\xi$ and 
$\zeta$ in \eqref{eq:noise assump}.
Condition \eqref{eq:noise assump xi} is satisfied for example when $\xi$ is isotropic or
  radially symmetric; condition \eqref{eq:noise assump zeta} when $\zeta$ has bounded variance
  and $\norm{\xi}$ is bounded away from $1$.  (Physically, you would interpret
  this as the noise being bounded away from the signal in energy.)  
Taking the expectation and using the
  assumption on the noise yields
  \begin{eqnarray}
 \Psi^2(\mu) &=& \int_{\Rn\times \Rn}\mathbb{E}_{(\xi,\zeta)}
 \left[ \norm{(x-T_{(\xi,\zeta)x})-(y-T_{(\xi,\zeta)}y)}^2 \right]
 \gamma(dx, dy)\quad \gamma\in C_*(\mu,\pi_\mu)\nonumber\\
 &\geq&c  \int_{\Rn\times \Rn} \norm{x-y}^2
 \gamma(dx, dy)\quad \gamma\in C_*(\mu,\pi_\mu)\nonumber\\
 &=& c W^2_2(\mu, \pi_\mu).
 \label{eq:msr cp}  
  \end{eqnarray}
With this, we have shown that $\Psi$ is linearly metrically subregular
for $0$ on $\mathscr{P}(\Rn)$ with constant $1/\sqrt{c}$ where 
$c$ is given by \eqref{eq:noise assump xi}.  By Remark \ref{r:G compact}, 
we can then apply Corollary \ref{t:msr convergence - linear} to conclude that, 
if an invariant measure exists, the random function iteration of 
repeatedly projecting onto the randomly selected affine subspaces 
converges R-linearly with respect to the Wasserstein metric 
to an invariant measure. 
  
 It remains to show that an invariant measure exists and is unique.  This follows
 from the observation that a much stronger property holds, namely that
 $\Phi$ is a contraction in expectation.  Indeed, the same 
 calculation behind \eqref{eq:cosine} shows that  
  \begin{align*}
    \norm{T_{(\xi,\zeta)}x - T_{(\xi,\zeta)}y}^{2} = 
    \left(1-\cos^{2}\left(\tfrac{a+\xi}{\|a+\xi\|},\tfrac{x-y}{\|x-y\|}\right)\right) \norm{x-y}^{2}.
  \end{align*}
Again, taking the expectation over $(\xi, \zeta)$ yields 
  \begin{align*}
    \mathbb{E}\left[\norm{T_{(\xi,\zeta)} x - T_{(\xi,\zeta)}
        y}^{2}\right] \le (1-c) \norm{x-y}^{2}.
  \end{align*}
  From Theorem \ref{thm:contraInExpec} we get that there exists a
  {\em unique} invariant measure $\pi_0$ for $\mathcal{P}$ (even $\pi_0 \in
  \mathscr{P}_{2}$) and that it satisfies
  \begin{align*}
    W_{2}^{2}(\mu\mathcal{P}^{k},\pi_0) \le (1-c)^{k}
    W_{2}^{2}(\mu,\pi_0).
  \end{align*}
  Convergence is therefore Q-linear, not just R-linear.  
  \end{proof}
  
  Note that the noise satisfying \eqref{eq:noise assump} depends implicitly on the 
  point $\xbar$, which 
  will determine the concentration of the invariant distribution of the Markov operator.  This 
  corresponds to the fact that the exact projection, while unique, depends on the 
  point being projected.  One would expect the invariant distribution to be concentrated
  on the exact projection.  
  
  Extending this model to finitely many distorted
  affine subspaces as illustrated in Fig. \ref{fig:Axb} 
  (i.e.\ we are given $m$ normal vectors $a_{1},
  \ldots, a_{m} \in \mathbb{R}^{n}$ and displacement vectors $b_{1},
  \ldots, b_{m}$) yields a stochastic 
  version of cyclic projections \eqref{eq:CP} which converges Q-linearly 
  (geometrically) in the Wasserstein metric to a unique invariant measure for 
  the given noise model. 
  
  Indeed, for a collection of (not necessarily distinct) points 
  $\xbar_j\in H_j\equiv \set{y}{\ip{a_j}{y}=b_j}$ ($j=1,2,\dots,m$) denote  by 
$P^{j}_{(\xi_j,\zeta_j)}$ the  exact projection onto the
  $j$-th random affine subspace centered on $\xbar_j$, i.e.\
  \begin{align*}
    P^{j}_{(\xi_j,\zeta_j)} x= x - \frac{\act{a_{j}+\xi_{j},x-\xbar_j} -
   \zeta_{j}}{\norm{a_{j}+\xi_{j}}^{2}} (a_{j}+\xi_{j}),
  \end{align*}
  where $(\xi_{i})_{i=1}^{m}$ and $(\zeta_{i})_{i=1}^{m}$ are
  i.i.d.\ and $(\xi_{i}) \indep (\zeta_{i})$. The stochastic cyclic projection
  mapping is 
\begin{align*}
    T_{(\xi,\zeta)} x = P^{m}_{(\xi_m,\zeta_m)} \circ \ldots 
    \circ P^{1}_{(\xi_1,\zeta_1)} x, \qquad x \in
    \mathbb{R}^{n}
  \end{align*}
  where 
  $(\xi, \zeta) = ((\xi_m, \xi_{m-1}, \dots,\xi_1),(\zeta_m, \zeta_{m-1},\dots,\zeta_1))$.
Following the same pattern of proof as Proposition \ref{t:linear convergence affine feas} 
we see that $T_{(\xi,\zeta)}$
is a contraction in expectation:
  \begin{align*}
    \mathbb{E}\left[ \norm{T_{(\xi,\zeta)} x - T_{(\xi,\zeta)}
        y}^{2}\right] \le (1-c)^{m} \norm{x-y}^{2}
  \end{align*}
  where 
  \begin{align}
    c&:= \min_{j=1,\dots,m}\inf_{\substack{z \in \mathbb{S}}}
    \mathbb{E}\left[ \frac{\act{a_j+\xi_j,z}^{2}}%
    {\norm{a_j + \xi_j}^{2}} \right] > 0.
    \label{eq:noise assump xi2}
  \end{align}

  Hence, there exists a unique invariant measure and $(\mu\mathcal{P}^{k})$ 
  converges geometrically to it in the $W_{2}$ metric.

\appendix

\section{Appendix}
\label{sec:toolbox}

\begin{prop}[Convergence with subsequences]\label{thm:cvg_subsequences}
  Let $(G,d)$ be a metric space. Let $(x_{k})$ be a sequence on $G$
  with the property that any subsequence has a convergent subsequence
  with the same limit $x\in G$. Then $x_{k} \to x$.
\end{prop}
\begin{proof}
  Assume that $x_{k} \not \to x$, i.e.\ there exists $\epsilon>0$ such
  that for all $N \in \mathbb{N}$ there is $k=k(N) \ge N$ with $d(x_{k},x)
  \ge \epsilon$. But by assumption the subsequence $(x_{k(N)})_{N \in
    \mathbb{N}}$ has a convergent subsequence with limit $x$, which is
  a contradiction and hence the assumption is false.
\end{proof}
\begin{rem}
  In a compact metric space, it is enough, that all cluster points are
  the same, because then every subsequence has a convergent
  subsequence.
\end{rem}

\begin{lemma}
  Let $(G,d)$ be a separable complete metric space and let the metric 
$d_{\times}$ on  $G\times G$ satisfy
  \begin{align}\label{eq:prodMetricCondition}
    d_{\times}\left( \icol{x_{k}\\ y_{k}}, \icol{x \\ y}\right)
    \to 0  && \Leftrightarrow && d(x_{k},x) \to 0 \quad\text{ and }\quad
    d(y_{k},y) \to 0.
  \end{align}
Then $\mathcal{B}(G\times G) = \mathcal{B}(G) \otimes
  \mathcal{B}(G)$.
\end{lemma}
\begin{proof}
  First we note that for $A,B \subset G$ it holds that $A\times B$ is
  closed in $(G\times G,d_{\times})$ if and only if $A,B$ are closed
  in $(G,d)$ by \eqref{eq:prodMetricCondition}. Since the
  $\sigma$-algebra $\mathcal{B}(G) \otimes \mathcal{B}(G)$ is
  generated by the family $\mathcal{A}:=\mysetc{A_{1}\times
    A_{2}}{A_{1},A_{2}\subset G \text{ closed}}$. One has 
    $\mathcal{B}(G\times G) \supset \mathcal{B}(G) \otimes
  \mathcal{B}(G)$

  For the other direction, note
  that any metric $d_{\times}$ with the property
  \eqref{eq:prodMetricCondition} has the same open and closed sets. If
  $A$ is closed in $(G\times G,d_{\times})$ and $\tilde d_{\times}$ is
  another metric on $G\times G$ satisfying
  \eqref{eq:prodMetricCondition}, then for $(a_{k},b_{k})\in A$ with
  $(a_{k},b_{k}) \to (a,b) \in G\times G$ w.r.t.\ $\tilde d_{\times}$
  it holds that $d(a_{k},a)\to 0$ and $d(b_{k},b)\to 0$ and hence $
  d_{\times}( (a_{k},b_{k}),(a,b)) \to 0$ as $k\to\infty$, i.e.\ $(a,b)\in A$, so $A$
  is closed in $(G\times G,\tilde d_{\times})$. It follows that all
  open sets in $(G\times G, d_{\times})$ are the same for any metric
 that satisfies \eqref{eq:prodMetricCondition}. 
 So, without loss of generality,  let 
  \begin{align}\label{eq:productMaxNorm}
    d_{\times}\left( \icol{x_{1}\\y_{1}}, \icol{x_{2}\\ y_{2}} \right) &=
    \max(d(x_{1},x_{2}),d(y_{1},y_{2})).
  \end{align}
Moreover, 
  separability of $G\times G$ yields that any open set is the
  countable union of balls: there exists a sequence $(u_{k})$
 on $U$ that is dense for $U \subset G\times G$ open. We can find a sequence of 
constants  $\epsilon_{k}>0$ with $\bigcup _{k} \mathbb{B}(u_{k},\epsilon_{k})
  \subset U$. If there exists $x \in U$, which is not covered by any
  ball, then we may enlarge a ball, so that $x$ is covered: since
  there exists $\epsilon>0$ with $\mathbb{B}(x,\epsilon) \subset U$
  and there exists $m \in \mathbb{N}$ with $d(x,u_{m})< \epsilon/2$ by
  denseness, we may set $\epsilon_{m} = \epsilon/2$ to get $x \in
  \mathbb{B}(u_{m},\epsilon_{m}) \subset \mathbb{B}(x,\epsilon)\subset
  U$.  Now to continue the proof, let $d_{\times}$ be given by
  \eqref{eq:productMaxNorm}. Then for any open $U \subset G\times G$
  there exists a sequence $(u_{k})$ on $U$ and  a corresponding sequence of 
  positive constants $(\epsilon_{k})$ such that $U =
  \bigcup_{k} \mathbb{B}(u_{k},\epsilon_{k})$.  This together with 
  the fact that 
  \begin{align*}
    \mathbb{B}(u_{k},\epsilon_{k}) = \mathbb{B}(u_{k,1},\epsilon_{k})
    \times \mathbb{B}(u_{k,2},\epsilon_{k}) \in \mathcal{B}(G) \otimes
    \mathcal{B}(G) \quad (u_{k}=(u_{k,1},u_{k,2}) \in G\times G) 
  \end{align*}
yields
      $\mathcal{B}(G\times G) \subset \mathcal{B}(G) \otimes
  \mathcal{B}(G),$
 which establishes equality of the $\sigma$-algebras.
\end{proof}

\begin{lemma}[couplings]\label{lemma:couplings}
  Let $G$ be a Polish space and let $\mu,\nu \in
  \mathscr{P}(G)$. Let $\gamma \in C(\mu,\nu)$, where
  \begin{align}
    \label{eq:couplingsDef}
    C(\mu,\nu) := \mysetc{\gamma \in \mathscr{P}(G\times G)}{ \gamma(A
      \times G) = \mu(A), \, \gamma(G\times A) = \nu(A) \quad \forall A
      \in \mathcal{B}(G)},
  \end{align}
  then
  \begin{enumerate}[(i)]
  \item\label{item:couplings1} $\Supp \gamma \subset \Supp \mu \times \Supp 
\nu$,
  \item\label{item:couplings2} $\overline{\mysetc{x}{(x,y) \in \Supp \gamma}} = 
\Supp \mu$.
  \end{enumerate}
\end{lemma}
\begin{proof}
  We let the product space be equipped with the metric in
  \eqref{eq:productMaxNorm} (constituting a separable complete metric space 
since 
  $G$ is Polish).
  \begin{enumerate}[(i)]
  \item Suppose $(x,y) \in \Supp \gamma$ and let $\epsilon>0$, then
    \begin{align*}
      \mu(\mathbb{B}(x,\epsilon)) = \gamma (\mathbb{B}(x,\epsilon)
      \times G) \ge \gamma(\mathbb{B}(x,\epsilon)\times
      \mathbb{B}(y,\epsilon)) = \gamma(\mathbb{B}( (x,y),\epsilon))>0.
    \end{align*}
    Analogously, we have $\nu(\mathbb{B}(y,\epsilon))>0$. So
    $(x,y)\in\Supp \mu \times \Supp \nu$.
  \item Suppose $x \in \Supp \mu$, then $\gamma(\mathbb{B}(x,\epsilon)
    \times G) >0$ for all $\epsilon>0$.  Since $G$ is Polish, the support of the measure is 
    nonempty whenever the measure is nonzero, and (again, since $G$ is Polish) 
    the support of the measure is closed, 
    there either exists $y \in G$ with $(x,y) \in \Supp \gamma$ or
    there exists a sequence $(x_{k},y_{k})$  on $\Supp \gamma$ with
    $x_{k} \to x$ as $k\to \infty$. Hence the assertion follows. \qedhere
  \end{enumerate}
\end{proof}

\begin{lemma}[convergence in product space]\label{lemma:weakCVG_productSpace}
  Let $G$ be a Polish space and suppose $(\mu_{k}),(\nu_{k})
  \subset \mathscr{P}(G)$ are tight sequences. Let $X_{k} \sim \mu
  _{k}$ and $Y_{k} \sim \nu_{k}$ and denote by $\gamma_{k} =
  \mathcal{L}( (X_{k},Y_{k}))$ the joint law of $X_{k}$ and
  $Y_{k}$. Then $(\gamma_{k})$ is tight.\\
  If furthermore, $\mu_{k} \to \mu \in \mathscr{P}(G)$ and $\nu_{k}
  \to \nu \in \mathscr{P}(G)$, then cluster points of $(\gamma_{k})$
  are in $C(\mu,\nu)$, where the set of couplings $C(\mu,\nu)$ is
  defined in \eqref{eq:couplingsDef} in Lemma \ref{lemma:couplings}.
\end{lemma}
\begin{proof}
  By tightness
  of $(\mu_{k})$ and $(\nu_{k})$, there exists for any $\epsilon> 0$ a
  compact set $K \subset G$ with $\mu_{k}(G\setminus K) < \epsilon/2$
  and $\nu_{k}(G\setminus K) < \epsilon/2$ for all $n \in \mathbb{N}$,
  so also
  \begin{align*}
    \gamma_{k}(G\times G \setminus K\times K) &\le \gamma_{k}(
    (G\setminus K)\times G) + \gamma_{k}( G \times (G\setminus K))
    \\ &= \mu_{k}(G\setminus K) + \nu_{k}(G\setminus K) \\ &< \epsilon
  \end{align*}
  for all $k \in \mathbb{N}$, implying tightness of $(\gamma_{k})$.
  By Prokhorov's Theorem, every subsequence of $(\gamma_{k})$ has
  a convergent subsequence $\gamma_{k_{j}} \to \gamma$ as
  $j\to\infty$ where $\gamma \in \mathscr{P}(G \times G)$. 
  
  It remains to show that $\gamma \in C(\mu,\nu)$.  Indeed,  since
  for every $f \in C_{b}(G \times G)$ we have $\gamma_{n_{k}} f \to
  \gamma f$, we can choose $f(x,y) = g(x) \1_{G}(y)$ with $g \in
  C_{b}(G)$.  Also,  
  \begin{align*}
    \mu g \leftarrow \mu_{n_{k}} g = \gamma_{n_{k}} f \to \gamma f =
    \gamma(\cdot\times G) g,
  \end{align*}
which  implies the equality $\mu = \gamma(\cdot\times G)$. Similarly
  $\nu = \gamma(G\times \cdot)$ and hence $\gamma \in C(\mu,\nu)$.
\end{proof}

\begin{lemma}[properties of the Prokhorov-L\`evy 
distance]\label{lemma:prokhorovDist_properties}
  Let $(G,d)$ be a separable complete metric space. 
  \begin{enumerate}[(i)]
  \item\label{item:prokLeviRep} The Prokhorov-L\`evy distance (Definition \ref{d:PL})
  has the representation
    \begin{align*}
      d_{P}(\mu,\nu) = \inf\mysetc{\epsilon>0}{ \inf_{\mathcal{L}(X,Y)
          \in C(\mu,\nu)} \mathbb{P}(d(X,Y) > \epsilon) \le \epsilon},
    \end{align*}
    where the set of couplings $C(\mu,\nu)$ is defined in
    \eqref{eq:couplingsDef} in Lemma \ref{lemma:couplings}.  Furthermore,
    the inner infimum for fixed $\epsilon>0$ is attained and the outer
    infimum is also attained.
  \item $d_{P}(\mu,\nu) \in [0,1]$.
  \item $d_{P}$ metrizes convergence in distribution, i.e.\ 
for $\mu_{k}, \mu
    \in \mathscr{P}(G)$, $k \in \mathbb{N}$ the sequence $\mu_{k}$ converges  to 
    $\mu$ in distribution if and only if $d_{P}(\mu_{k},\mu) \to 0$ 
as $k \to \infty$.
  \item $(\mathscr{P}(G), d_{P})$ is a separable complete metric space.
  \item\label{item:prokhorov5} For $\mu_{j},\nu_{j} \in \mathscr{P}(G)$ and 
  $\lambda_{j} \in [0,1]$,
    $j=1, \ldots, m$ with $\sum_{j=1}^{m}\lambda_{j} = 1$ we have 
    \begin{align*}
      d_{P}(\sum_{j}\lambda_{j}\mu_{j}, \sum_{j} \lambda_{j}\nu_{j}) \le 
\max_{j}
      d_{P}(\mu_{j}, \nu_{j}).
    \end{align*}
  \end{enumerate}
\end{lemma}
\begin{proof}
  \begin{enumerate}[(i)]
  \item See \cite[Corollary to Theorem 11]{Strassen65} for the first assertion. To
    see that the infimum is attained, let $\gamma_{k} \in C(\mu,\nu)$
    be a minimizing sequence, i.e.\ for $(X_{k},Y_{k}) \sim
    \gamma_{k}$ it holds that  $\mathbb{P}(d(X_{k},Y_{k}) > \epsilon) =
    \gamma_{k}(U_{\epsilon}) \to \inf_{(X,Y) \in C(\mu,\nu)}
    \mathbb{P}(d(X,Y) >\epsilon)$, where
    $U_{\epsilon}:=\mysetc{(x,y)}{d(x,y) > \epsilon} \subset G \times
    G$ is open. The sequence $(\gamma_{k})$ is tight and for a
    cluster point $\gamma$ we have $\gamma \in C(\mu,\nu)$ by
    Lemma \ref{lemma:weakCVG_productSpace}. From \cite[Theorem
    36.1]{Parthasarathy} it follows that $\gamma(U_{\epsilon}) \le
    \liminf_{j} \gamma_{k_{j}}(U_{\epsilon})$.\\
    To see, that the outer infimum is attained, let $(\epsilon_{k})$
    be a minimizing sequence, chosen to be monotonically nonincreasing
    with limit $\epsilon \ge 0$. One has that $U_{\epsilon} =
    \bigcup_{k} U_{\epsilon_{k}}$ where $U_{\epsilon_{k}} \supset
    U_{\epsilon_{k+1}}$ and hence $\gamma(U_{\epsilon}) = \lim_{k}
    \gamma(U_{\epsilon_{k}}) \le \lim_{k} \epsilon_{k} = \epsilon$.
  \item Clear by \eqref{item:prokLeviRep}.
  \item See \cite{Billingsley}.
  \item See \cite[Lemma 1.4]{Prokh56}.
  \item If $\epsilon>0$ is such that $\mu_{j}(A) \le
    \nu_{j}(\mathbb{B}(A,\epsilon))+\epsilon$ and $\nu_{j}(A)\le
    \mu_{j}(\mathbb{B}(A,\epsilon))+\epsilon$ for all $j =1,\ldots,m$ and all
    $A \in \mathcal{B}(G)$, then also $\sum_{j} \lambda_{j} \mu_{j}(A)
    \le \sum_{j} \lambda_{j} \nu_{j}(\mathbb{B}(A,\epsilon)) +\epsilon$ as well
    as $\sum_{j} \lambda_{j} \nu_{j}(A) \le \sum_{j} \lambda_{j}
    \mu_{j}(\mathbb{B}(A,\epsilon)) +\epsilon$.
  \end{enumerate}
\end{proof}



\begin{lemma}[properties of the Wasserstein metric]\label{lemma:WassersteinMetric_prop}
Recall  $\mathscr{P}_{p}(G)$ and $W_{p}$ from Definition \ref{d:PL}.
  \begin{enumerate}[(i)]
  \item The representation of $\mathscr{P}_{p}(G)$ is independent of
    $x$ and for $\mu,\nu \in \mathscr{P}_{p}(G)$ the 
    distance $W_{p}(\mu,\nu)$ is finite.
  \item\label{lemma:WassersteinMetric_prop ii} 
  The distance $W_{p}(\mu,\nu)$ is attained when it is finite. 
  \item\label{lemma:WassersteinMetric_prop iii} 
  The metric space $(\mathscr{P}_{p}(G),W_{p}(G))$ is complete and separable.
  \item\label{lemma:WassersteinMetric_prop iv} 
  If $W_{p}(\mu_{k},\mu) \to 0$ as $k\to\infty$ for the sequence 
  $(\mu_{k})$ on 
  $\mathscr{P}(G)$, then $\mu_{k} \to \mu$ as $k\to\infty$.
  \end{enumerate}
\end{lemma}
\begin{proof}
  \begin{enumerate}[(i)]
  \item See \cite[Remark after Definition 6.4]{Villani2008}.
  \item From Lemma \ref{lemma:weakCVG_productSpace} we know that a 
  minimizing sequence $(\gamma_{k})$ for $W_{p}(\mu,\nu)$  is tight and hence there is a
    cluster point $\gamma \in C(\mu,\nu)$. By continuity of the metric
    $d$ it follows that $d$ is lower semi-continuous and bounded from
    below and from \cite[Theorem 9.1.5]{stroock2010probability}
    it follows that $\gamma d \le \liminf_{j} \gamma_{k_{j}} d =
    W_{p}(\mu,\nu)$.
  \item See \cite[Theorem 6.9]{Villani2008}.
  \item See \cite[Theorem 6.18]{Villani2008}.
  \end{enumerate}
\end{proof}
Note that the converse to 
Lemma \ref{lemma:WassersteinMetric_prop}
\eqref{lemma:WassersteinMetric_prop iv}
does not hold.  

%

\end{document}